\documentclass[10pt,a4paper]{article}

\usepackage{amsmath}
\usepackage{amsfonts}
\usepackage{amsthm}
\usepackage{amssymb}
\usepackage{url}
\usepackage[usenames,dvips]{color}
\usepackage[dvips]{graphicx}
\usepackage{psfrag, float}  

\usepackage{verbatim}	

\newtheorem{theorem}{Theorem}

\newtheorem{ansatz}{Ansatz}
\newtheorem{lemma}{Lemma}[section]
\newtheorem{corollary}[lemma]{Corollary}
\newtheorem{definition}[lemma]{Definition}

\newtheorem{notation}[lemma]{Notation}

\newtheorem{conjecture}[lemma]{Conjecture}
\theoremstyle{remark}
\newtheorem{remark}[lemma]{Remark}

\newcommand{\be}{\begin{equation}}
  \newcommand{\ee}{\end{equation}}
\newcommand{\eps}{\varepsilon}

\newcommand{\dps}{\displaystyle}
\newcommand{\RR}{\mathbb{R}}

\newcommand{\TT}{\mathbb{T}}
\newcommand{\ZZ}{\mathbb{Z}}

\newcommand{\PP}{\mathcal{P}}
\newcommand{\GG}{\mathcal{G}}
\newcommand{\II}{\mathcal{I}}
\newcommand{\BB}{\mathcal{B}}

\newcommand{\KK}{\mathcal{K}}
\newcommand{\SSS}{\mathcal{S}}
\newcommand{\DD}{\mathcal{D}}
\newcommand{\XX}{\mathcal{X}}
\newcommand{\OO}{\mathcal{O}}
\newcommand{\FF}{\mathcal{F}}

\newcommand{\TTT}{\mathcal{T}}

\newcommand{\NNN}{\mathcal{N}}

\newcommand{\CCC}{\mathcal{C}}

\newcommand{\ii}{^{-1}}
\newcommand{\de}{\delta}
\newcommand{\pa}{\partial}
\newcommand{\la}{\lambda}

\newcommand{\inn}{\mathrm{in}}
\newcommand{\out}{\mathrm{out}}

\newcommand{\ol}{\overline}
\newcommand{\La}{\Lambda}

\renewcommand{\Re}{\mathrm{Re\, }}
\renewcommand{\Im}{\mathrm{Im\,}}
\newcommand{\wt}{\widetilde}

\newcommand{\g}{\hat g}
\newcommand{\Lb}{\Lambda}
\newcommand{\lb}{\lambda}
\newcommand{\ccirc}{\mathrm{circ}}
\newcommand{\eell}{\mathrm{ell}}

\newcommand{\ff}{\mathrm{f}}
\newcommand{\bb}{\mathrm{b}}

\newcommand{\sixmap}{\widetilde P}

\newcommand{\Jminus}{-1.81}	
\newcommand{\Jplus}{-1.56}	
\newcommand{\eminus}{0.48}	
\newcommand{\eplus}{0.67}	
\newcommand{\Gminus}{1.67}	
\newcommand{\Gplus}{1.42}	

\newcommand{\Jminusnew}{-1.6}	
\newcommand{\Jplusnew}{-1.3594}	
\newcommand{\eminusnew}{0.59}	
\newcommand{\eplusnew}{0.91}	
\newcommand{\Gminusnew}{0.56}	
\newcommand{\Gplusnew}{0.32}	

\usepackage{hyperref}
\begin{document}

\title{Kirkwood gaps and diffusion along mean motion resonances\\
  in the restricted planar three-body problem}

\author{Jacques F{\'e}joz\footnote{Universit{\'e} Paris-Dauphine and
    Observatoire de Paris (\url{fejoz@imcce.fr})},\; Marcel Gu{\`a}rdia\footnote{University of Maryland at College Park (\url{marcel.guardia@upc.edu})},
  Vadim Kaloshin\footnote{University of Maryland at College Park (\url{kaloshin@math.umd.edu})} and Pablo Rold{\'a}n\footnote{Universitat Polit{\`e}cnica de Catalunya (\url{pablo.roldan@upc.edu})}}

\maketitle

\begin{abstract}
  We study the dynamics of the restricted planar three-body problem
  near  mean motion resonances, i.e. a resonance involving the
  Keplerian periods of the two lighter bodies revolving around the
  most massive one. This problem is often used to model
  Sun--Jupiter--asteroid systems. For the primaries (Sun and Jupiter), 
  we pick a realistic mass ratio $\mu=10^{-3}$ and a small eccentricity 
  $e_0>0$. The main result is a construction of a variety of \emph{non local} diffusing orbits
  which show a drastic change of the osculating (instant) eccentricity of 
  the asteroid, while the osculating semi major axis is kept almost
  constant. The proof relies on the careful analysis of the circular
  problem, which has a hyperbolic structure, but for which diffusion
  is prevented by KAM tori. In the proof we verify certain
  non-degeneracy conditions numerically.

  Based on the work of Treschev, it is natural to conjecture 
  that  the time of diffusion for this problem is 
  $\sim \frac{-\ln (\mu e_0)}{\mu^{3/2} e_0}$.  We expect our 
  instability  mechanism to apply to
  realistic values of $e_0$ and we give heuristic arguments in its
  favor. If so, the applicability of Nekhoroshev theory to the
  three-body problem as well as the long time stability become
  questionable.

  It is well known that, in the Asteroid Belt, located between the
  orbits of Mars and Jupiter, the distribution of asteroids has the
  so-called {\it Kirkwood gaps} exactly at mean motion resonances of
  low order. Our mechanism gives a possible explanation of their
  existence. To relate the existence of Kirkwood gaps with Arnol'd
  diffusion, we also state a conjecture on its existence for a typical
  $\eps$-perturbation of the product of the pendulum and the rotator.
  Namely, we predict that a positive conditional measure of initial
  conditions concentrated in the main resonance exhibits Arnol'd
  diffusion on time scales $\frac{- \ln \eps }{\eps^{2}}$.
\end{abstract}


\tableofcontents
\nopagebreak
\section{Introduction and main results}
\label{sec:Intro}

\subsection{The problem of the stability of gravitating bodies}

The stability of the Solar System is a longstanding problem.  Over
the centuries, mathematicians and astronomers have spent an inordinate
amount of energy proving stronger and stronger stability theorems for
dynamical systems closely related to the Solar System, generally
within the frame of the Newtonian $N$-body problem:
\begin{equation}\label{eq:nBodyProblem}
  \ddot q_i = \sum_{j \neq i} m_j \frac{q_j - q_i}{\|q_j-q_i\|^3},
  \quad q_i \in \mathbf{R}^2, \quad i=0,1,...,N-1,
\end{equation}
and its planetary subproblem, where the mass $m_0$ (modelling the
Sun) is much larger than the other masses $m_i$.

A famous theorem of Lagrange entails that the observed variations in
the motion of Jupiter and Saturn come from resonant terms of large
amplitude and long period, but with zero average (see
\cite{Laskar:2006} and references therein, or
\cite[Example~6.16]{ArnoldKN88}). Yet it is a mistake, which Laplace
made, to infer the topological stability of the planetary system,
since the theorem deals only with an approximation of the first order
with respect to the masses, eccentricities and inclinations of the
planets \cite[p.~296]{Laplace:1789}. Another key result is Arnol'd's
theorem, which proves the existence of a set of positive Lebesgue
measure filled by invariant tori in planetary systems, provided that
the masses of the planets are small \cite{Arnold:1963, Fejoz:2004}.
However, in the phase space the gaps left by the invariant tori leave
room for instability.

It was a big surprise when the numerical computations of Sussman,
Wisdom and Laskar showed that over the life span of the Sun, or even over
a few million years, collisions and ejections of inner planets are probable
(due to the exponential divergence of solutions, only a
probabilistic result seems within the reach of numerical experiments);
see for example \cite{Sussman:1992, Laskar:1994}, or \cite{Laskar:2010}
for a recent account. Our Solar System, as well
as newly discovered extra-solar systems, are now widely believed to be
unstable, and the general conjecture about the $N$-body problem is
quite the opposite of what it used to be:

\begin{conjecture}[Global instability of the $N$-body problem]
  \label{conj:Herman}
  In restriction to any energy level of the $N$-body problem, the
  non-wandering set is nowhere dense. (One can reparameterize orbits
  to have a complete flow, despite collisions.)
\end{conjecture}

According to
Herman~\cite{Herman:1998}, this is the oldest open problem in
dynamical systems (see also ~\cite{Kolmogorov54b}). This conjecture
would imply that bounded orbits form a {\it nowhere dense set} and
that no topological stability whatsoever holds, in a very strong
sense.  It is largely confirmed by numerical experiments. In our Solar
System, Laskar for instance has shown that collisions between Mars and
Venus could occur within a few billion years. The coexistence of a
nowhere dense set of positive measure of bounded quasi-periodic motions
with an open and dense set of initial conditions with unbounded orbits
is a remarkable conjecture.

Currently the above conjecture is largely out of reach. A more modest
but still very challenging goal, also stated in \cite{Herman:1998}, is a local version of the conjecture:

\begin{conjecture}[Instability of the planetary
  problem]\label{conj:local}
  If the masses of the planets are small enough, the wandering set
  accumulates on the set of circular, coplanar, Keplerian motions.
\end{conjecture}

There have been some prior attempts to prove such a conjecture. For
instance, Moeckel discovered an instability mechanism in a special
configuration of the $5$-body problem~\cite{Moeckel:1995}. His proof
of diffusion was limited by the so-called big gaps problem between
hyperbolic invariant tori; this problem was later solved in this
setting by Zheng~\cite{Zheng:2010}. A somewhat opposite strategy was
developed by Bolotin and McKay, using the Poincar{\'e} orbits of the
second species to show the existence of symbolic dynamics in the
three-body problem, hence of chaotic orbits, but considering far from
integrable, non-planetary conditions; see for
example~\cite{Bolotin:2006}. Also, Delshams, Gidea and Rold{\'a}n have shown an
instability mechanism in the spatial restricted three-body problem, but only
locally around the equilibrium point $L_1$ (see \cite{DelshamsGR11}).

In this paper we prove the existence of large instabilities in a
realistic planetary system and describe the associated instability
mechanism. We thus provide a step towards the proof of
Conjecture~\ref{conj:local}.

In his famous paper \cite{Arnold64}, Arnol'd says: ``In contradistinction with
stability, instability\footnote{In the translation the word ``nonstability''
is used, which seems to refer to instability.} is itself stable. I believe that 
the mechanism of ``transition chain'' which guarantees that instability in 
our example is also applicable to the general case (for example, to the problem 
of three bodies)''. In this paper we exhibit a regime of realistic
motions of a three body problem where ``transition chains'' {\bf do} occur
and lead to Arnol'd's mechanism of instability. Such instabilities occur near
mean motion resonances, defined below. To the best of our knowledge, this 
is the first regime of motions of the problem of three bodies naturally 
modelling a region in the Solar system, where nonlocal transition chains are 
established\footnote{``Nonlocal'' means that motions on the boundary tori in this chain
  differ significantly, uniformly with respect to the small
  parameter. In our case, the eccentricity of orbits of the massless
planet (asteroid) varies by $\OO(1)$, uniformly with
  respect to small values of the eccentricity of the primaries, while
  the semi major axis stays nearly constant.
See Section \ref{main-result} for more details.}. Previous results showing transition chains of tori in the problem of three bodies naturally 
modelling a region in the Solar system are confined to small neighborhoods of the Lagrangian Equilibrium points  \cite{CapinskiZ11, DelshamsGR11}, and therefore, are local in the Configuration and Phase space.

The instability mechanism shown in this paper is related to a generalized 
version of Mather's acceleration problem~
\cite{Mather96, BolotinT99, DelshamsLS00, Gelfreich:2008,
 Kaloshin03, Piftankin06}.  Some parts of the proof rely on
numerical computations, but our strategy allows us to keep these
computations simple and convincing.

We consider the planetary problem
(\ref{eq:nBodyProblem}) with one planet mass (say, $m_1$) larger
than the others: $m_0 \gg m_1 \gg m_2,...,m_{N-1}$. The
equations of motion of the lighter objects ($i=2,...,N-1$) can
advantageously be written as
\begin{equation}\label{eq:3BodyProblem}
  \ddot q_i = m_0 \frac{q_0 - q_i}{\|q_0-q_i\|^3}+
  m_1 \frac{q_1 - q_i}{\|q_1-q_i\|^3}+\sum_{j \neq i, j>1}
  m_j \frac{q_j - q_i}{\|q_j-q_i\|^3}.
\end{equation}
Letting the masses $m_j$ tend to $0$ for $j=2,...,N-1$, we obtain a
collection of $(N-2)$ independent {\it restricted problems}:
\begin{equation}\label{eq:R3BodyProblem}
  \ddot q_i = m_0 \frac{q_0 - q_i}{\|q_0-q_i\|^3}+
  m_1 \frac{q_1 - q_i}{\|q_1-q_i\|^3},
\end{equation}
where the massless bodies are influenced by, without themselves
influencing, the \emph{primaries} of masses $m_0$ and $m_1$.

For $N=3$, this model is often used to approximate the dynamics of
Sun-Jupiter-asteroid or other Sun-planet-object problems, and it is the
simplest one conjectured to have a wide range of instabilities.

\subsection{An example of relevance in astronomy}

\subsubsection{The asteroid belt}

One place in the Solar system where the dynamics is well approximated by
the restricted three-body problem is the asteroid belt. The asteroid 
belt is located between the orbits of Mars and Jupiter and consists of
1.7 million objects ranging from asteroids of $950$ kilometers to dust
particles. Since the mass of Jupiter is approximately $2960$ masses of
Mars, away from close encounters with Mars, one can neglect the
influence of Mars on the asteroids and focus on the influence of
Jupiter.  We also omit interactions with the second biggest planet in the
Solar System, namely Saturn, which actually is not so small. Indeed,
its mass is about a third of the mass of Jupiter and its semi major
axis is about $1.83$ times the semi major axis of Jupiter. This
implies that the strength of interaction with Saturn is around $10\%$
of the strength of interaction with Jupiter. However, instabilities
discussed in this paper are fairly robust and we believe that they are
not destroyed by the interaction with Saturn (or other celestial
bodies), which to some degree averages out.

With these assumptions one can model the motion of the objects in the
asteroid belt by the restricted problem.  Denote by
$\mu=m_1/(m_0+m_1)$ the mass ratio, where $m_0$ is the mass of the Sun
and $m_1$ is the mass of Jupiter.  For $\mu=0$ (namely, neglecting the
influence of Jupiter), bounded orbits of the asteroid are ellipses. Up
to orientation, the ellipses are characterized by their semi major axis
$a$ and eccentricity $e$.

The aforementioned theorem of Lagrange asserts that, for small $\mu>0$,
the semi major axis $a(t)$ of an asteroid satisfies
$|a(t)-a(0)|\lesssim \mu$ for all $|t|\lesssim 1/\mu$.  For very small
$\mu$ the time of stability was greatly improved by Niederman
\cite{Niederman96} using Nekhoroshev theory; see the discussion in the
next section. Nevertheless, if one looks at the asteroid distribution
in terms of their semi major axis, one encounters several gaps, the
so-called \emph{Kirkwood gaps}. It is believed that the existence of
these gaps is due to instability mechanisms.

\subsubsection{Kirkwood gaps and Wisdom's ejection mechanism}
\label{sec:ejection}

{\it Mean motion resonances} occur when the  ratio between 
the period of Jupiter and the period of the asteroid is rational. 
In particular, the Kirkwood gaps correspond to  the ratios $3:1,\ 5:2$ and $7:3$.
\begin{figure}[h]
  \begin{center}
  \includegraphics[width=7.5cm]{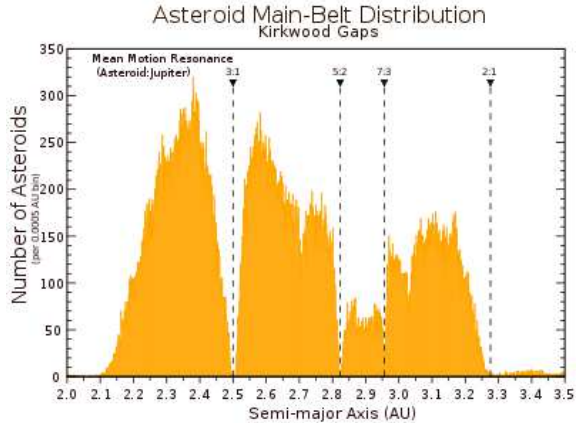}
  \end{center}
  \caption{Kirkwood gaps}
  \label{fig:Kirkwood}
\end{figure}

In this section we present a heuristic explanation of the reason why these gaps exist.

It is conjectured and confirmed by numerical data \cite{Wisdom82}
that the eccentricities of asteroids appropriately placed in the
Kirkwood gaps change by a magnitude of order one. Notice that in
the real data, the eccentricities of most asteroids in the asteroid belt
are between $0$ and $0.25$; see for example
\url{http://en.wikipedia.org/wiki/File:Main belt e vs a.png}.

As the eccentricity of the asteroid grows while its semi major axis is
nearly constant, its perihelion gets closer and closer to the origin,
namely at the distance $a(t)[1-e(t)]$, where $a(t)$ and $e(t)$ are the
semi major axis and eccentricity of the asteroid respectively (see
Figure \ref{fig:EllipseChange}, where the inner circle is the orbit of
Mars).  In particular, a close encounter with Mars becomes
increasingly probable.  Eventually Mars and the asteroid
come close to each other, and the asteroid most probably gets ejected
from the asteroid belt.

A surprising fact is that {\it the change of eccentricity of 
the asteroid is only possible due to the ellipticity of the motion 
of Jupiter}, due to the following count of dimensions. For circular
motions of Jupiter the problem reduces to two degrees of freedom (see
Section \ref{sec:SettingsAndResults}) and plausibly there are
invariant $2$-dimensional tori separating the three dimensional energy
surfaces; see for example \cite{Giorgilli:1989, Fejoz02b,  CellettiC07}.  If the
eccentricity of Jupiter is not zero, the system has two and a half
degrees of freedom and then KAM tori do not prevent drastic changes in
the eccentricity.

Heuristically, the conclusion is that, if the eccentricity of 
the asteroid changes by a magnitude of order one in the
Sun-Jupiter-asteroid restricted problem, then the asteroid might 
come into zones where the restricted problem does not describe 
the dynamics appropriately, due to the influence of Mars.

\medskip The main result of this paper is that for certain mean motion
resonances there are unstable motions which lead to significant changes
in the eccentricity.  We only present results for two
particular resonances ($1:7$ and $3:1$), because 
the proof relies on numerical computations. The resonance $3:1$ 
corresponds to one of most noticeable Kirkwood gaps. We are
confident that our mechanism of instability applies to other
resonances, and thus to the other Kirkwood gaps, as long as 
the orbits of the unperturbed problem stay away from collisions. 
Thus, the instability mechanism showed in this paper
gives insight into the existence of the Kirkwood gaps. 

Another instability mechanism, using the adiabatic invariant theory,
can be seen in \cite{NeishtadtS04} where a heuristic explanation is
given.  Let $\eps_J = \mu_J^{1/2}/e_J,$ where $\mu_J$ is mass ratio
and $e_J$ is eccentricity of Jupiter. They study the case when
$\eps_J$ is relatively small: $0.025,\ 0.05,\ 0.1,\ 0.2$. In reality
it is close to $0.6$. In contrast, we study the case of large
$\eps_J$. 

\subsubsection{Capture in resonance of other objects}\label{sec:capture}

Many known light objects in the Solar System display a
mean motion resonance of low order with Jupiter or some other planet.
Some of them are: Trojan satellites, which librate around one of the
two Lagrangian points of a planet, hence in $1:1$ 
resonance with the planet; Uranus, which is close to the $1:7$ 
resonance with Jupiter, thus giving an example of an ``outer''
restricted problem that is close in phase space to the solutions we are
studying; or the Kuiper Belt beyond Neptune, whose objects, behaving
in the exact opposite manner to those of the asteroid belt, seem
to \emph{concentrate} close to mean motion resonances (in particular,
the Keplerian ellipse of the dwarf planet Pluto notoriously meets the
ellipse of Neptune).  The current existence of these resonant objects,
and thus their relative stability, seemingly contradicts the above
mechanism. This calls at least for a short explanation, although there
are many effects at work.

The main point is that an elliptic
stability zone lies in the eye of a resonance, where some kind of long term
stability prevails. Besides, 
the geometry of the system often prevents the ejection mechanism described in
Section~\ref{sec:ejection}, because there is no such body as Mars
to propel the asteroid through a close encounter. In many cases, the
mean motion resonance itself precludes collisions with the main planet, 
for example the Trojan asteroids with respect to Jupiter, or 
Pluto with respect to Neptune; for a discussion of this effect in the
asteroid belt, see~\cite{Robutel:2005}. 

The
complete picture certainly includes secular resonances, close
encounters between asteroids, as well as more complicated
kinds of resonance involving more bodies (for example the second
Kirkwood gap, where a four-body problem resonance seems to play a
crucial role). We refer to \cite{Morbidelli:2002, Robutel:2005} for
further astronomical details.

\subsection{Main results}\label{main-result}

Let us consider the three-body problem and assume that
the massless body moves in the same plane as the two primaries.
We normalize the total mass to one, and we call the three bodies the Sun
(mass $1-\mu$), Jupiter (mass $\mu$ with $0<\mu \ll 1$) and the
asteroid (zero mass). If the energy of the primaries is negative, 
their orbits describe two ellipses with the same eccentricity, say
$e_0\geq0$.
For convenience, we 
denote by $q_0(t)$ the normalized position of the primaries
(or ``fictitious body''), so that the Sun and Jupiter have respective
positions $-\mu q_0(t)$ and $(1-\mu)q_0(t)$. The Hamiltonian of the
asteroid is
\begin{equation}\label{def:Ham:Cartesian}
  K(q,p,t)= \frac{\|p\|^2}{2} - \frac{1-\mu}{\left\|q+\mu
      q_0(t)\right\|}-\frac{\mu}{\left\|q-(1-\mu)q_0(t)\right\|}
\end{equation}
where $q,p\in\RR^2$. Without loss of generality one can assume that
$q_0(t)$ has semi major axis equal to 1 and period $2\pi$. For
$e_0\geq 0$ this system has two and a half degrees of freedom.

When $e_0=0$, the primaries describe uniform circular motions aroung
their center of mass. (This system is called the restricted planar circular
three-body problem).
Thus, in a frame rotating with the primaries, the system becomes
autonomous and hence has only 2 degrees of freedom. Its energy in the
rotating frame is a first integral, called {\it the Jacobi
integral}\footnote{Celestial mechanics's works often prefer to use the
{\it Jacobi constant} $C$, given by $J=\frac{(1-\mu)\mu-C}{2}$.}.
It is defined by
\begin{equation}\label{def:Jacobi}
  J= \frac{\|p\|^2}{2}- \frac{1-\mu}{\left\|q+\mu
      q_0(t)\right\|} -
  \frac{\mu}{\left\|q-(1-\mu)q_0(t)\right\|}-(q_1p_2-q_2p_1).
\end{equation}

The aforementioned KAM theory applies to both the circular and the
elliptic problem \cite{Arnold:1963, SiegelM95} and asserts that if
the mass of Jupiter is small enough, there is a set of initial
conditions of positive Lebesgue measure leading to quasiperiodic
motions, in the neighborhood of circular motions of the asteroid.

If Jupiter has a circular motion, since the system has only 2 degrees
of freedom, KAM invariant tori are 2-dimensional and separate the
3-dimensional energy surfaces.  But in the elliptic problem,
3-dimensional KAM tori do not prevent orbits from wandering on a
5-dimensional phase space. In this paper we prove the existence of a
wide enough set of wandering orbits in the elliptic planar restricted
three-body problem.

Let us write the Hamiltonian \eqref{def:Ham:Cartesian} as
\[
K(q,p,t)=K_0(q,p)+K_1(q,p,t,\mu),
\]
with
\[
\begin{split}
  K_0(q,p)&=\frac{\|p\|^2}{2}-\frac{1}{\|q\|},\\
  K_1(q,p,t,\mu)&=\frac{1}{\|q\|}-\frac{1-\mu}{\left\|q+\mu q_0(t)\right\|}
  -\frac{\mu}{\left\|q-(1-\mu)q_0(t)\right\|}.
\end{split}
\]
The Keplerian part $K_0$ allows us to associate elliptical elements to
every point $(q,p)$ of the phase space of negative energy $K_0$. We
are interested in the drift of the eccentricity $e$ under the flow of
$K$. (The reader will easily distinguish this notation from
other meanings of $e$).

We will see later that $K_1=\OO(\mu)$ uniformly, away from collisions. 
Notice that there is a competition between the integrability of $K_0$ and
the non-integrability of $K_1$, which allows for wandering.
In this work we consider a realistic value of the
mass ratio, $\mu=10^{-3}$.

\begin{notation}
In what follows, we abbreviate {\it the restricted planar
circular three-body problem
to the circular problem, and the restricted planar elliptic three-body
problem to the elliptic problem.}
\end{notation}

Here is the main result of this paper.\\

\noindent\textbf{Main Result (resonance $1:7$). } {\em 
Consider
  the elliptic problem with mass ratio $\mu=10^{-3}$ and 
  eccentricity of Jupiter $e_0>0$. 
Assume it is in general position\footnote{Later we state three Ans{\"a}tze
that formalize the non-degeneracy conditions we need.}.
  Then, for $e_0$ small enough, there
  exists a time $T>0$ and a trajectory whose eccentricity $e(t)$ satisfies
  that
  \[
  e(0)< \eminus \quad\text{ and }\quad e(T)>\eplus
  \]
  while
  \[\left|a(t)-7^{2/3}\right|\leq 0.027\quad\text{ for }t\in [0,T].\]
}

\begin{figure}[h]
  \begin{center}
    \includegraphics[width=7cm]{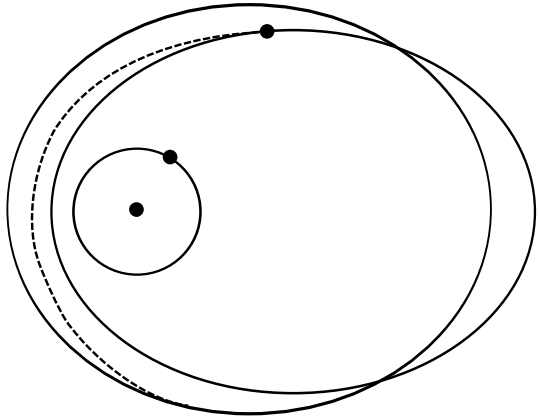}
  \end{center}
  \caption{Transition from the instant ellipse of eccentricty $e=0.48$
    to the instant ellipse of eccentricty $e=0.67$. The dashed line
    represents the transition; however, the actual
    diffusing orbit is very complicated and the diffusion is very
    slow.}
  \label{fig:EllipseChange}
\end{figure}

We will make this result more precise in
Section~\ref{sec:SettingsAndResults}, Theorem~\ref{th:MainTheorem:detail},
after providing some appropriate definitions.
We stress that the instabilities discussed in the Main Result are non-local neither in the action space nor in the configuration space. This is the first result showing nonlocal instabilities in the planetary three body problem.

In \cite{GalanteK10a,GalanteK10b, GalanteK11} it is shown that in 
the circular problem with realistic mass ratio 
$\mu=10^{-3}$ there exists an unbounded Birkhoff region of instability for 
eccentricies larger than $0.66$ and Jacobi integral $J=1.8$. This allows 
them to prove a variety of unstable motions, including oscillatory motions 
and all types of final motions of Chazy.


The analogous result for the $3:1$ resonance is as follows.

\medskip\noindent\textbf{Main Result (resonance $3:1$). } {\em 
  Consider the elliptic problem with mass ratio $\mu=10^{-3}$ and 
  eccentricity of Jupiter $e_0>0$. 
Assume it is in general position.
Then, for $e_0$ small enough, there
  exists a time $T>0$ and a trajectory whose eccentricity $e(t)$ satisfies
  that
  \[
  e(0)< \eminusnew\quad\text{ and }\quad e(T)>\eplusnew
  \]
  while
  \[\left|a(t)-3^{-2/3}\right|\leq 0.149\quad\text{ for }t\in [0,T].\]
}

Thus we claim the existence of orbits of the asteroid whose change in
eccentricity is above $0.3$.  In Appendix~\ref{speed}, we state two
conjectures about the stochastic behavior of orbits near a resonance:
one is for Arnol'd's example and another one is for our elliptic
problem. These conjectures are based on numerical experiments; see
for example~\cite{Chirikov:1979,Sagdeev:1988,Wisdom82}. 
We also provide some heuristic arguments using the
dynamical structures explored in this paper. Loosely speaking, we
claim that near a resonance there is polynomial instability for 
a positive measure set of initial conditions on the time scale 
$\frac{-\ln(\mu e_0)}{\mu^{3/2}e_0}$. 

Most of the paper is devoted to the resonance $1:7$. But the proof
seems robust with respect to the precise resonance
considered. In appendix~\ref{App:Resonance1:3}, we show how to modify the proof of
the main result to deal with the resonance $3:1$, whose importance in
explaining the
Kirkwood gaps is emphasized in the introduction. 

We believe that our mechanism applies to a substantially 
larger interval of eccentricities, but proving this requires
more sophisticated numerics; see Remark~\ref{rem:maxrange}.

\subsection{Refinements and comments}

\subsubsection{Smallness of the eccentricity of Jupiter}
When Jupiter describes a circular motion, the Jacobi integral is an
integral of motion and then KAM theory prevents global instabilities.
We consider the eccentricity $e_0$ as a small parameter so that we can
compare the dynamics of the elliptic problem with the dynamics of the
circular one.

The difference between the elliptic and circular
Hamiltonians is $\OO(\mu e_0)$. The analysis of the difference,
performed in Section
\ref{sec:Expansion:Hamiltonian}, shows that this difference can be reduced to
$\OO(\mu e_0^5)$ (or even smaller) using averaging.
This makes us believe that $e_0$ does not need to be infinitesimally small for
our mechanism to work. Even the realistic value $e_0 \thickapprox 0.048$
is not out of question. However, having a realistic $e_0$
becomes mostly a matter of {\it numerical experiment, not of mathematical
proof} ---the limit and the interest of perturbation theory is to
describe dynamical behavior in terms of asymptotic models. See
Appendix~\ref{diffusion-structure} for more details.

\subsubsection{On infinitesimally small masses $\mu$}
In the Main Result, we do not know what happens asymptotically if
we let $\mu\rightarrow 0$, since our estimates worsen.  Indeed, one of the
crucial steps of the proof is to study the transversality of certain
invariant manifolds (see Section \ref{sec:MainStepsProof}) and this
transversality becomes exponentially small with respect to $\mu$ as
$\mu\rightarrow 0$.  On the other hand, the Main Result holds for
realistic values of $\mu$, which is out of reach of many qualitative
results of perturbation theory where parameters are conveniently
assumed to be as small as needed. See Appendix \ref{speed} for more
details.

\subsubsection{Speed of diffusion}\label{sec:speed}
In Appendix \ref{speed} we discuss the relation of our problem
with a priori unstable systems and Mather's accelerating problem.
We conjecture that, for the orbits constructed in
this paper, the diffusion time $T$ can be chosen to be
\begin{equation}\label{instability-time}
T \sim -\dfrac{ \ln (\mu e_0)}{\mu^{3/2} e_0}.
\end{equation}
Time estimates in the a priori unstable setting can be found in \cite{BertiB02, BertiBB03, Treschev04,GideaL05}.

De~la~Llave~\cite{delaLlave04},
Gelfreich-Turaev \cite{Gelfreich:2008}, and
Piftankin \cite{Piftankin06}, using Treschev's techniques
of separatrix maps (see for instance \cite{PiftankinT07}),
proved linear diffusion for Mather's acceleration problem.
Using these techniques, a smart choice of diffusing orbits might lead to
even faster diffusion in our problem, in times of the order 
$T \sim -\ln \mu (\mu^{3/2} e_0)^{-1}$;
see Appendix \ref{speed} for more details\footnote{This does not seem
crucial, since the real value $e_0$ is not smaller than $\mu$.}.

An analytic proof of this conjecture might require restrictive
conditions between $\mu$ and $e_0$. However, for realistic values of
$\mu$ and $e_0$ or smaller, that is $0<\mu\le 10^{-3}$ and $0<e_0<0.048$,
we expect that the speed of our mechanism of diffusion also
obeys the above heuristic formula.

On the other hand, the above formula
probably does not hold in the neighborhood of circular motions of 
the masless body, which might be much more stable than more eccentric 
motions. This could explain the fact that Uranus, 
whose eccentricity of 0.04 is significantly smaller than most asteriods 
from the asteriod belt, and which is roughly in $1:7$-resonance with 
Jupiter (its period is 7.11 times larger than that of Jupiter) has not 
been expelled yet; see also Section \ref{sec:capture}.  However, 
a deeper analysis would require to compare the distances of the various 
celestial bodies to the mean motion resonance, as well as the splitting 
of their invariant manifolds.




\subsubsection{On Nekhoroshev's stability}

Consider an analytic nearly integrable system of the form
$H_\eps(\theta,I)=H_0(I)+\eps H_1(\theta,I)$ with $\theta\in \mathbb
T^n$ and $I$ in the unit ball $B^n$. Suppose $H_0$ is {\it convex} (or
even suppose the weaker condition that $H_0$ is {\it
  steep}).\footnote{Recall that $H_0$ is called {\it steep} if for any
  affine subspace $L$ of $\mathbb R^n$ the restriction $H_0|_{L}$ has
  only isolated critical points.} Then a famous result of Nekhoroshev
states
that for some $c>0$ independent of $\eps$ we have
\[
|I(t)-I(0)|\lesssim \eps^{1/2n} \qquad \text{ for }\qquad
|t| \lesssim \exp (c \, \eps^{-1/2n}).
\]
See for instance \cite{Niederman96} for the history and precise references
and \cite{Xue10} for the estimate on the involved constant $c$.

Niederman \cite{Niederman96} applied Nekhoroshev theory to
the planetary $N$-body problem. He showed that {\it the semi major
axis} obeys the above estimate for exponentially long time,
$\exp (c \, \eps^{-1/2n})$, with $\eps$ being the smallness
of the planetary masses. However, the constant $c$ along with
other constants involved in the proof are not optimal.
Specifically, $\eps$ needs to be as small as $3\cdot 10^{-24}$ 
to have stability time comparable
to the age of the Solar system.
Moreover, the stability of semi major axis
does not imply the stability of the eccentricity, which we conjecture
has substantial deviations in polynomially long time.

Notice that our results along the predictions of Treschev's 
(see Appendix \ref{speed}) state the possibility of polynomial 
instability for eccentricities for the elliptic problem. 

With $\eps \sim \mu$, there was a hope to apply this result
to the long time stability of e.g. the Sun-Jupiter-Saturn system; see
\cite{GalganiG85}). However, (\ref{instability-time}) indicates
absence of even $\OO(\eps^{-2})$-stability.  Indeed, the unperturbed
Hamiltonian of the three body problem is neither convex, nor steep.
This turns out to be not just a technical problem but a true
obstruction to exponentially long time stability, since Nekhoroshev's
theory does not apply to this kind of systems. See Appendix
\ref{speed} for more details.


\subsection{Mechanism of instability}

The Main Result gives an example of large instability for this
mechanical system. It can be interpreted as an example of Arnol'd
diffusion; see \cite{Arnold64}.  Nevertheless, Arnol'd diffusion
usually refers to nearly integrable systems, whereas Hamiltonian
\eqref{def:Ham:Cartesian} cannot be considered so close to integrable
since $\mu=10^{-3}$. The mechanism of diffusion used in this
paper is similar to the so-called Mather's
accelerating problem (\cite{Mather96,BolotinT99,DelshamsLS00,
Gelfreich:2008,Kaloshin03,Piftankin06}).  This analogy
is explained in Section \ref{sec:Circular:Outer}.

Arguably, the main source of instabilities are
\emph{resonances}.  One of the most natural kind of resonances in the 
three-body problem is \emph{mean motion orbital
  resonances}\footnote{The mean motions are the frequencies of the
  Keplerian revolution of Jupiter and the asteroid around the Sun: in
  our case the asteroid makes one full revolution while Jupiter makes
  seven revolutions.}.  Along such a resonance, Jupiter and the
asteroid will regularly be in the same relative position. Over a long
time interval, Jupiter's perturbative effect could thus pile up and
(despite its small amplitude due to the small mass of Jupiter) could
modify the eccentricity of the asteroid, instead of averaging out.

According to Kepler's third Law, this resonance takes place when
$a^{3/2}$ is close to a rational, where $a$ is the semi major axis of the
instant ellipse of the asteroid.  In our case we consider
\emph{$a^{3/2}$ close to 7} in Section \ref{sec:SettingsAndResults} and  
\emph{$a^{3/2}$ close to $1/3$} in appendix \ref{App:Resonance1:3}.
Nevertheless, we expect that the same mechanism takes
place for a large number of mean motion orbital resonances.

The semi major axis $a$ and the eccentricity $e$ describe completely
an instant ellipse of the asteroid (up to orientation).  Thus,
geometrically the Main Results say that the asteroid evolves from a
Keplerian ellipse of eccentricity $e=0.48$ to one of eccentricity
$e=0.67$ (for the resonance $1:7$) and from $e=0.59$ to $e=0.91$ (for the
resonance $3:1$), while keeping its semi major axis almost constant; see Figure
\ref{fig:EllipseChange}.  In Figure \ref{fig:a-e:diffusion} we
consider the plane $(a,e)$, which describes the ellipse of the
asteroid.  The diffusing orbits given by the Main Results correspond
to nearly horizontal lines.

\begin{figure}[h]
  \begin{center}
    \psfrag{a}{$a$}\psfrag{e}{$e$}\psfrag{1}{$1$}\psfrag{0}{$0$}
    \includegraphics[width=7cm]{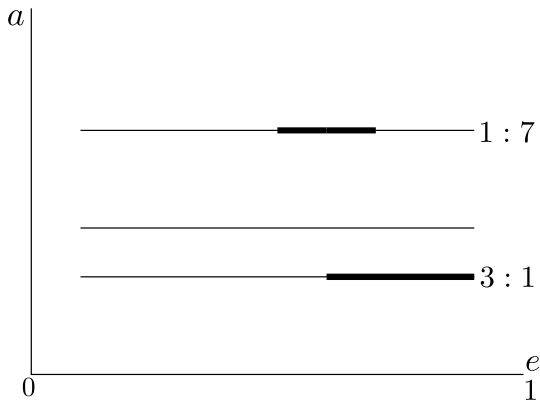}
  \end{center}
  \caption{The diffusion path that we study in
    the $(a,e)$ plane.  The horizontal lines represent the resonances
    along which we drift. The thick segments are the diffusion paths whose
    existence we prove in this paper.}
  \label{fig:a-e:diffusion}
\end{figure}

A qualitative description of such a diffusing orbit is given at the
end of section~\ref{sec:ProofDiffusion}.

\subsection{Sketch of the proof}\label{sec:MainStepsProof}

Our overall strategy is to:

\begin{itemize}
\item[(A)] Carefully study the structure of the restricted three-body
  problem along a chosen resonance.

\item[(B)] Show that, generically within the class of problems sharing
  the same structure, global instabilities exist. One could say that
  this step is similar, in spirit, to ``abstract'' proofs of existence
  of instabilities for generic perturbations of a priori chaotic
  systems such as in Mather's accelerating problem.

\item[(C)] Check numerically that the generic conditions (which we
  call Anz{\"a}tze) are satisfied in our case.
\end{itemize}

Step (B) is the core of the paper and we now give more details about
it. 

For the elliptic problem, 
the diffusing orbit that we are looking for lies
in a neighborhood of a ($3$-dimensional) normally hyperbolic invariant
cylinder $\Lb$ and its local invariant manifolds, which exist near
our mean motion resonance. The vertical component of the cylinder can be
parameterized by the eccentricity of the asteroid and the horizontal
components by its mean longitude and time.

If the stable and unstable invariant manifolds of $\Lambda$ intersect
transversally, the elliptic problem induces two different dynamics
on the cylinder (see Sections \ref{sec:CylinderExpansion}
and \ref{sec:Ell:Outer}): {\it the inner and the outer dynamics}.  The
inner dynamics is simply the restriction of the Newtonian flow to
$\Lambda$.  The outer dynamics is obtained by a limiting process: it
is observed asymptotically by starting very close to the
cylinder and its unstable manifold, traveling all the way to
a homoclinic intersection, and coming back close to the cylinder 
along its stable manifold; see Definition~\ref{definition:OuterMap}.

Since the system has different homoclinic orbits to the cylinder, one
can define different outer dynamics.  In our diffusing mechanism we
use two different outer maps. The reason is that each of the outer
maps fails to be defined in the whole cylinder, and so we need to
combine the two of them to achieve diffusion; see Section
\ref{Section:Circular}.

The proof consists in the following five 
steps:
\begin{enumerate}
\item Construct a smooth family of hyperbolic periodic orbits for the
  circular problem with varying Jacobi integral (Ansatz
  \ref{ans:NHIMCircular}).
\item Prove the existence of the normally hyperbolic invariant
  cylinder $\Lambda$, whose vertical size is lower bounded uniformly
  with respect to small values of $e_0$ (Corollary
  \ref{coro:NHIMCircular} and Theorem \ref{th:Elliptic:NHIM}).
\item Establish the transversality of the stable and unstable
  invariant manifolds of this cylinder (Ansatz~\ref{ans:NHIMCircular}
  and Theorem \ref{th:Elliptic:NHIM}), a key feature to define a
  limiting ``outer dynamics'', in addition to the inner dynamics, over
  $\Lambda$ (section~\ref{sec:Circular:Outer}).
\item Compare the inner and outer dynamics on $\Lambda$ and, in
  particular, check that they do not share any common invariant
  circles (Theorem \ref{th:InnerAndOuter:Elliptic} and
  \ref{th:Transition}). Then one can drift along $\Lambda$ by
  alternating the inner and outer maps in a carefully chosen
  order~\cite{Moeckel:2002}.
\item Construct diffusing orbits by shadowing such a polyorbit (Lemma
  \ref{lemma:shadowing}).
\end{enumerate}
This program faces difficulties at each step, as explained next.

\subsubsection{Existence of a family of hyperbolic periodic 
orbits of the circular problem}

This part is mainly numerical. Using averaging and the symmetry 
of the problem we guess a location of periodic orbits of 
a certain properly chosen Poincare map of the circular problem. 
Then for an interval of Jacobi integral $[J_-,J_+]$ and each 
$J\in [J_-,J_+]$ we compute them numerically and verify that 
they are hyperbolic. For infinitesimally small $\mu$ 
hyperbolicity follows from averaging.

\subsubsection{Existence of a normally hyperbolic invariant cylinder
  $\Lambda$}

The first difficulty comes from the proper degeneracy of the Newtonian
potential: at the limit $\mu=0$ (no Jupiter), the asteroid has a
one-frequency, Keplerian motion, whereas symplectic geometry 
allows for a three-frequency motion (as with any potential other than
the Newtonian potential $1/r$ and the elastic potential $r^2$). Due to
this degeneracy, switching to $\mu>0$ (even with $e_0=0$) is a
singular perturbation.

\subsubsection{Transversality of the stable and unstable invariant manifolds}

Establishing the transversality of the invariant
manifolds of $\Lambda$, is a delicate problem, even for
$e_0=0$. Asymptotically when
$\mu\to 0$, the difference (splitting angle) between the invariant manifolds
becomes exponentially small with respect to $\mu$, that is of order
$\exp(-c/\sqrt{\mu})$ for some constant $c>0$. Despite inordinate
efforts of specialists, all known techniques fail to estimate this
splitting, because the relevant Poincar{\'e}-Melnikov integral is not
algebraic.
Note that this step is significantly simpler when one 
studies generic systems.


At the expense of creating other difficulties, setting $\mu=10^{-3}$
avoids this splitting problem, since for this value of the parameter
we see that the splitting of separatrices is not extremely small
and can be detected by means of a computer. Besides,
$10^{-3}$ is a realistic value of the mass ratio for the
Sun-Jupiter model. Since the splitting of the separatrices varies
smoothly with respect to the eccentricity $e_0$ of the primaries, it
suffices to estimate the splitting for $e_0=0$, that is in
the circular problem. \emph{This is a key point for the numerical
  computation}, which thus remains relatively simple. On the other
hand, in the next two steps it is crucial to have $e_0>0$,
otherwise the KAM tori separates the Jacobi integral energy levels.

Moreover, recall that the cylinder $\Lambda$ has two branches
of both stable and unstable invariant manifolds 
(both originated by a family of periodic orbits of the circular 
problem, see Figures \ref{fig:invmfld2_H174}, \ref{fig:orbitpdel_H174}
for $1:7$ and Figures \ref{fig:invmfld2:2}, \ref{fig:resonancedel:2}
for $3:1$). In certain regions, the intersection between one of 
the branches of the stable and unstable  invariant manifolds is 
tangential, which prevents us from defining the outer map. Nevertheless,
then we check that the other two branches intersect
transversally and we define a different outer map.
Thus, we combine the two outer maps depending on which
branches of the invariant manifolds intersect transversally.

\subsubsection{Asymptotic formulas for the outer and inner maps}

Using classical perturbation theory
and the specific properties of the underlying system, we reduce
the inner and (the two different) outer dynamics to three $2$-dimensional
symplectic smooth maps of the form
\begin{equation}\label{def:InnerMap:ell:intro}
  \FF_{e_0}^\inn:\left(\begin{array}{c} I\\
      t
    \end{array}\right)\mapsto \left(\begin{array}{l} I+ e_0
      \left(A^+(I,\mu)e^{it}+ A^-(I,\mu)e^{-it}\right)+\OO\left(\mu
        e_0^2\right)\\
      t+\mu\TTT_0(I,\mu)+\OO(\mu e_0)
    \end{array}\right)
\end{equation}
and
\begin{equation}\label{def:OuterMap:Elliptic:Intro}
  \FF_{e_0}^{\out,\ast}:\left(\begin{array}{c} I\\
      t
    \end{array}\right)\mapsto \left(\begin{array}{l} I+ e_0\left(B^{\ast,+}
        (I,\mu)e^{it}+B^{\ast,-} (I,\mu)e^{-it}\right) +\OO\left(\mu
        e_0^2\right)\\
      t+\mu\omega^\ast(I,\mu)+\OO(\mu e_0)
    \end{array}\right),\,\,\,\ast=\ff,\bb,
\end{equation}
where $(I,t)$ are conjugate variables which parameterize a 
connected component of the 3-dimensional normally hyperbolic invariant
cylinder $\Lambda$ intersected with a transversal Poincar{\'e}
section, and $A^\pm, \TTT_0,B^{\ast,\pm}, \omega^\ast$ are smooth
functions. The superindexes $\ff$ and $\bb$ stand for the forward and
backward heteroclinic orbits that are used to define the outer maps.
The choice of this notation will be clear in Section
\ref{Section:Circular}.
Note that these maps are real and thus
$A^-$ and $B^{\ast,-}$ are complex conjugate to $A^+$ and $B^{\ast,+}$
respectively.

\subsubsection{Non-degeneracy implies the existence of diffusing orbits}

As shown in Section \ref{sec:ProofDiffusion}, the existence of
diffusing orbits is established provided that the smooth functions
\begin{equation}\label{def:no-common-curves}
  \KK^{\ast,+}(I,\mu)=B^{\ast,+} \left(I,\mu\right)-
  \frac{e^{ i\mu\omega^\ast(I,\mu)}-1}{e^{ i\mu\TTT_0(I,\mu)}-1}A^+
  \left( I,\mu\right)\,\,\,\ast=\ff,\bb
\end{equation}
do not vanish on the set $I\in[I_-,I_+]$ where the  corresponding
outer map is defined. Since $A^+$ and $A^-$ are complex conjugate, as well as
$B^{\ast,+}$ and $B^{\ast,-}$,
we do not need to consider the complex conjugate $\KK^{\ast,-}(I,\mu)$.
We check numerically that $\KK^{\ast,+}(I,\mu)\neq 0$ in their
domain of definition. The conditions $\KK^{\ast,+}(I,\mu)\neq 0$
imply the absence of common invariant curves for the inner and outer
maps. This reduces the proof of the Main Result to shadowing,
which therefore leads to the existence of diffusing
orbits. 

It turns out that, in this problem, {\it no large
  gaps} appear. This fact is not surprising since
the elliptic problem has three time scales.

Finally, notice that the complex functions
$\KK^{\ast,+}(I,\mu)$ can be regarded as a 2-dimensional real-valued
function depending smoothly on $(I,\mu)$.  If the dependence on
$\mu$ is non-trivial,
a complex valued function $\KK^{\ast,+}(I,\mu)$ does not vanish at
any point of its domain of definition except for a finite number of values $\mu$.





\subsection{Nature of numerics} \label{subsec:numerics}

In this section we outline which parts of the mechanism are based
on numerics.

\begin{itemize}
\item On each $3$-dimensional energy surface the circular problem has
  a well-defined Poincar{\'e} map $F_J:\Sigma_J \to \Sigma_J$ of a
  $2$-dimensional cylinder $\Sigma_J$ for a range of energies $J$.  For each
  $J$ in some interval $[J_-,J_+]$ we establish the existence of a
  saddle periodic orbit $p_J$ such that $F^7_J(p_J)=p_J$.

\item We show that for all $J\in[J_-,J_+]$ there are two intersections of
$W^s(p_J)$ and $W^u(p_J)$. Each intersection is transversal for almost all
values of $J$, but it becomes tangent at an exceptional (discrete) set of
values of $J$. Nevertheless, we check that at least one of the two
intersections is transversal for each $J\in[J_-,J_+]$; see
Figure~\ref{fig:splittings}.

\item Each transversal intersection $q_J$ gives rise to a homoclinic
  orbit, denoted $\gamma_J$. For each $J\in [J_-,J_+]$ we compute
  several Melnikov integrals of certain quantities related to $\Delta
  H_{\text ell}$ along $\gamma_J$ and $p_J$. Out of these integrals we
  compute the leading terms of the dynamics of the elliptic problem
  and verify a necessary condition for diffusion.
\end{itemize}

The precise hypotheses which are based on numerics are 
Ans{\"a}tze~\ref{ans:NHIMCircular}, \ref{ans:TwistInner}
(Section~\ref{Section:Circular}) and~\ref{ans:B}
(Section~\ref{sec:ProofDiffusion}).

As seen in the appendices
\ref{app:NHIMCircular}-\ref{sec:NumericalStudyInnerOuter}, the
numerical values that we deal with are several orders of magnitude
larger than the estimated error of our computations, and therefore
these computations are reliable. Moreover, all the computations that
we perform are standard and low-dimensional.


\subsection{Main theorem for the $1:7$ resonance}
\label{sec:SettingsAndResults}

The model of the Sun, Jupiter and a massless asteroid in Cartesian
coordinates is given by the Hamiltonian~\eqref{def:Ham:Cartesian}.
First, let us consider the case $\mu=0$, that is, we consider Jupiter
with zero mass. In this case, Jupiter and the asteroid do not 
influence each other and thus the system reduces to two
uncoupled 2-body problems (Sun-Jupiter and Sun-asteroid) which
are integrable.


Let us introduce the so-called
Delaunay variables, denoted by $(\ell, L, \g, G)$, which are
angle-action coordinates of the Sun-asteroid system. The variable
$\ell$ is the mean anomaly, $L$ is the square root of the semi major axis,
$\g$ is the argument of the perihelion and $G$ is the angular
momentum. Delaunay variables are obtained from Cartesian
variables via the following symplectic transformation (see~\cite{ArnoldKN88} for more details and
background, or~\cite[Appendix]{Fejoz:2013} for a straightforward
definition). First define polar coordinates for the position:
\[
q=(r\cos\phi,r\sin\phi).
\]
Then, the actions of the Delaunay coordinates are defined by
\begin{equation}\label{def:L}
  -\frac{1}{2L^2}=\frac{\|p\|^2}{2}-\frac{1}{\|q\|}
  \qquad \textrm{and} \qquad  G=-J-\frac{1}{2L^2}
\end{equation}
(recall that $\mu=0$ for these definitions). Using these actions, the
eccentricity of the asteroid is expressed as
\begin{equation}\label{def:eccentricity}
  e=\sqrt{1-\frac{G^2}{L^2}}.
\end{equation}
To define the angles $\ell$ and $\g$, let $v$ be the true anomaly, so that
\begin{equation}\label{def:Angles}
  \phi=v+\g.
\end{equation}
Then, from $v$ one can obtain the eccentric anomaly $u$ using
\begin{equation}
  \tan\frac{v}{2}=\sqrt{\frac{1+e}{1-e}}\tan\frac{u}{2}.
\end{equation}
From the eccentric anomaly, the mean anomaly is given by Kepler's
equation
\begin{equation}\label{def:MeanAnomaly}
  u-e\sin u=\ell.
\end{equation}

We apply the Delaunay change of coordinates given above to the elliptic
problem; see Appendix~\ref{sec:RotatingToDelaunay}.
In  Delaunay coordinates, the Hamiltonian \eqref{def:Ham:Cartesian}
can be split into the Keplerian part $-1/(2L^2)$, the circular part of
the perturbing function $\mu \Delta H_{circ}$, and the remainder which
vanishes when $e_0=0$:
\begin{equation}\label{def:HamDelaunayNonRot}
  \hat H(L,\ell,G, \g-t,t)=-\frac{1}{2L^2}+\mu\Delta H_\ccirc(L,\ell,G,
  \g-t,\mu)+ \mu e_0\Delta H_\eell(L,\ell,G, \g-t,t,\mu,e_0).
\end{equation}
For $e_0=0$, the circular problem only depends on $\g-t$. To simplify
the comparison with the circular problem, we consider rotating
Delaunay coordinates, in which $\Delta H_\ccirc$ is autonomous. Define
the new angle $g=\g -t$ (the argument of the pericenter, measured in
the rotating frame) and a new variable $I$ conjugate to time $t$.
Then we have
\begin{equation}\label{def:HamDelaunayRot}
  H(L,\ell,G, g,I,t)=-\frac{1}{2L^2}-G+\mu\Delta H_\ccirc(L,\ell,G, g,\mu)+
  \mu e_0 \Delta H_\eell(L,\ell,G, g,t,\mu,e_0)+I.
\end{equation}
In these new variables, the difference in the number of degrees of freedom
of the elliptic and circular problems becomes more apparent. When
$e_0=0$ the system is autonomous and then $I$ is constant, which
corresponds to the conservation of the Jacobi integral
(\ref{def:Jacobi}). Therefore, the circular problem reduces to 2
degrees of freedom. Moreover, it will later be crucial to view the
circular problem as an approximation of the elliptic one, in order to
reduce the (possibly impracticable) numerical computations needed by a
direct approach to the corresponding lower dimensional, and thus
simpler, computations of the circular problem.

Recall that, in this section, we consider the $1:7$ mean motion 
orbital resonance between Jupiter and the asteroid,
that is, the period of the asteroid is
approximately seven times the period of Jupiter.  
In rotating Delaunay variables, this corresponds to
\begin{equation}\label{def:Resonance}
\dot\ell \sim \frac{1}{7}\quad\text{ and }\quad\dot g\sim -1.
\end{equation}
A nearby resonance is
$\dot\ell \sim \frac{1}{7}\quad\text{ and }\quad\dot t\sim 1,$
but we stick to the previous one.

The resonance takes place when $L\sim 7^{1/3}$. We study the
dynamics in a large neighborhood of this resonance and we show
that one can drift along it. Namely, we find trajectories that
keep $L$ close to $7^{1/3}$ while the $G$-component changes
noticeably. Using \eqref{def:eccentricity}, we see that $e$ also
changes by order one.  In this setting, the Main Result can be
rephrased as follows.

\begin{theorem}\label{th:MainTheorem:detail}
  Assume Ans{\"a}tze~\ref{ans:NHIMCircular}, \ref{ans:TwistInner}
  and~\ref{ans:B}. Then there exists $e_0^\ast>0$ such that for every $e_0$
with
  $0<e_0<e_0^\ast$, there exist $T>0$ and an orbit of the
  Hamiltonian~\eqref{def:HamDelaunayRot} which satisfy
  \[
  G(0)>\Gminus\text{ and }G(T)<\Gplus
  \]
  whereas
  \[
  \left| L(t)-7^{1/3}\right|\leq 0.007.
  \]
\end{theorem}

Ans{\"a}tze~\ref{ans:NHIMCircular}
(Section~\ref{Section:Circular}), \ref{ans:TwistInner}
(Section~\ref{Section:Circular}) and~\ref{ans:B}
(Section~\ref{sec:ProofDiffusion}) are hypotheses which, broadly
speaking, assert that the Hamiltonian~\eqref{def:HamDelaunayRot} 
is in general position in  some domain of the phase space;
see also Section~\ref{subsec:numerics}. They are backed up 
by the numerics in the appendices. 


By definition the Hamiltonian \eqref{def:HamDelaunayRot} is autonomous
and thus preserved. Hence, we will restrict ourselves to a level
of energy which, without loss of generality, can be taken as
$H=0$. Therefore, since $|I-G|=\OO(\mu)$, 
the drift in $G$ is equivalent
to the drift in $I$ for orbits satisfying
$\left| L(t)-7^{1/3}\right|\leq 7\mu$.

The proof of Theorem~\ref{th:MainTheorem:detail} is structured as
follows.

In Section \ref{Section:Circular}, we study the dynamics of the
circular problem ($e_0=0$). The Hamiltonian
\eqref{def:HamDelaunayRot} becomes
\begin{equation}\label{def:HamDelaunayCirc}
  H_\ccirc(L,\ell,G, g)=-\frac{1}{2L^2}-G+\mu \Delta H_\ccirc(L,\ell,G, g,\mu).
\end{equation}

\begin{enumerate}
\item  Ansatz~\ref{ans:NHIMCircular} says that for an interval of
Jacobi energies $[J_-,J_+]$ the circular problem has a smooth family
of hyperbolic periodic orbits $\lb_J$, whose stable and unstable
manifolds intersect transversally for each $J \in [J_-,J_+]$. 

\item Ansatz \ref{ans:TwistInner} asserts that the period of these periodic 
orbits changes monotonically with respect to the Jacobi integral.

\item Ansatz \ref{ans:B} asserts that Melnikov functions 
associated with symmetric homolinic orbits created by 
the above periodic orbits are in general position. 
\end{enumerate}

Ansatz \ref{ans:NHIMCircular} implies  the existence of 
a normally hyperbolic invariant cylinder (Corollary \ref{coro:NHIMCircular}). 
Later in the section (Subsections \ref{sec:Circular:Inner} and 
\ref{sec:Circular:Outer}) we calculate the aforementioned outer and 
inner maps for the circular problem (see \eqref{def:InnerMap:ell:intro} 
and \eqref{def:OuterMap:Elliptic:Intro}).

Then in Section \ref{sec:Elliptic} we consider the elliptic case
($0<e_0\ll 1$) as a perturbation of the circular case. Theorem
\ref{th:Elliptic:NHIM} asserts that the normally hyperbolic invariant
cylinder obtained for the circular problem persists, and its stable and
unstable manifolds intersect transversally for each $J \in
[J_-+\delta,J_+ - \delta]$ with small $\delta>0$. These objects give
rise to the inner and outer maps for the elliptic problem.  
  Theorem \ref{th:InnerAndOuter:Elliptic} provides expansions for the
inner and outer maps; see formulas (\ref{def:InnerMap:ell:th}) and
(\ref{def:OuterMap:Elliptic:th}) respectively.

Finally, in Section \ref{sec:ProofDiffusion}, Theorem~\ref{th:Transition} 
completes the proof of Theorem
\ref{th:MainTheorem:detail}. This is done by comparing the inner and
the two outer maps in Lemma \ref{lemma:Averaging} and constructing 
a transition chain of tori. Ansatz \ref{ans:B} ensures that 
the first order of the inner and outer maps of the elliptic problem 
are in general position. It turns out that in this problem there are {\bf no large gaps},
due to the specific structure of times scales and the Fourier series involved.
This contrasts with the typical situation 
near a resonance; see for instance \cite{DelshamsLS06a}.

\begin{notation}
  From now on, we omit the dependence on the mass ratio $\mu$ (keeping in mind
  the question of what would happen if we let $\mu$
  vary). Recall that in this work we consider a realistic value $\mu=10^{-3}$.
\end{notation}

\section{The circular problem}\label{Section:Circular}

\subsection{Normally hyperbolic invariant cylinders}

The circular problem is given by the Hamiltonian
\eqref{def:HamDelaunayRot} with $e_0=0$.  Since it does not depend on
$t$, $I$ is an integral of motion. 
We study the dynamics close to the resonance
$7\dot \ell+\dot g\sim 0$. Since $t$ is a cyclic variable, we consider
the two degree of freedom Hamiltonian of the circular problem $H_\ccirc$,
for which conservation of energy corresponds to conservation of
the Jacobi constant \eqref{def:Jacobi}. 

Note that
the circular problem is reversible with respect to the involution
\begin{equation}\label{def:involution}
 \Psi(L,\ell,G, g,I,t)=(L,-\ell,G,-g,I,-t).
\end{equation}
This symmetry facilitates several numerical computations.

\begin{ansatz}\label{ans:NHIMCircular}
  Consider the Hamiltonian \eqref{def:HamDelaunayCirc} with
  $\mu=10^{-3}$.  In every energy level $J\in [J_-, J_+]=[\Jminus, \Jplus]$,
  there exists a hyperbolic periodic orbit $\lb_J=(L_J(t), \ell_J(t),
  G_J(t), g_J(t))$ of period $T_J$ with
  \[
  \left| T_J-14\pi\right|<60 \mu,
  \]
  such that
  \[
  \left| L_J(t)-7^{1/3}\right|< 7\mu
  \]
  for all $t\in\RR$. The periodic orbit and its period depend smoothly on
$J$.

  Every $\lb_J$ has two branches of stable and
  unstable invariant manifolds $W^{s,j}(\lb_J)$
  and $W^{u,j}(\lb_J)$ for $j=1,2$. For every
  $J\in [J_-,J_+]$ either $W^{s,1}(\lb_J)$ and
  $W^{u,1}(\lb_J)$ intersect transversally or $W^{s,2}(\lb_J)$ and
  $W^{u,2}(\lb_J)$ intersect transversally.
\end{ansatz}

This ansatz is backed up by the numerics of
Appendix~\ref{app:NHIMCircular}.

We study the elliptic problem as a perturbation of the circular
one. In contrast with Ansatz~\ref{ans:NHIMCircular}, 
in the perturbative setting we do \emph{not} reduce the dimension of the phase space
to study the inner and outer dynamics of the circular
problem. Namely, we consider the \emph{Extended Circular Problem}
given by the Hamiltonian~\eqref{def:HamDelaunayRot} with $e_0=0$. In
other words, we keep the conjugate variables $(I,t)$ even if $t$ is a
cyclic variable. Consider the energy level $H=0$, so that
$I=-H_\ccirc(\ell,L,g,G)$.
Therefore, the periodic orbits obtained in Ansatz~\ref{ans:NHIMCircular}
become invariant $2$-dimensional tori which
belong to constant hyperplanes $I=I_0$ for every
\begin{equation}\label{def:IntervalI}
I_0\in [I_-,I_+]=[-J_+,-J_-]=[1.56,1.81].
\end{equation}
The union of these 2-dimensional invariant tori forms a
normally hyperbolic invariant $3$-dimensional manifold $\Lambda_0$, diffeomorphic
to a cylinder. 
Applying the implicit function theorem 
with the energy as a parameter, we see that the cylinder $\Lambda_0$ is
analytic (by Ansatz~\ref{ans:NHIMCircular}, the periodic orbits
are hyperbolic, thus non-degenerate).

\begin{corollary}\label{coro:NHIMCircular}
  Assume Ansatz~\ref{ans:NHIMCircular}. The Hamiltonian
  \eqref{def:HamDelaunayRot} with $\mu=10^{-3}$ and $e_0=0$ has an
  analytic normally hyperbolic invariant $3$-dimensional cylinder
  $\Lambda_0$, which is foliated by $2$-dimensional invariant tori.

  The cylinder $\Lambda_0$ has two branches of stable and unstable
  invariant manifolds, which we call $W^{s,j}(\Lambda_0)$ and
  $W^{u,j}(\Lambda_0)$ for $j=1,2$. In the constant invariant planes
  $I=I_0$,  for every $I_0\in [I_-,I_+]$ either
  $W^{s,1}(\Lambda_0)$ and $W^{u,1}(\Lambda_0)$ intersect transversally or
  $W^{s,2}(\Lambda_0)$ and $W^{u,2}(\Lambda_0)$ intersect transversally.
\end{corollary}


We define a global Poincar{\'e} section and work with maps to reduce the
dimension by one.  Two choices are natural: $\{t=0\}$ and
$\{g=0\}$, since both variables $t$ and $g$ satisfy $\dot t\neq 0$ and
$\dot g\neq 0$.  We choose the section $\{g=0\}$, with associated
Poincar{\'e} map
\begin{equation}\label{def:PoincareMap}
  \PP_0:\{g=0\}\longrightarrow \{g=0\}.
\end{equation}

Since we are studying the resonance
\eqref{def:Resonance}, the intersection of the cylinder $\Lambda_0$
with the section $\{g=0\}$ is formed by seven cylinders (see Figure
\ref{fig:PoincareSection}), denoted $\wt \Lambda_0^j$,
$j=0,\ldots, 6$. Namely,
\begin{equation}\label{def:Cylinder:Poincare}
\Lambda_0\cap \{g=0\}=\wt\Lambda_0=\bigcup_{j=0}^6\wt\Lambda_0^j.
\end{equation}
As a whole, $\wt\Lambda_0$ is a normally hyperbolic
invariant manifold for the Poincar{\'e} map $\PP_0$. One can also consider
the Poincar{\'e} map $\PP_0^7$ ---the seventh iterate of $\PP_0$. 
For this map, each $\wt\Lambda_0^j$ is a normally hyperbolic invariant
manifold (of course, so is their union). We focus on the connected
components $\wt\Lambda_0^j$ since they have a natural system of
coordinates. This system of coordinates is used later to study
the inner and outer dynamics on them. We particularly work with
$\wt \Lambda_0^3$ and $\wt\Lambda_0^4$ for, in every invariant plane
$I=I_0$, they are connected by at least one heteroclinic
connection (of $\PP_0^7$) that is symmetric with respect to the
involution \eqref{def:involution}. We call it a \emph{forward heteroclinic
orbit} if it is asymptotic to $\wt \Lambda_0^3$ in the past and
$\wt\Lambda_0^4$ in the future, and a \emph{backward heteroclinic orbit} if it
is asymptotic to $\wt \Lambda_0^4$ in the past and $\wt\Lambda_0^3$
in the future. 

\begin{figure}[h]
  \begin{center}
    \includegraphics[width=5cm]{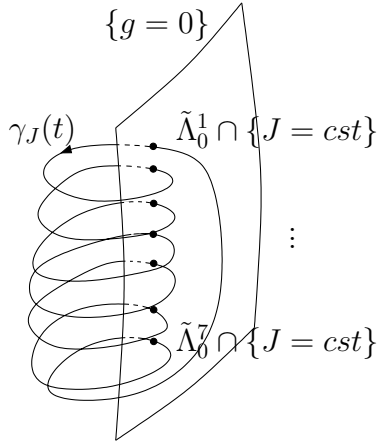}
  \end{center}
  \caption{The periodic orbit obtained for every energy level intersects
    the Poincar{\'e} section $\{g=0\}$ seven times, as shown
    schematically in this picture. Thus, for the
    Poincar{\'e} map $\PP_0$, the normally hyperbolic invariant manifold
    $\wt\Lambda_0$ has seven connected components $\wt\Lambda^0_0,
    \ldots, \wt\Lambda_0^6$.}
  \label{fig:PoincareSection}
\end{figure}

Let $\DD^{\ff}$ (where $\ff$
stands for forward) denote the subset of $[I_-,I_+]$ where $W^u(\wt
\Lambda_0^3)$ and $W^s(\wt \Lambda_0^4)$ intersect transversally and
let $\DD^{\bb}$ (where $\bb$ stands for backward) denote the
subset of $[I_-,I_+]$ where $W^s(\wt \Lambda_0^3)$ and $W^u(\wt
\Lambda_0^4)$ intersect transversally.  By Corollary
\ref{coro:NHIMCircular} we have $\DD^\ff\cup\DD^\bb= [I_-,I_+]$.

\begin{corollary}\label{coro:NHIMCircular:Poincare}
  Assume Ansatz~\ref{ans:NHIMCircular}. The Poincar{\'e} map $\PP^7_0$
  defined in \eqref{def:PoincareMap}, which is induced by the
  Hamiltonian \eqref{def:HamDelaunayRot} with $\mu=10^{-3}$ and
  $e_0=0$, has seven analytic normally hyperbolic invariant manifolds
  $\wt\Lambda_0^j$ for $j=0,\ldots,6$. They are foliated by
  one-dimensional invariant curves. 
For each $j$, there exists an analytic
  function $\GG^j_0: [I_-,I_+]\times\TT\rightarrow (\RR \times
  \TT)^3$,
  \begin{equation}\label{def:NHIM:Param0}
    \GG_0^j(I,t)=\left(\wt
      \GG_0^j(I),0,I,t\right)=\left(\GG_0^{j,L}(I),\GG_0^{j,\ell}(I),
      \GG_0^{j,G}(I),0,I,t\right),
  \end{equation}
  that parameterizes $\wt \Lambda_0^j$:
  \[
  \wt\Lambda^j_0=\left\{ \GG_0^j(I,t):(I,t)\in[I_-,I_+]\times \TT\right\}.
  \]

Moreover,
  the associated invariant manifolds $W^{u}(\wt\Lambda^3_0)$ and
  $W^{s}(\wt\Lambda^4_0)$ intersect transversally within the
  hypersurface $I=I_0$ provided $I_0\in \DD^\ff$. The manifolds
  $W^{s}(\wt\Lambda^3_0)$ and $W^{u}(\wt\Lambda^4_0)$ intersect
  transversally within the hypersurface $I=I_0$ provided $I_0\in
  \DD^\bb$.  
Within the hypersurface $I=I_0$,
each of these intersections has one point on the symmetry axis of the
involution~\eqref{def:involution}.
Let $\Gamma^{\ast}_0$, where $\ast=\ff,\bb$, denote the set of
transversal intersections on the symmetry axis. 
  For both the forward and backward case, there exists an analytic function
  \[
  \CCC^{*}_0: \DD^*\times\RR\rightarrow
  \left(\RR\times\TT\right)^3, \quad (I,t) \mapsto
  \CCC^*_0(I,t),\,\,\,\ast=\ff,\bb\] 
  that parameterizes $\Gamma^{\ast}_0$:
  \[
  \Gamma^*_0=\left\{ \CCC^*_0(I,t)=(\CCC_0^{*,L}(I),\CCC_0^{*,\ell}(I),
    \CCC_0^{*,G}(I),0,I,t):(I,t)\in
    \DD^\ast\times\TT\right\},\,\,\,\ast=\ff,\bb. 
  \]
\end{corollary}

The subscript $0$ in the parameterizations $\GG$ and $\CCC$ indicates
the $g$-coordinate. We keep it although it is redundant in the Poincar{\'e}
section because later we use these parameterizations in the
full phase space.

Again, the implicit function theorem implies that
$W^s(\tilde\Lambda_0^3)$ and $W^u(\tilde\Lambda_0^4)$ are analytic
(taking the distance from the cylinder $\tilde\Lambda^3_0$ or
$\tilde\Lambda^4_0$ as a small parameter, as in~\cite{Mey75} with the
cylinder as factor variable).

Corollary \ref{coro:NHIMCircular} gives global coordinates $(I,t)$ for
each cylinder $\wt\Lambda^j_0$. These coordinates are symplectic with
respect to the canonical symplectic form
\begin{equation}\label{def:InnerDiffForm}
  \Omega_0=dI\wedge dt.
\end{equation}
Indeed, consider the pullback of the canonical form
$dL\wedge d\ell+dG\wedge dg+dI\wedge dt$ to the cylinders
$\wt\Lambda_0^j$. By Corollary \ref{coro:NHIMCircular:Poincare} in the
cylinders we
have $g=0$, $\ell=\GG_0^{j,\ell}(I)$ and
$L=\GG_0^{j,L}(I)$. Then, it is easy to see that the pullback of
$dL\wedge d\ell+dG\wedge dg+dI\wedge dt$ is just $\Omega_0$.

Next we consider the inner and the two outer maps in one of these
cylinders. We choose $\wt\Lambda_0^3$. As explained before,
the reason is that the heteroclinic connections with the following
cylinder $\wt\Lambda_0^4$ intersect the symmetry axis of the
involution \eqref{def:involution} and thus they are easier to 
study numerically (see Figure~\ref{fig:invmfld2_it23}). Since $I$ is
conserved by the inner and outer maps, these maps are integrable and
the variables $(I,t)$ are the action-angle variables. In these variables, it
is easier to understand the influence of ellipticity.

\subsection{The inner map}\label{sec:Circular:Inner}
To study the
diffusion mechanism, one could consider the normally hyperbolic invariant
manifold $\wt\Lambda_0=\bigcup_{j=0}^6\wt\Lambda_0^j$. Nevertheless, since
$\wt\Lambda_0$ is not connected, 
 it is more convenient to consider just one of the
cylinders that form $\wt\Lambda_0$, for instance
$\wt\Lambda_0^3$. Then the inner map
$\FF_0^\inn:\wt\Lambda_0^3 \rightarrow \wt\Lambda_0^3$ is defined as the analytic
Poincar{\'e} map $\PP^7_0$ restricted to the symplectic invariant
submanifold $\wt\Lambda_0^3$. We express $\FF_0^\inn$ using the
global coordinates $(I,t)$ of $\wt\Lambda_0^3$.

Since
$I$ is an integral of motion, the inner map has the form
\begin{equation}\label{def:InnerMap:Circular}
  \FF_0^\inn:\left(\begin{array}{c} I\\
      t
    \end{array}\right)\mapsto \left(\begin{array}{c} I\\
      t+\mu\TTT_0(I)
    \end{array}\right),
\end{equation}
where the function $\TTT_0$ is independent of $t$ because
the inner map preserves the differential form
\eqref{def:InnerDiffForm}, which does not depend on $t$, and $I$
is a first integral.  In fact, $14\pi+\mu\TTT_0(I)$ is the period of
the periodic orbit obtained in Ansatz \ref{ans:NHIMCircular} on the
corresponding energy surface. 
In Section \ref{sec:Circular:Outer}, the function $\TTT_0(I)$ is
written as an integral; see \eqref{def:T0:Integral}.

\begin{ansatz}\label{ans:TwistInner}
  The analytic symplectic inner map $\FF_0^\inn$ defined in
  \eqref{def:InnerMap:Circular} is twist, that is
  \[
  \pa_I \TTT_0(I)\neq 0\qquad \text{for }I\in [I_-,I_+].
  \]
  Moreover, the function $\TTT_0(I)$ satisfies
  \begin{equation}\label{eq:IntervalTwist}
    0<\mu\TTT_0(I)<\pi.
  \end{equation}
\end{ansatz}

This ansatz is based on the numerics of Appendix
\ref{app:NHIMCircular}. The ansatz is crucial in Section
\ref{sec:ProofDiffusion} to prove the existence of a transition chain
of invariant tori.

\subsection{The outer map}\label{sec:Circular:Outer}

First we recall the construction of the outer map in a general perturbative setting.
Next we apply it to the circular problem, and in
section~\ref{sec:Ell:NHIM} to the elliptic problem. The
outer map is sometimes called scattering map; see for instance
\cite{DelshamsLS08}.

Let $\PP_0$ be a map of a compact manifold $M$. Let $\Lambda_0 \subset
M$ be a normally hyperbolic invariant manifold of $\PP_0$, whose inner
map $\PP_0|_{\Lb_0}$ has zero Lyapunov exponents: $\lim_{n\rightarrow
  +\infty} \ln \|d\PP_0^n(z)v\|/n = 0$ for any $z\in \Lb_0$ and $v\in
T_z\Lb_0$ (where $\|\cdot\|$ is some smooth Riemannian norm on $M$).
Further assume that the stable and unstable invariant manifolds of
$\Lambda_0$ intersect transversally.

Let $\PP$ be a small perturbation of $\PP_0$.  Since $\Lb_0$ is
normally hyperbolic it persists under small perturbation of $\PP_0$.
Let $\Lambda\subset M$ be a normally hyperbolic invariant manifold of
$\PP$.

Then, the outer map associated to $\PP$ and $\Lambda$ (a particular
case being $\PP=\PP_0$ and $\Lambda = \Lambda_0$) is defined over some
domain as follows.

\begin{definition}\label{definition:OuterMap}
Assume that $W^s_\Lambda$ and $W^u_\Lambda$ intersect transversally along a
homoclinic manifold $\Gamma$, that is
  $$T_z W^s_\Lambda+T_z W^u_\Lambda=T_z M \quad \mbox{and} \quad T_z
  W^s_\Lambda\cap T_z W^u_\Lambda=T_z\Gamma \quad \mbox{ for }z\in
  \Gamma.$$
  Then,
  we say that $\SSS(x_-)=x_+$, if there exists a point $z\in \Gamma$
  such that for some $C>0$ we have
  \begin{equation}
    \label{eq:outer}
    \mathrm{dist}\left(\PP^n(z),\PP^n(x_\pm)\right)<
    C\lambda^{-|n|}\qquad \text{for all }n\in\ZZ^\pm.
  \end{equation}
\end{definition}

Condition~(\ref{eq:outer}) indeed defines a map $x_- \mapsto x_+$
locally uniquely, as justified in~\cite{DelshamsLS08}.

\begin{remark}
  Since $\Lb$ is normally hyperbolic, for every point $x \in \Lb$ there
  are strong stable and unstable manifolds $W^{ss}(x)$ and
  $W^{su}(x)$. Then $\SSS(x_-)=x_+$ holds if and only if $W^{su}(x_-) \cap
  W^{ss}(x_+) \ne \emptyset$ and the intersection occurs on $\Gamma$.

  When the Lyapunov exponents of the inner dynamics
  $\PP|_\Lambda$ are positive, for the points $x_-$ and $x_+$ to be
  still uniquely defined given $z\in \gamma$, $\lb$ must exceed the
  maximal Lyapunov exponent i.e., the convergence towards $\Lb$ must
  dominate the motion inside of $\Lb$. Otherwise, one cannot
  distinguish if the
  orbit of $z$ is (backward- or forward-) asymptotic to a point of
  $\Lambda$ or to the stable manifold of this point.
\end{remark}

\begin{remark}
  \label{rmk:regularity0}
  If the Lyapunov exponents of the inner map $\PP|_\Lambda$ are zero
  (and, in particular, of the unperturbed map $\PP_0$), the outer map
  $\SSS$ is $C^\infty$.
  If the Lyapunov exponents of the inner map are
  small (thus in particular for a map $\PP$ close enough to $\PP_0$),
  the outer map is $C^k$, where $k$ tends to infinity as the Lyapunov
  exponents tend to $0$.

  Strictly speaking, there is hardly any published regularity theorem
  from which these assertions follow directly. In order to prove them,
  one can first localize in the neighborhood of a small continuous set
  of hyperbolic periodic orbits of $\PP_0$, modify $\PP$ outside this
  neighborhood in order to embed the periodic orbits into a compact
  \emph{invariant} normally hyperbolic cylinder, and characterize 
  the stable and unstable manifolds of the modified system in terms of
  an equation of class $C^k$, the perturbative parameter being the
  distance from the invariant cylinder. Such arguments belong to the well
  understood theory of normally hyperbolic invariant manifolds, and we omit further details,
  refering to the techniques developped in \cite{Fen72, Chaperon:2004},
  or~\cite[Appendix~B]{Bernard:2011} for a closer context.
\end{remark}

\smallskip We apply a variant of this definition
to the dynamics of the circular problem (unperturbed case).  
As in the previous section, we look for an outer map that sends
$\wt\Lambda_0^3$ to itself.  Now one has to be more careful since the
transversal intersections obtained in
Corollary~\ref{coro:NHIMCircular:Poincare} correspond to heteroclinic
connections between $\wt\Lambda_0^3$ and $\wt\Lambda_0^4$ and between
$\wt\Lambda_0^4$ and $\wt\Lambda_0^3$. Thus the outer maps induced by 
$\PP_0^7$ do not leave $\wt\Lambda_0^3$ invariant.
To overcome this problem we
compose these heteroclinic outer maps 
(denoted by $\SSS^{\ff}$ and $\SSS^{\bb}$ below)
with the Poincar{\'e} map $\PP_0$ as many times as necessary so that the
composition sends $\wt\Lambda_0^3$ to itself.


Therefore, the smooth outer maps $\FF_0^{\out,\pm}$ that we consider
connect $\wt\Lambda_0^3$ to itself and are defined as
\begin{equation}\label{def:OuterCircular:Composition}
 \begin{split}
  \FF_0^{\out,\ff}=\PP_0^6\circ\SSS^\ff: \wt\Lambda_0^3\longrightarrow
\wt\Lambda_0^3,\\
  \FF_0^{\out,\bb}=\SSS^\bb\circ\PP_0: \wt\Lambda_0^3\longrightarrow \wt\Lambda_0^3,
 \end{split}
\end{equation}
where
$\SSS^\ff$ is the outer map which connects $\wt\Lambda_0^3$ and
$\wt\Lambda_0^4$ through $W^u(\wt\Lambda_0^3)\cap W^s(\wt\Lambda_0^4)$, and
$\SSS^\bb$ is the outer map which connects $\wt\Lambda_0^4$ and
$\wt\Lambda_0^3$ through $W^u(\wt\Lambda_0^4)\cap W^s(\wt\Lambda_0^3)$.
Note the abuse of notation since the
forward and backward outer maps are only defined provided $I\in
\DD^\ff$ and $I\in\DD^\bb$ respectively and not in the whole cylinder
$\wt\Lambda_0^3$.

The outer map is always exact symplectic; see \cite{DelshamsLS08}.
So, in the circular problem, since $I$ is preserved, the outer maps
are of the form
\begin{equation}\label{def:OuterMap:Circular}
  \FF_0^{\out,\ast}:\left(\begin{array}{c} I\\
      t
    \end{array}\right)\mapsto \left(\begin{array}{c} I\\
      t+\mu\omega^\ast(I)
    \end{array}\right),\,\,\,\ast=\ff,\bb.
\end{equation}

Outer maps can be defined with either discrete or continuous time.
Since the Poincar{\'e}-Melnikov theory is considerably simpler for flows
than for maps, we compute $\FF^{\out,\ast}_0$ using continuous time.
Moreover, in Section \ref{sec:Ell:Outer} we also use flows to
study the outer map of the elliptic problem as a perturbation of
\eqref{def:OuterMap:Circular}.  

The outer map induced by the flow associated to  Hamiltonian
\eqref{def:HamDelaunayRot} with $e_0=0$ does not preserve the section
$\{g=0\}$ but the inner map does. We
reparameterize the flow so that both maps preserve this section.
This reparameterization corresponds to identifying the variable $g$ with
time and is given by
\begin{equation}\label{def:Reduced:ODE:Circ}
  \begin{array}{rlcrl}
    \dps\frac{d}{ds} \ell=&\dps\frac{\pa_L H}{-1+\mu\pa_G\Delta
      H_\ccirc}&\text{   }&\dps\frac{d}{ds} L=&\dps-\frac{\pa_\ell
      H}{-1+\mu\pa_G\Delta H_\ccirc}\\
    \dps\frac{d}{ds} g=&\dps1&\text{   }&\dps\frac{d}{ds} G=&\dps-\frac{\pa_g
      H}{-1+\mu\pa_G\Delta H_\ccirc}\\
    \dps\frac{d}{ds} t=&\dps\frac{1}{-1+\mu\pa_G\Delta H_\ccirc}&\text{
    }&\dps\frac{d}{ds} I=&\dps0
  \end{array}
\end{equation}
where $H$ is Hamiltonian \eqref{def:HamDelaunayRot} with $e_0=0$.
Notice that this reparameterization implies the change of direction of
time.  However, the geometric objects stay the same. In particular,
the new flow also possesses the normally hyperbolic invariant cylinder 
obtained in Corollary \ref{coro:NHIMCircular} and its invariant manifolds.

We refer to this system as the \emph{reduced circular problem}.  
We call it reduced because we identify $g$ with the time $s$. 
Note that the right hand side of equation~\eqref{def:Reduced:ODE:Circ} does
not depend on $t$.
Let $\Phi_0^\ccirc$ denote the flow associated to 
the $(L,\ell,G,g)$ components of equation \eqref{def:Reduced:ODE:Circ} 
(which are independent of $t$ and $I$). Componentwise it can be written as
\begin{equation}\label{def:Flow:Circular}
  \Phi^\ccirc_0\{s,(L,\ell,G,g)\} =
  \left(\Phi^L_0\{s,(L,\ell,G,g)\}, \Phi^\ell_0\{s,(L,\ell,G,g)\},
    \Phi^G_0\{s,(L,\ell,G,g)\}, g+s\right).
\end{equation}

Then, the outer map is computed as follows.
Let
\begin{equation}\label{def:OrbitsForOuterMap}
\begin{split}
\gamma_I^\ast(\sigma) &=
\Phi^\ccirc_0\{\sigma,(\CCC_0^{\ast,L}(I),\CCC_0^{\ast,\ell}(I),\CCC_0^{\ast,G}(I),0)\},\,\,\,\ast=\ff,\bb\\
\la_I^j(\sigma) &=
\Phi^\ccirc_0\{\sigma,(\GG_0^{j,L}(I),\GG_0^{j,\ell}(I),\GG_0^{j,G}(I),0)\}\,\,\,j=3,4\\
\end{split}
\end{equation}
be trajectories of the circular problem. 
Every trajectory $\gamma_I^\ast$ has the
initial condition at the heteroclinic point of the Poincar{\'e} map $\PP_0^7$
obtained in Ansatz \ref{ans:NHIMCircular} with action $I$, since
$\CCC^{\ast}_0$ is the parameterization of the intersection $\Gamma^{\ast}_0$
given in Corollary~\ref{coro:NHIMCircular:Poincare}. 
Every trajectory $\la_I^j$ has the
initial condition at the fixed point of the Poincar{\'e} map $\PP_0^7$,
since $\GG_0^j$ is the parameterization of the invariant cylinder
$\wt\Lambda_0^j$ given in Corollary~\ref{coro:NHIMCircular:Poincare}.

\begin{lemma}\label{lem:Omega0}
  Assume Ansatz~\ref{ans:NHIMCircular}. The functions
  $\omega^{\ff,\bb}(I)$ involved in the definition of the outer maps
  in \eqref{def:OuterMap:Circular} are given by
  \[
  \omega^\ast(I)= \omega^\ast_\out(I)+\omega_\inn^\ast (I),
  \]
where
\begin{equation}\label{def:Omega0:OuterPart}
 \omega^\ast_\out(I)=\omega_+^\ast(I)-\omega^\ast_-(I)
\end{equation}
  with
  \begin{equation}\label{def:Omega0PlusMinus}
 \begin{split}
 \omega_+^\ast(I)&=\lim_{N\rightarrow+\infty}\left(\int_0^{ 14N\pi
      }\frac{(\pa_G\Delta
        H_\ccirc) \circ \gamma_I^\ast(\sigma)}{-1+\mu(\pa_G\Delta
        H_\ccirc) \circ \gamma_I^\ast(\sigma)} \, d\sigma+N\TTT_0(I)\right)\\
\omega^\ast_-(I)&=\lim_{N\rightarrow-\infty}\left(\int_0^{ 14N\pi
      }\frac{(\pa_G\Delta
        H_\ccirc) \circ \gamma_I^\ast(\sigma)}{-1+\mu(\pa_G\Delta
        H_\ccirc) \circ \gamma_I^\ast(\sigma)} \, d\sigma+N\TTT_0(I)\right),\,\,\,\ast=\ff,\bb
\end{split}
\end{equation}
and
\begin{equation}\label{def:Omega0:InnerPart}
\begin{split}
 \omega_\inn^\ff(I)&=\int_0^{ -12\pi
      }\frac{(\pa_G\Delta
        H_\ccirc) \circ \la_I^4(\sigma)}{-1+\mu(\pa_G\Delta
        H_\ccirc) \circ \la_I^4(\sigma)} \, d\sigma\\
\omega_\inn^\bb(I)&=\int_0^{ -2\pi
      }\frac{(\pa_G\Delta
        H_\ccirc) \circ \la_I^3(\sigma)}{-1+\mu(\pa_G\Delta
        H_\ccirc) \circ \la_I^3(\sigma)} \, d\sigma.
\end{split}
\end{equation}
(Recall that $\TTT_0(I)$ is defined by \eqref{def:InnerMap:Circular}).
\end{lemma}

Note that the minus sign in the limit of integration of
$\omega_\inn^\ast(I)$ appears because
the reparameterized flow \eqref{def:Reduced:ODE:Circ}
reverses time. 

Using that the circular problem is symmetric with respect
to \eqref{def:involution} and that the heteroclinic points $\CCC_0^{\ff}$
and $\CCC_0^{\bb}$ belong to the symmetry axis, we find that
$\omega^\ast_-=-\omega^\ast_+$, $\ast=\ff,\bb$.

The geometric interpretation of $\omega^{\ff,\bb}(I)$ is that the $t$-shift
occurs since the homoclinic orbits approach different points of the
same invariant curve in the future and in the past. This shift is
equivalent to the shift in $t$ that appears in Mather's Problem
\cite{Mather96}. See, for instance, formula (2.1) in Theorem 2.1 of
\cite{DelshamsLS00} and the constants $a$ and $b$ used in formula
(1.4) of \cite{BolotinT99}.

\begin{proof}
We compute $\omega^\ff(I)$. The function $\omega^\bb(I)$ is computed
analogously.
  Since the $t$-component of the reduced circular system \eqref{def:Reduced:ODE:Circ} does not depend on $t$, its behavior is given by
\[      \Phi_0^t\{s,(L,\ell,G,g,t)\}=t+\wt\Phi_0 \{s,(L,\ell,G,g)\}\]
where
\begin{equation}\label{eq:Flow:Circ:Time}
 \wt\Phi_0 \{s,(L,\ell,G,g)\}=\int_0^s \frac{1}{-1+\mu\pa_G\Delta
        H_\ccirc\left(\Phi^\ccirc_0\{\sigma,(L,\ell,G,g)\}\right)} \
      d\sigma.
  \end{equation}
Note that, using this reduced flow, the inner map 
  \eqref{def:InnerMap:Circular} is just the $(-14\pi)$-time map in the
  time $s$. Then, the original period of the periodic orbits obtained
  in Ansatz \ref{ans:NHIMCircular} is expressed using the reduced 
  flow as
  \begin{equation}\label{eq:PeriodThroughIntegrals}
    14\pi+\mu\TTT_0(I)=\int_0^{-14\pi}\frac{1}{-1+\mu(\pa_G\Delta
      H_\ccirc) \circ \la^3_I(\sigma)} \, d\sigma.
  \end{equation}
  This allows us to define the function $\TTT_0(I)$
  in~\eqref{def:InnerMap:Circular} through integrals as
  \begin{equation}\label{def:T0:Integral}
    \TTT_0(I)=\int_0^{-14\pi} \frac{(\pa_G\Delta H_\ccirc) \circ
      \la^3_I(\sigma)}{-1+
      \mu (\pa_G\Delta H_\ccirc) \circ \la^3_I(\sigma)} \, d\sigma.
  \end{equation}

  Consider now a point
  $(\CCC_0^{\ff,L}(I),\CCC_0^{\ff,\ell}(I),\CCC_0^{\ff,G}(I),0,I,t)$ in
  $W^u(\wt \Lambda^3_0)\cap W^s(\wt \Lambda^4_0)\cap\{g=0\}$. Since the first
  four components are independent of $t$,
  this point is forward
  asymptotic (in the reparameterized time) to a point
  \[
  \left(\GG_0^{3,L}(I),\GG_0^{3,\ell}(I),\GG_0^{3,G}(I),0,I,t+\mu\omega_+^\ff(I)\right)
  \]
  and backward asymptotic  (in the reparameterized time) to a
  point \[\left(\GG_0^{4,L}(I),\GG_0^{4,\ell}(I),
    \GG_0^{4,G}(I),0,I,t+\mu\omega_-^\ff(I)\right).\]
  Using~\eqref{eq:Flow:Circ:Time}, the functions $\omega_\pm^\ff(I)$ can
  be defined as
  \begin{equation}\label{def:OmegaSmall}
  \begin{split}
    \omega_+^\ff(I)=\lim_{T\rightarrow+\infty}\int_0^{ T}\Bigg(
    \frac{1}{-1+ \mu (\pa_G\Delta H_\ccirc) \circ \gamma^\ff_I(\sigma)}
    -\frac{1}{-1+\mu (\pa_G\Delta H_\ccirc) \circ
      \la^3_I(\sigma)}\Bigg) \, d\sigma\\
\omega_-^\ff(I)=\lim_{T\rightarrow-\infty}\int_0^{ T}\Bigg(
    \frac{1}{-1+ \mu (\pa_G\Delta H_\ccirc) \circ \gamma^\ff_I(\sigma)}
    -\frac{1}{-1+\mu (\pa_G\Delta H_\ccirc) \circ
      \la^4_I(\sigma)}\Bigg) \, d\sigma.
\end{split}
\end{equation}
  Since the system is $14\pi$-periodic in the time $s$ due
  to the identification of $s$ with $g$, it is more convenient to
  write these in integrals as
  \[
\begin{split}
  \omega^\ff_+(I)=\lim_{N\rightarrow+\infty}\int_0^{ 14N\pi
  }\Bigg(\frac{1}{-1+ \mu (\pa_G\Delta H_\ccirc) \circ
    \gamma^\ff_I(\sigma)} -\frac{1}{-1+\mu (\pa_G\Delta H_\ccirc) \circ
    \la^3_I(\sigma) }\Bigg) \,d\sigma.\\
  \omega^\ff_-(I)=\lim_{N\rightarrow-\infty}\int_0^{ 14N\pi
  }\Bigg(\frac{1}{-1+ \mu (\pa_G\Delta H_\ccirc) \circ
    \gamma^\ff_I(\sigma)} -\frac{1}{-1+\mu (\pa_G\Delta H_\ccirc) \circ
    \la^4_I(\sigma) }\Bigg) \,d\sigma.
\end{split}
 \]
  Then, taking~\eqref{eq:PeriodThroughIntegrals} into account, we
  obtain
  \[
    \omega_\pm^\ff(I)=\lim_{N\rightarrow\pm\infty}\Bigg(\int_0^{ 14N\pi }
    \frac{1}{-1+\mu(\pa_G\Delta H_\ccirc) \circ \gamma^\ff_I(\sigma)}\,
    d\sigma+N(14\pi+\TTT_0(I))\Bigg),
  \]
  from which the formulas for $\omega_\pm^\ff$ in \eqref{def:Omega0PlusMinus} follow.

Finally we compute $\omega_\inn^\ff(I)$. This term corresponds 
to the contribution of $\PP^6_0$ to the outer map in formula 
\eqref{def:OuterCircular:Composition}. Then, taking into account that 
$t$ is defined modulo $2\pi$, it is straightforward to obtain  $\omega_\inn^\ff(I)$ 
in \eqref{def:Omega0PlusMinus}.
\end{proof}

\section{The elliptic problem}\label{sec:Elliptic}

Everything is now set up to study the elliptic problem. We obtain
perturbative expansions of the inner and outer maps. To this end, we
apply Poincar{\'e}-Melnikov techniques to the reduced elliptic
problem, which is given by
\begin{equation}\label{def:Reduced:ODE}
  \begin{array}{rlcrl}
    \frac{d}{ds} \ell=&\dps\frac{\pa_L H}{-1+\mu\pa_G\Delta H_\ccirc +\mu e_0\pa_G\Delta H_\eell}&\text{   }&\frac{d}{ds} L=&\dps-\frac{\pa_\ell H}{-1+\mu\pa_G\Delta H_\ccirc +\mu e_0\pa_G\Delta H_\eell}\\
    \frac{d}{ds} g=&1&\text{   }&\frac{d}{ds} G=&\dps-\frac{\pa_g H}{-1+\mu\pa_G\Delta H_\ccirc +\mu e_0\pa_G\Delta H_\eell}\\
    \frac{d}{ds} t=&\dps\frac{1}{-1+\mu\pa_G\Delta H_\ccirc +\mu e_0\pa_G\Delta H_\eell}&\text{   }&\frac{d}{ds} I=&\dps-\frac{\mu e_0\pa_t \Delta H_\eell}{-1+\mu\pa_G\Delta H_\ccirc +\mu e_0\pa_G\Delta H_\eell}.
  \end{array}
\end{equation}
This system is a perturbation of \eqref{def:Reduced:ODE:Circ}. 
One can study the inner map either with this system or with the system
associated to the Hamiltonian \eqref{def:HamDelaunayNonRot}.
Nevertheless, to simplify the exposition
we use only \eqref{def:Reduced:ODE} for both the inner and outer maps.
Again, we 
consider the Poincar{\'e} map associated with this system and the section $\{g=0\}$,
\begin{equation}\label{def:Elliptic:Poincare}
  \PP_{e_0}:\{g=0\}\longrightarrow \{g=0\},
\end{equation}
which is a perturbation of \eqref{def:PoincareMap}.

Two main results are introduced in this section:
\begin{itemize}
\item
  Existence of a normally hyperbolic invariant manifold
  with transversal intersections of its stable and unstable invariant manifolds for the elliptic problem
  (Theorem \ref{th:Elliptic:NHIM}).
\item  Computation of  the $e_0$-expansions of the inner and outer maps associated
  to it (Theorem~\ref{th:InnerAndOuter:Elliptic}).
\end{itemize}

Theorem
\ref{th:Elliptic:NHIM} is a direct consequence of Corollary
\ref{coro:NHIMCircular:Poincare}, because we study the elliptic
problem as a perturbation of the circular one.  

The proof of Theorem
\ref{th:InnerAndOuter:Elliptic} consists of several steps.  In Section
\ref{sec:Expansion:Hamiltonian} we obtain the $e_0$-expansion of the
elliptic Hamiltonian, and from it, in Section \ref{sec:FlowExpansion},
we deduce some properties of the $e_0$-expansion of the flow associated to
the system \eqref{def:Reduced:ODE}. In Section \ref{sec:CylinderExpansion}
we analyze the normally hyperbolic invariant
cylinders $\wt\Lambda_{e_0}^j$, which are the perturbation of the
cylinders $\wt\Lambda_{0}^j$ obtained in Corollary
\ref{coro:NHIMCircular:Poincare}.  This allows us to derive
formulas for the inner map, perturbative in $e_0$. Finally, in Section
\ref{sec:Ell:Outer} we use the expansions to compute the outer
maps using Poincar{\'e}-Melnikov techniques.  
The inner and outer maps are defined over the cylinder
$\wt\Lambda_{e_0}^3$, which is $e_0$-close to the cylinder
$\wt\Lambda_{0}^3$ of Corollary~\ref{coro:NHIMCircular:Poincare}.


\subsection{The specific form of the inner and outer maps}
\label{sec:Ell:NHIM}

For $e_0$ small enough the flow associated to the Hamiltonian
\eqref{def:HamDelaunayRot} has a normally hyperbolic invariant
cylinder $\Lambda_{e_0}$, which is $e_0$-close to $\Lambda_0$ given in
Corollary \ref{coro:NHIMCircular}. Analogously, the Poincar{\'e} map
$\PP_{e_0}$ associated to this system has a
normally hyperbolic invariant cylinder
$\wt\Lambda_{e_0}=\Lambda_{e_0}\cap \{g=0\}$. Moreover, $\wt\Lambda_{e_0}$ is formed
by seven connected components $\wt\Lambda_{e_0}^j$, $j=0,\ldots,6$,
which are $e_0$-close to the cylinders $\wt\Lambda^j_0$ obtained
in Corollary \ref{coro:NHIMCircular:Poincare}.

Recall that, by Corollary \ref{coro:NHIMCircular:Poincare}, in the
invariant planes $I=\text{constant}$ there are forward and backward
transversal heteroclinic connections between $\wt \Lambda_0^3$ and $\wt
\Lambda_0^4$ provided $I\in \DD^\ff$ and $I\in\DD^\bb$
respectively. For the elliptic problem and $e_0$ small enough we 
have transversal heteroclinic connections in slightly smaller
domains. We define
\begin{equation}
 \DD_\de^{\ast}=\{I\in \DD^\ast: \mathrm{dist} (I,\pa\DD^\ast)>\de\},\,\,\,\ast=\ff,\bb.
\end{equation}

\begin{theorem}\label{th:Elliptic:NHIM}
Let $\PP_{e_0}$ be the Poincar{\'e} map associated to the Hamiltonian
  \eqref{def:HamDelaunayRot} and the section $\{g=0\}$.
  Assume Ansatz~\ref{ans:NHIMCircular}. For any $\de>0$, there exists
  $e_0^\ast>0$ such that for $0<e_0<e_0^\ast$ the map $\PP_{e_0}^7$
  has seven normally
  hyperbolic
  locally\footnote{See remark right below.} invariant manifolds
  $\wt\Lambda^j_{e_0}$, which are $e_0$-close in the $\CCC^1$-topology to
  $\wt\Lambda_0^j$.  There exist functions $\GG^j_{e_0}:
  [I_-+\de,I_+-\de]\times\TT\rightarrow (\RR\times\TT)^3$,
  $j=0,\ldots,6$, which can be expressed in coordinates as
  \begin{equation}\label{def:NHIM:Param:ell}
    \GG^j_{e_0}(I,t)=\left(\GG_{e_0}^{j,L}(I,t),\GG_{e_0}^{j,\ell}(I,t), \GG_{e_0}^{j,G}(I,t),0,I,t\right),
  \end{equation}
  that parameterize $\wt\Lambda^j_{e_0}$. In other words
  $\wt\Lambda^j_{e_0}$ is a graph over $(I,t)$ defined as
  \[
  \wt\Lambda^j_{e_0}=\left\{ \GG_{e_0}(I,t):(I,t)\in[I_-+\de,I_+-\de]\times \TT\right\}.
  \]

  Moreover, the invariant manifolds $W^{u}(\wt\Lambda_{e_0}^3)$ and
  $W^{s}(\wt\Lambda_{e_0}^4)$ intersect transversally provided $I\in
  \DD^\ff_\de$ and the invariant manifolds $W^{u}(\wt\Lambda_{e_0}^4)$
  and $W^{s}(\wt\Lambda_{e_0}^3)$ intersect transversally provided
  $I\in \DD^\bb_\de$.  One of these intersections is
  $e_0$-close in the $\CCC^1$-topology to the manifolds
  $\Gamma^{\ff,\bb}_0$ defined in Corollary
  \ref{coro:NHIMCircular:Poincare}.  
\end{theorem}

Let $\Gamma^{\ff,\bb}_{e_0}$ denote these intersections. There exist
functions
\[
\CCC^\ast_{e_0}(I,t)=\left(\CCC_{e_0}^{\ast,L}(I,t),\CCC_{e_0}^{\ast,\ell}(I,t),
  \CCC_{e_0}^{\ast,G}(I,t),0,I,t\right),\,\,\,\ast=\ff,\bb
\]
that parameterize them; namely,
\[
\Gamma^\ast_{e_0}=\left\{
  \CCC^\ast_{e_0}(I,t):(I,t)\in[I_-+\de,I_+-\de]\times\TT\right\},\,\,\,\ast=\ff,\bb.
\]

For the elliptic problem, the coordinates $(I,t)$ are symplectic not
with respect to the canonical symplectic form $dI\wedge dt$. Indeed, if
we pull back the canonical form $dL\wedge d\ell+dG\wedge
dg+dI\wedge dt$ to the cylinders $\wt\Lambda_{e_0}^j$, we obtain the
symplectic form
\begin{equation}\label{def:InnerDiffForm:Ell}
 \Omega^j_{e_0}=\left(1+e_0 a^j_1(I,t)+e_0^2a^j_2 (I,t)+e_0^3
   a^j_\geq(I,t)\right)dI\wedge dt,
\end{equation}
for certain functions $a_k^j:[I_-,I_+]\times\TT\rightarrow \RR$. The
functions $a^j_\geq$ are the $e_0^3$ Taylor remainders, and thus
depend on $e_0$ even if we do not write explicitly this dependence to
simplify notation.

\begin{remark}\label{rmk:regularity}
  The objects and maps of Theorem~\ref{th:Elliptic:NHIM} have increasing
  regularity when $e_0$ tends to $0$. Indeed, by Gronwall's inequality
  the Lyapunov exponents of $\Lambda_{e_0}$ tend to zero with $e_0$.
  So for every $k\geq1$, if $e_0$ is small enough, the invariant
  manifold $\Lambda_{e_0}$ and subsequent objects are of
  class $C^k$ (see Remark~\ref{rmk:regularity0}). For the sake of
  simplicity, we do not henceforth emphasize regularity issues. 
  The main point is that for $e_0$ small enough all objects of our
  construction are smooth enough, and in particular it is
  possible to apply the KAM theorem to the invariant manifolds
  $\tilde\Lambda_0^j$.
\end{remark}

\begin{remark}
  Theorem \ref{th:Elliptic:NHIM} only guarantees local invariance for
  $\wt\Lambda^j_{e_0}$.
  Namely, the boundary might not be invariant.  Nevertheless, in
  Section~\ref{sec:ProofDiffusion} we show the existence of
  invariant tori in $\wt\Lambda^j_{e_0}$ that act as
  boundaries of $\wt\Lambda^j_{e_0}$. Thanks to these tori, one can
  choose $\wt\Lambda^j_{e_0}$ to be invariant. For this reason,
  we refer to $\wt\Lambda^j_{e_0}$ as a normally hyperbolic
  invariant manifold.
\end{remark}

Our analysis depends heavily on the harmonic structure of the
various maps involved. Thus we need the following definition.

\begin{notation}
  For every function $f$ that is $2\pi$-periodic in $t$, let
  $\NNN(f)$ denote the set of integers $k \in \ZZ$ such that the $k$-th
  harmonic of $f$ (possibly depending on other variables) is non-zero.
\end{notation}

One can define inner and outer
maps 
in the invariant cylinder $\wt\Lambda_{e_0}^3$ given
in Theorem \ref{th:Elliptic:NHIM}
as we have done in $\wt\Lambda_0^3$ for the circular problem.
The next sections are devoted to the perturbative analysis of these maps.
We state here the main outcome.

\begin{theorem}\label{th:InnerAndOuter:Elliptic}
Let $\PP_{e_0}$ be the Poincar{\'e} map associated to the Hamiltonian
  \eqref{def:HamDelaunayRot} and the section $\{g=0\}$.
  Assume Ansatz~\ref{ans:NHIMCircular}. The normally hyperbolic
  invariant manifold $\wt\Lambda^3_{e_0}$ given in Theorem
  \ref{th:Elliptic:NHIM} of the map $\PP^7_{e_0}$ 
  has associated inner and outer maps.
\begin{itemize}
  \item The inner map is of the form
    \begin{equation}\label{def:InnerMap:ell:th}
      \FF_{e_0}^\inn:\left(\begin{array}{c} I\\
          t
        \end{array}\right)\mapsto \left(\begin{array}{l} I+ e_0
          A_1(I,t)+e_0^2 A_2(I,t)+\OO\left(e_0^3\right)\\ 
          t+\mu\TTT_0(I)+e_0 \TTT_1(I,t)+e_0^2 \TTT_2(I,t)+\OO\left(
            e_0^3\right) 
        \end{array}\right),
    \end{equation}
    where the functions $A_1$, $A_2,$ $\TTT_1,$ and $\TTT_2$ satisfy
    \begin{align}
      \NNN\left(A_1\right)=\{\pm 1\},\,\,\,
      \NNN\left(A_2\right)=\{0,\pm 1,\pm
      2\}\label{eq:InnerMap:I:1,2:Harmonics:0}\\ 
      \NNN\left(\TTT_1\right)=\{\pm 1\},\,\,\,
      \NNN\left(\TTT_2\right)=\{0,\pm 1,\pm
      2\}\label{eq:InnerMap:t:1,2:Harmonics:0}. 
    \end{align}
  \item The outer maps are of the form
    \begin{equation}\label{def:OuterMap:Elliptic:th}
      \FF_{e_0}^{\out,\ast}:\left(\begin{array}{c} I\\
          t
        \end{array}\right)\mapsto \left(\begin{array}{c} I+
          e_0B^\ast(I,t) +\OO\left(e_0^2\right)\\ 
          t+\mu\omega^\ast(I)+\OO(e_0)
        \end{array}\right),\,\,\,\ast=\ff,\bb,
    \end{equation}
    where the functions $B^\ast$ satisfy
    \begin{equation}\label{eq:OuterMap:I:Harmonics:0}
      \NNN\left(B^\ast\right)=\{\pm 1\}.
    \end{equation}
  \end{itemize}
\end{theorem}


\subsection{The $e_0$-expansion of the elliptic Hamiltonian}\label{sec:Expansion:Hamiltonian}

Now we expand $\Delta H_\eell$ in \eqref{def:HamDelaunayRot}
with respect to $e_0$. These expansions are used in Sections
\ref{sec:FlowExpansion}, \ref{sec:CylinderExpansion} and
\ref{sec:Ell:Outer}. The most important goal is to see which harmonics
in $t$ have $e_0$ and $e_0^2$ terms.
Note that the circular problem is independent of $t$.

Define the function
\begin{equation}\label{def:Potential}
  \BB(r,v,g, t)=\frac{1}{\left|r e^{i(v+g-t)}-r_0(t) e^{iv_0(t)}\right|}.
\end{equation}
This function is the potential $|q-q_0(t)|^{-1}$
expressed in terms of  $g=\hat g-t$, where $\hat g$ is the argument of the perihelion,
the  true  anomaly $v$ of the asteroid defined in \eqref{def:Angles} and the radius $r$.
The functions $r_0(t)$ and $v_0(t)$ are the radius and the true anomaly of Jupiter.
The functions $r_0(t)$ and $v_0(t)$ are the only ones in the definition of $\BB$
that depend on $e_0$.

Then, the perturbation in \eqref{def:HamDelaunayNonRot} is expressed as
\[
\begin{split}
  \mu\Delta H_\ccirc(L,\ell,G,g) +\mu e_0\Delta H_\eell(L,\ell,G,g,t)=&-\frac{1-\mu}{\mu} \BB\left(-\frac{r}{\mu},v,g, t\right)\\
  &\left.-\frac{\mu}{1-\mu} \BB\left(\frac{r}{1-\mu},v,g-t, t\right)+\frac{1}{r}\right|_{(r,v)=(r(L,\ell,G), v(L,\ell,G))}.
\end{split}
\]



First we deduce some properties of the expansion of the function
$\BB$:
\begin{equation}\label{def:ExpansionB}
  \BB(r,v, g, t)= \BB_0(r,v,g)+e_0 \BB_1(r,v,g, t)+e_0^2 \BB_2(r,v, g,
  t)+\OO\left(e_0^3\right).
\end{equation}
From these properties, we deduce the expansion of $\Delta H_\eell$.

\begin{lemma}\label{lemma:ExpansionB}
 The functions in the $e_0$-expansion of $\BB$ have the following properties.
  \begin{itemize}
  \item $\BB_0$ satisfies  $\NNN(\BB_0)=\{0\}$.
  \item $\BB_1$ satisfies $\NNN(\BB_1)=\{\pm 1\}$ and is given by
    \begin{equation}\label{def:B1}
      \BB_1(r,v,g, t)=-\frac{1}{2\Delta^3(r,v,g)}\left(2\cos t-3r
        \cos(v+g+t)+r\cos(v+g-t)\right), 
    \end{equation}
    where
    \[
    \Delta(r,v,g)=\left(r^2+1-2r\cos (v+g)\right)^{1/2}.
    \]
  \item $\BB_2$ satisfies $\NNN(\BB_2)=\{0,\pm 1,\pm 2\}$.
  \end{itemize}
\end{lemma}
Note that the elliptic problem is a peculiar perturbation of the
circular problem in the sense that the $k$-th $e_0$-order has
non-trivial $t$-harmonics at most up to order $k$. This fact is
crucial when we compare the inner and outer dynamics in Section
\ref{sec:ProofDiffusion}.

\begin{proof}[Proof of Lemma \ref{lemma:ExpansionB}]
  We look for the $e_0$-expansions of the
  functions $r_0(t)$ and $v_0(t)$ involved in the definition of $\BB$.
  We obtain them using the eccentric, true
  and mean anomalies of Jupiter.

  From the relation $t=u_0-e_0\sin u_0$ (see \eqref{def:MeanAnomaly}), we obtain that
  \[
  u_0(t)=t+e_0\sin t+\frac{e_0^2}{2}\sin 2t+\OO\left(e_0^3\right).
  \]
  Then, using $r_0=1-e_0\cos u_0$,
  \[
  r_0(t)=1-e_0\cos t+e_0^2\sin^2 t+\OO\left(e_0^3\right).
  \]
  For the eccentric anomaly we use
  \[
  \tan\frac{v_0}{2}=\sqrt{\frac{1+e_0}{1-e_0}}\tan\frac{u_0}{2}
  \]
  (see \eqref{eq:FromUtoV}), to obtain
  \[
  v_0=u_0+e_0\sin u_0+e_0^2\left(\frac{9}{2}\sin u_0-2\sin 2u_0\right)+\OO\left(e_0^3\right)
  \]
  and then
  \[
  v_0(t)=t+2e_0\sin t+e^2_0\left(\frac{9}{2}\sin t-\sin 2t\right)+\OO\left(e_0^3\right).
  \]
  Plugging $r_0(t)$ and $v_0(t)$ into \eqref{def:Potential}, it can be easily seen that
  the expansion \eqref{def:ExpansionB} satisfies all the properties of $\BB_0$, $\BB_1$ and $\BB_2$
  stated in the lemma.
\end{proof}

One can now easily study the first order expansion of $\Delta H_\eell$:
\[
\Delta H_\eell=\Delta H^1_\eell+e_0\Delta H^2_\eell+\OO\left(e_0^2\right).
\]
(recall from formula~\eqref{def:HamDelaunayRot} that one power
of $e_0$ has already been factored out of the definition of $\Delta
H_\eell$). In particular,
\begin{equation}\label{def:Ham:Elliptic:order1}
  \begin{split}
    \Delta H^1_\eell(L,\ell,G,g,t) =&
    -\frac{1-\mu}{\mu}\BB_1\left(-\frac{r(L,\ell,G)}{\mu},v(L,\ell,G),g,t\right)\\
    &-\frac{\mu}{1-\mu}\BB_1\left(\frac{r(L,\ell,G)}{1-\mu},v(L,\ell,G),g,t\right),
  \end{split}
\end{equation}
where $\BB_1$ is the function defined in Lemma \ref{lemma:ExpansionB}.

\begin{corollary}\label{coro:ExpansioHamiltonian}
  The functions in the $e_0$-expansion  of $\Delta H_\eell$ satisfy
  $$\NNN\left(\Delta H^1_\eell\right)=\{\pm 1\} \quad \mbox{and} \quad
  \NNN\left(\Delta H^2_\eell\right)=\{0,\pm 1,\pm 2\}.$$
\end{corollary}

\subsection{Perturbative analysis of the flow}\label{sec:FlowExpansion}
Before studying the inner and outer maps perturbatively, we need to study the
first orders with respect to $e_0$ of the flow
$\Phi_{e_0}\{s,(L,\ell,G,g,I,t)\}$ associated to the vector field
\eqref{def:Reduced:ODE}, particularly their dependence on the variable $t$.
Recall that we already know the dependence on $t$ of  the
$0$-order thanks to formulas \eqref{def:Flow:Circular} and
\eqref{eq:Flow:Circ:Time}.

\begin{lemma}\label{lemma:FLowExpansion}
 The flow
  $\Phi_{e_0}\{s,(L,\ell,G,g,I,t)\}$ has a perturbative expansion
  \[
\begin{split}
  \Phi_{e_0}\{s,(L,\ell,G,g,I,t)\}=&\Phi_0\{s,(L,\ell,G,g,I,t)\}+e_0 \Phi_1\{s,(L,\ell,G,g,I,t)\}\\
&+e_0^2\Phi_2\{s,(L,\ell,G,g,I,t)\}+\OO\left(e_0^3\right)
\end{split}
  \]
  that satisfies
  \begin{align}
    \NNN\left(\Phi_1\{s,(L,\ell,G,g,I,t)\}\right)&=\{\pm 1\}\label{eq:Flow:1:Harmonics}\\
    \NNN\left(\Phi_2\{s,(L,\ell,G,g,I,t)\}\right)&=\{0,\pm 1,\pm 2\}.\label{eq:Flow:2:Harmonics}
  \end{align}
\end{lemma}
\begin{proof}
Let $z=(L,\ell,G,g,I)$ and let $\XX_{e_0}$ denote the vector field \eqref{def:Reduced:ODE}, which has expansion
\[
\XX_{e_0}=\XX_{0}+e_0\XX_{1}+e_0^2\XX_{2}+\OO\left(e_0^3\right).
\]
First we prove \eqref{eq:Flow:1:Harmonics}.  The $e_0$-order $\Phi_1$ is a solution of the ordinary differential equation
  \[
  \frac{d}{ds} \xi=D\XX_0\left(\Phi_0\{s,(z,t)\}\right)\xi+\XX_1\left(\Phi_0\{s,(z,t)\}\right)
  \]
  with initial condition $\xi(0)=(0,0)$. By \eqref{def:Reduced:ODE:Circ},
$\XX_0$ is independent of $t$ and thus,
  \[
  D\XX_0\left(\Phi_0\{s,(z,t)\}\right)=D\XX_0\left(\Phi_0^\mathrm{circ}(s,z)\right),
  \]
 where $\Phi_0^\mathrm{circ}$ is defined in \eqref{def:Flow:Circular}.
Then, this term is also independent of $t$. From Corollary
\ref{coro:ExpansioHamiltonian}, we deduce that  $\NNN(\XX_1)=\{\pm 1\}$
and thus $\XX_1$ is written as
  \[
  \XX_1(z,t)=\XX_1^+(z)e^{it}+\XX_1^-(z)e^{-it},
  \]
  Therefore, using formulas \eqref{def:Flow:Circular} and
\eqref{eq:Flow:Circ:Time}, we have
  \[
  \XX_1\left(\Phi_0\{s,(z,t)\}\right)=\left(\XX_1^+\left(\Phi_0^\mathrm{circ}\{s,z\}\right)e^{i\wt\Phi_\text{0}\{s,z\}}\right)e^{it}+\left(\XX_1^-\left(\Phi_0^\mathrm{circ}\{s,z\}\right)e^{i\wt\Phi_{0}\{s,z\}}\right)e^{-it}.
  \]
  To prove \eqref{eq:Flow:1:Harmonics}, it is enough to use variation of
constants formula. Consider $M_{z}(s)$, the fundamental matrix of the linear equation
  \[
  \frac{d}{ds}\xi=D\XX_0\left(\Phi_0^\mathrm{circ}(s,z)\right)\xi.
  \]
  Then
  \[
  \Phi_1\{s,(z,t)\}=\Phi_1^{+}\{s,z\}e^{it}+\Phi_1^{-}\{s,z\}e^{-it}
  \]
  with
  \[
  \Phi_1^{\pm}\{s,z\}=M_{z}(s)\int_0^s M^{-1}_{z}(\sigma)\left(\XX_1^\pm\left(\Phi_0^\mathrm{circ}\{s,z\}\right)e^{\pm i\wt\Phi_\text{0}\{s,z\}\}}\right)d\sigma.
  \]

  The proof of \eqref{eq:Flow:2:Harmonics} follows the same lines. Indeed,
$\Phi_2$ is a solution of an equation of the form
  \[
  \frac{d}{ds} \xi=D\XX_0\left(\Phi_0^\mathrm{circ}\{s,z\}\right)\xi+\Xi(s,g,I,t)
  \]
  with initial condition $\xi(0)=(0,0,0)$. The function $\Xi$ is given in terms of the previous orders of $\XX_{e_0}$ and $\Phi_{e_0}$ as
  \[
  \Xi=\frac{1}{2}D^2\XX_0\left(\Phi_0^\mathrm{circ}\right)\left(\Phi_1\right)^{\otimes 2}+D\XX_1\left(\Phi_0^\mathrm{circ}\right)\Phi_1+\XX_2\left(\Phi_0^\mathrm{circ}\right),
  \]
  so it satisfies $\NNN(\Xi)=\{0,\pm 1, \pm 2\}$.
  Since the homogeneous linear equation is the same as the one for
  $\Phi_1$ and does not depend on $t$, we easily obtain
  \eqref{eq:Flow:2:Harmonics}.
\end{proof}

\subsection{Perturbative analysis of the invariant cylinder and its
  inner map}
\label{sec:CylinderExpansion}

This section is devoted to studying the normally hyperbolic
invariant manifold of the elliptic problem $\wt\Lambda_{e_0}^3$, whose
existence was proved in Theorem \ref{th:Elliptic:NHIM}, and the
associated inner map. We study the inner map of the elliptic problem
as a perturbation of~\eqref{def:InnerMap:Circular}, taking $e_0$ as the
small parameter.  The inner map is denoted by
$\FF_{e_0}^\inn:\wt\Lambda^3_{e_0}\rightarrow\wt\Lambda^3_{e_0}$.
It is defined as the $(-14\pi)$-Poincar{\'e} map of the flow
$\Phi_{e_0}$, given in Lemma \ref{lemma:FLowExpansion}, restricted
to $\wt\Lambda_{e_0}^3$.

We want to see which $t$-harmonics appear in the first orders of the inner
map, and we also want to compute the first order of the $I$-component. To
this end we use the classical theory of normally hyperbolic invariant manifolds \cite{Fenichel74, Fenichel77}. This theory  ensures
the existence of the functions $\GG^j_{e_0}$ parameterizing the normally
hyperbolic manifolds $\wt\Lambda_{e_0}^j$ of the map $\PP_{e_0}^7$.
Moreover, they can be made unique imposing
\begin{equation}\label{def:UniquenessNHIM}
\pi_I\GG^j_{e_0}(I,t)=I \quad\textrm{and} \pi_t\GG^j_{e_0}(I,t)=t,
\end{equation}
where $\pi_\ast$ is the projection with respect to the corresponding
component of the function. Since we only need the cylinder $\wt\Lambda_{e_0}^3$ and the dynamics on it, we consider the case $j=3$. The map $\GG_{e_0}^3$ satisfies the invariance equation
\begin{equation}\label{eq:Invariance:NHIM}
  \wt\PP_{e_0}\circ \GG^3_{e_0}= \GG^3_{e_0}\circ\FF^{\inn}_{e_0},
\end{equation}
where $\wt\PP_{e_0}=\PP_{e_0}^7$ and $\FF_{e_0}^{\inn}$ is the inner map
of the elliptic problem, namely the Poincar{\'e} map $\PP_{e_0}^7$ restricted
to the cylinder $\wt\Lambda_{e_0}^3$.

Since we have regularity with respect to parameters, the invariance equation allows us
to obtain expansions of the parameterizations of both $\wt\Lambda_{e_0}^3$ and the inner map
$\FF_{e_0}^{\inn}$ with respect to $e_0$. Let us expand $\GG^3_{e_0}$ and $\FF^{\inn}_{e_0}$ as
\begin{align}
  \GG^3_{e_0}&=\GG^3_0+e_0\GG^3_1+e_0^2\GG^3_2+\OO\left(e_0^3\right)\label{def:ExpansionG}\\
  \FF^{\inn}_{e_0}&=\FF^{\inn}_0+e_0\FF^{\inn}_1+e_0^2\FF^{\inn}_2+\OO\left(e_0^3\right).\label{def:ExpansionInnerMap}
\end{align}
Then, $\GG^3_0$ is the function defined in \eqref{def:NHIM:Param0} and
$\FF^{\inn}_0$ is the inner map of the circular problem obtained in
\eqref{def:InnerMap:Circular}, which is defined in $\wt\Lambda^3_0$. Recall that
\begin{equation}\label{def:Harmonics:PoincareMap}
\wt\PP_{e_0}(L,\ell,G,0,I,t)=\PP_{e_0}^7(L,\ell,G,0,I,t)=\Phi_{e_0}\{-14\pi,(L,\ell,G,0,I,t)\}.
\end{equation}
Then we have
\[
\NNN\left(\wt\PP_1\right)=\{\pm 1\}\,\,\,\text{ and }\,\,\,\NNN\left(\wt\PP_2\right)=\{0,\pm 1,\pm 2\}.
\]
Expanding equation \eqref{eq:Invariance:NHIM} with respect to $e_0$, we
deduce the properties of the inner map. They are summarized in the next
lemma, which reproduces the part of Theorem \ref{th:InnerAndOuter:Elliptic}
referring to the inner dynamics.

Recall that $\la_I^3(\sigma)$ has been defined in
\eqref{def:OrbitsForOuterMap}, $\wt\Phi_0$ in
\eqref{eq:Flow:Circ:Time} and $\GG^3_0$ in Corollary
\ref{coro:NHIMCircular:Poincare}.

\begin{lemma}\label{lemma:InnerMap:Elliptic}
  Assume Ansatz~\ref{ans:NHIMCircular}. The expansions of the
  functions $\GG_{e_0}^3$ and $\FF^\inn_{e_0}$ in
  \eqref{def:ExpansionG} and \eqref{def:ExpansionInnerMap} satisfy
  that
\[
 \NNN\left(\GG_{1}^3\right)=\{\pm 1\}\,\,\,\text{ and }\,\,\,\NNN\left(\GG_{2}^3\right)=\{0,\pm 1,\pm 2\}
\]
and
\[
 \NNN(\FF^\inn_1)=\{\pm 1\}\,\,\,\text{ and }\,\,\,\NNN(\FF^\inn_2)=\{0,\pm 1,\pm 2\}.
\]
Namely, the inner map is of the form
  \begin{equation}\label{def:InnerMap:ell}
    \FF_{e_0}^\inn:\left(\begin{array}{c} I\\
        t
      \end{array}\right)\mapsto \left(\begin{array}{l} I+ e_0 A_1(I,t)+e_0^2 A_2(I,t)+\OO\left(\mu e_0^3\right)\\
        t+\mu\TTT_0(I)+e_0 \TTT_1(I,t)+e_0^2 \TTT_2(I,t)+\OO\left(\mu e_0^2\right)
      \end{array}\right),
  \end{equation}
  where the functions $A_1$, $A_2$ $\TTT_1$ and $\TTT_2$ satisfy
  \begin{align}
    \NNN\left(A_1\right)=\{\pm 1\},\,\,\, \NNN\left(A_2\right)=\{0,\pm 1,\pm 2\}\label{eq:InnerMap:I:1,2:Harmonics}\\
    \NNN\left(\TTT_1\right)=\{\pm 1\},\,\,\, \NNN\left(\TTT_2\right)=\{0,\pm 1,\pm 2\}\label{eq:InnerMap:t:1,2:Harmonics}.
  \end{align}
  Moreover $A_1$ can be split as
  \[
  A_1(I,t)=A_1^+(I)e^{it}+A_1^-(I)e^{-it},
  \]
  with
  \begin{equation}\label{def:A:plus/minus}
    A_1^\pm(I)=\mp i\mu \int_0^{-14\pi}\frac{\Delta H_{\eell}^{1,\pm}\circ
\la_I^3(\sigma)}{-1+\mu \pa_G\Delta H_\ccirc\circ\la_I^3(\sigma)}e^{\pm i\wt\la_I^3(\sigma)}d\sigma,
  \end{equation}
  where the functions $\Delta H_{\eell}^{1,\pm}$ are defined as
  \[
  \Delta H_{\eell}^1(L,\ell,G,g,t)=\Delta H_{\eell}^{1,+}(L,\ell,G,g)e^{it}+\Delta H_{\eell}^{1,\pm}(L,\ell,G,g)e^{-it},
  \]
  and
\begin{equation}\label{def:Orbit:Gamma3:time}
 \wt\la_I^3(\sigma)=
\wt\Phi_0\{\sigma,(\GG_0^{3,L}(I),\GG_0^{3,\ell}(I),\GG_0^{3,G}(I),0)\}.
\end{equation}
\end{lemma}
From the properties of $\GG_{e_0}^3$, we deduce the properties of the
symplectic form $\Omega_{e_0}^3$ defined on the cylinder
$\wt\Lambda_{e_0}^3$. Recall that $\Omega_{e_0}^3$ is the pullback of
the symplectic form $dL\wedge d\ell+dG\wedge dg+dI\wedge dt$ on the
invariant cylinder $\wt\Lambda_{e_0}^3$. In
equation~\eqref{def:InnerDiffForm:Ell} we called $a_j^3$ the
coefficients of its expansion:
$$\Omega^3_{e_0}=\left(1+e_0 a^3_1(I,t)+e_0^2a^3_2 (I,t)+e_0^3
   a^3_\geq(I,t)\right)dI\wedge dt.$$

\begin{corollary}\label{coro:SymplecticForm}
  Assuming Ansatz~\ref{ans:NHIMCircular}, the functions $a_1^3$ and
  $a_2^3$ satisfy
  \[
  \NNN\left(a_1^3\right)=\{\pm 1\}\,\,\,\text{ and }\,\,\,\NNN\left(a_2^3\right)=\{0,\pm 1,\pm 2\}.
  \]
\end{corollary}

\begin{proof}[Proof of Lemma \ref{lemma:InnerMap:Elliptic}]
In the proof we omit the superscript $3$ of the terms in the expansion of $\GG^3_{e_0}$.
  Expanding equation \eqref{eq:Invariance:NHIM} with respect to $e_0$, we have that  the first terms satisfy
  \begin{align}
    \wt\PP_0\circ \GG_0=&\GG_0\circ\FF^\inn_0\label{eq:Invariance:NHIM:0}\\
    \wt\PP_1\circ\GG_0+\left(D\wt\PP_0\circ\GG_0\right)\GG_1=&\GG_1\circ\FF_0^\inn+
    \left(D\GG_0\circ\FF_0^\inn\right)\FF_1^\inn\label{eq:Invariance:NHIM:1}\\
    \wt\PP_2\circ\GG_0+\left(D\wt\PP_1\circ\GG_0\right)\GG_1+
    \frac{1}{2}\left(D^2\wt\PP_0\circ\GG_0\right)\GG_1^{\otimes 2}+&
    \nonumber \\
    +\left(D\wt\PP_0\circ\GG_0\right)\GG_2=&\GG_2\circ\FF_0^\inn+
    \left(D\GG_1\circ\FF_0^\inn\right)\FF_1^\inn\nonumber\\
    +\frac{1}{2}\left(D^2\GG_0\circ\FF_0^\inn\right)(\FF_1^\inn)^{\otimes 2}+
    &\left(D\GG_0\circ\FF_0^\inn\right)\FF_2^\inn.\label{eq:Invariance:NHIM:2}
  \end{align}
  By the uniqueness condition \eqref{def:UniquenessNHIM},  $\GG_{1}$ is of the form
  \[
  \GG_1(g,I,t)=\left(\wt \GG_1(g,I,t),0,0,0\right)
  \]
  with $\wt \GG_1(g,I,t)=(\GG_1^L(g,I,t),\GG_1^\ell(g,I,t),\GG_1^G(g,I,t))$.

 Equation \eqref{eq:Invariance:NHIM:0} corresponds to the inner dynamics
 of the circular problem. We use equations \eqref{eq:Invariance:NHIM:1}
 and \eqref{eq:Invariance:NHIM:2} to deduce the properties of $\FF_1^\inn$
 and $\FF_2^\inn$ respectively.
  These equations can be solved iteratively starting with \eqref{eq:Invariance:NHIM:1}. Since
  \begin{equation}\label{def:DifferentialG}
    D\GG_0=\left(\begin{array}{c}D\wt \GG_0\\ \text{Id}\end{array}\right)\text{ and } D\GG_i=\left(\begin{array}{c}D\wt \GG_i\\ 0\end{array}\right)\,\,\text{ for } i\geq 1,
  \end{equation}
  we have
  \[
  \FF_1^{\inn,\ast}=\pi_\ast\left(\wt \PP_1\circ\GG_0+\left(D\wt\PP_0\circ\GG_0\right)\wt\GG_1\right),\,\,\,\ast=I,t.
  \]
  Replacing this into \eqref{eq:Invariance:NHIM:1} we obtain an equation for
$\GG_1$. The equation for every Fourier $t$-coefficient is uncoupled. Hence,
using the definition~\eqref{def:Harmonics:PoincareMap}, the $t$-independence
of $\wt\PP_0$, and  the uniqueness of $\GG_1$, we deduce that $\NNN(\GG_1)=\{\pm 1\}$. As a consequence we have  $\NNN(\FF_1^\inn)=\{\pm 1\}$.

  Reasoning analogously and using \eqref{def:Harmonics:PoincareMap} again, we see that $\NNN(\GG_2)=\{0,\pm 1,\pm 2\}$ and $\NNN(\FF_2^\inn)=\{0,\pm 1,\pm 2\}$.

  Now it only remains to prove formula \eqref{def:A:plus/minus}. Recall that
the $I$-component of the inner map can be written as
\[
 \FF_{e_0}^{\inn,I}(I,t)=\Phi_{e_0}^I\left\{-14\pi,\GG_{e_0}(I,t)\right\}
\]
since it is defined as the $(-14\pi)$-Poincar{\'e} map associated to the flow
of system \eqref{def:Reduced:ODE} restricted to the cylinder
$\wt\Lambda_{e_0}^3$. Recall that the minus sign in the time appears because
the system \eqref{def:Reduced:ODE} has the time reversed with respect to the
original one. Then, we apply the Fundamental Theorem of Calculus and use \eqref{def:Reduced:ODE} to obtain
\[
\begin{split}
 \FF_{e_0}^{\inn,I}(I,t)&=\int_0^{-14\pi}\frac{d}{ds}\Phi_{e_0}^I\left\{s,\GG_{e_0}(I,t)\right\}ds\\
&=-\int_0^{-14\pi}\frac{\mu e_0\pa_t\Delta
H_{\eell}\circ\Phi_{e_0}\left\{s,\GG_{e_0}(I,t)\right\}}{-1+\mu\pa_G\Delta
H_\ccirc\circ \Phi_{e_0}\left\{s,\GG_{e_0}(I,t)\right\}+\mu e_0\pa_g\Delta
H_\eell\circ\Phi_{e_0}\left\{s,\GG_{e_0}(I,t)\right\}}ds.
\end{split}
\]
From the expansions of the Hamiltonian $\Delta H_{\eell}$ (Corollary
\ref{coro:ExpansioHamiltonian}), of the flow $\Phi_{e_0}$ (Lemma
\ref{lemma:FLowExpansion}) and of the function $\GG_{e_0}$ just obtained, we
deduce
\[
 \FF_{e_0}^{\inn,I}(I,t)=-e_0\int_0^{-14\pi}\frac{\mu\pa_t \Delta
H^1_{\eell}\circ\Phi_{0}\left\{s,\GG_{0}(I,t)\right\}}{-1+\mu\pa_G\Delta
H_\ccirc\circ \Phi_{0}\left\{s,\GG_{0}(I,t)\right\}}ds+\OO\left(e_0^2\right).
\]
That is,
\[
 A_1(I,t)=-\int_0^{-14\pi}\frac{\mu \pa_t\Delta
H^1_{\eell}\circ\Phi_{0}\left\{s,\GG_{0}(I,t)\right\}}{-1+\mu\pa_G\Delta
H_\ccirc\circ \Phi_{0}\left\{s,\GG_{0}(I,t)\right\}}ds.
\]
To deduce the formulas for $A_1^\pm$ it is enough to split $\Delta H_\eell^1$ as
 \[
  \Delta H_{\eell}^1(L,\ell,G,g,t)=\Delta H_{\eell}^{1,+}(L,\ell,G,g)e^{it}+\Delta H_{\eell}^{1,\pm}(L,\ell,G,g)e^{-it},
  \]
and recall that, by formulas \eqref{def:Flow:Circular} and \eqref{eq:Flow:Circ:Time}, $\Phi_0$ can be written as
\[
 \Phi_0\left\{s, (L,\ell,G,g,I,t)\right\}=\left(\Phi_\ccirc\left\{s, (L,\ell,G,g,I)\right\},t+\wt\Phi_0\left\{s, (L,\ell,G,g,I)\right\}\right).
\]
\end{proof}

\subsection{The outer map}\label{sec:Ell:Outer}

This section is devoted to studying the outer maps
\begin{equation}\label{def:OuterMap:Elliptic:0}
  \FF_{e_0}^{\out,\ast}:\wt\Lambda^3_{e_0}\longrightarrow \wt\Lambda^3_{e_0},\,\,\,\ast=\ff,\bb
\end{equation}
for $e_0>0$.

Theorem \ref{th:Elliptic:NHIM} in Section~\ref{sec:Ell:NHIM} proves the existence of
$\Gamma^\ast_{e_0}$ for $\ast=\ff,\bb $, transversal intersections between the
invariant manifolds of $\wt\Lambda^3_{e_0}$ and $\wt\Lambda^4_{e_0}$.
We proceed as in Section \ref{sec:Circular:Outer} to define
the outer map $\FF_{e_0}^\out$.  We study it as a
perturbation of the outer map of the circular problem given in
\eqref{def:OuterMap:Circular}, using Poincar{\'e}-Melnikov
techniques. As explained in Section~\ref{sec:Circular:Outer},
the original flow associated to the Hamiltonian
\eqref{def:HamDelaunayNonRot} does not allow us to study
perturbatively $\FF_{e_0}^\out$. Instead, we use the reduced elliptic
problem defined in \eqref{def:Reduced:ODE}.

The results stated in Theorem \ref{th:InnerAndOuter:Elliptic} about the outer
map follow from the next lemma. The lemma also shows how to compute the first
order term of the outer map. We use the same notation as in Section~\ref{sec:Circular:Outer}. In particular, we use the
trajectories of the circular problem $\gamma_I^{\ff,\bb}(\sigma)$ and
$\la_I^{3,4}(\sigma)$ defined in \eqref{def:OrbitsForOuterMap}, and we define their corresponding $t$-component of the flow as
\begin{equation}\label{def:OrbitsForOuterMap:t}
\begin{split}
\wt\gamma_I^\ast(\sigma) &=
\wt\Phi_0\{\sigma,(\CCC_0^{\ast,L}(I),\CCC_0^{\ast,\ell}(I),\CCC_0^{\ast,G}(I),0)\},\,\,\,\ast=\ff,\bb\\
\wt\la_I^j(\sigma) &=
\wt\Phi_0\{\sigma,(\GG_0^{j,L}(I),\GG_0^{j,\ell}(I),\GG_0^{j,G}(I),0)\},\,\,\, j=3,4\\
\end{split}
\end{equation}
where $\wt\Phi_0$ is defined in \eqref{eq:Flow:Circ:Time} and $\CCC_0^\ast$
and $\GG_0^j$ are given in Corollary \ref{coro:NHIMCircular:Poincare}.


Recall that 
\[
\Delta H_\eell^{1,\pm}(\ell,L,g,G,t)=\Delta
H_\eell^{1,\pm}(\ell,L,g,G)e^{it}+\Delta
H_\eell^{1,\pm}(\ell,L,g,G)e^{-it}, 
\]
as defined in Corollary \ref{coro:ExpansioHamiltonian}, and that
the functions $\omega_\pm^\ast$ are defined in
\eqref{def:Omega0PlusMinus}.

\begin{lemma}\label{lemma:Outer:Elliptic}
  Assume Ansatz~\ref{ans:NHIMCircular}. The outer maps
$\FF_{e_0}^{\out,\ast}$ have the following expansion with
  respect to $e_0$:
  \begin{equation}\label{def:OuterMap:Elliptic}
    \FF_{e_0}^{\out,\ast}:\left(\begin{array}{c} I\\
        t
      \end{array}\right)\mapsto \left(\begin{array}{c} I+ e_0\left(B^{\ast,+} (I)e^{it}+B^{\ast,-} (I)e^{-it}\right) +\OO\left(e_0^2\right)\\
        t+\mu\omega^{\ast}(I)+\OO(e_0)
      \end{array}\right),\,\,\,\ast=\ff,\bb.
  \end{equation}
  The functions $B^{\ast,\pm}(I)$ are defined as
\begin{equation}\label{def:Omega:PlusMinus}
\begin{split}
B^{\ff,\pm}(I)&=B_\out^{\ff,\pm}(I)+B_\inn^{\ff,\pm}(I)e^{\pm i\mu\omega_\out^\ff(I)}\\
B^{\bb,\pm}(I)&=B_\inn^{\bb,\pm}(I)+B_\out^{\bb,\pm}(I)e^{\pm i\mu\omega_\inn^\bb(I)},
\end{split}
\end{equation}
where $\omega_\out^\ff(I)$ and $\omega_\inn^\bb(I)$ are the functions defined in \eqref{def:Omega0:OuterPart} and \eqref{def:Omega0:InnerPart} respectively and
  \begin{align}\label{def:Omega:PlusMinus:Out:for}
    B^{\ff,\pm}_\out(I)=&\pm i\mu\lim_{T\rightarrow+\infty}\int_0^T\left(\frac{\Delta H^{1,\pm}_\eell\circ\gamma_I^\ff(\sigma)}{-1+\mu\pa_G\Delta H_\ccirc\circ\gamma_I^\ff(\sigma)}e^{\pm i\wt\gamma_I^\ff(\sigma)}\right.\nonumber\\
    &\qquad\qquad\qquad\left.-\frac{\Delta H_\eell^{1,\pm}\circ\la_I^3(\sigma)}{-1+\mu\pa_G\Delta H_\ccirc\circ\la_I^3(\sigma)}e^{\pm i\left(\wt\la_I^3(\sigma)+\mu\omega_+^\ff(I)\right)}\right) d\sigma\\
    &\mp i\mu\lim_{T\rightarrow-\infty}\int_0^T\left(\frac{\Delta H_\eell^{1,\pm}\circ\gamma_I^\ff(\sigma)}{-1+\mu\pa_G\Delta H_\ccirc\circ\gamma_I^\ff(\sigma)}e^{\pm i\wt\gamma_I^\ff(\sigma)}\right.\nonumber\\
    &\qquad\qquad\qquad\left.-\frac{\Delta H_\eell^{1,\pm}\circ\la_I^4(\sigma)}{-1+\mu\pa_G\Delta H_\ccirc\circ\la_I^4(\sigma)}e^{\pm i\left(\wt\la_I^4(\sigma)+\mu\omega_-^\ff(I)\right)}\right)d\sigma,\nonumber
\end{align}
\begin{align}\label{def:Omega:PlusMinus:back}
  B_\out^{\bb,\pm}(I)=&\pm i\mu\lim_{T\rightarrow+\infty}\int_0^T\left(\frac{ \Delta H^{1,\pm}_\eell\circ\gamma_I^\bb(\sigma)}{-1+\mu\pa_G\Delta H_\ccirc\circ\gamma_I^\bb(\sigma)}e^{\pm i\wt\gamma_I^\bb(\sigma)}\right.\nonumber\\
    &\qquad\qquad\qquad-\left.\frac{\Delta H_\eell^{1,\pm}\circ\la_I^4(\sigma)}{-1+\mu\pa_G\Delta H_\ccirc\circ\la_I^4(\sigma)}e^{\pm i\left(\wt\la_I^4(\sigma)+\mu\omega_+^\bb(I)\right)}\right)d\sigma\\
    &\mp i\mu\lim_{T\rightarrow-\infty}\int_0^T\left(\frac{ \Delta H_\eell^{1,\pm}\circ\gamma_I^\bb(\sigma)}{-1+\mu\pa_G\Delta H_\ccirc\circ\gamma_I^\bb(\sigma)}e^{\pm i\wt\gamma_I^\bb(\sigma)}\right.\nonumber\\
    &\qquad\qquad\qquad-\left.\frac{ \Delta H_\eell^{1,\pm}\circ\la_I^3(\sigma)}{-1+\mu\pa_G\Delta H_\ccirc\circ\la_I^3(\sigma)}e^{\pm i\left(\wt\la_I^3(\sigma)+\mu\omega_-^\bb(I)\right)}\right)d\sigma,\nonumber
\end{align}
\begin{align}\label{def:Omega:PlusMinus:Inn}
  B_\inn^{\ff,\pm}(I)=&\mp i\mu\int_0^{-12\pi}\frac{\Delta H^{1,\pm}_\eell\circ\la_I^4(\sigma)}{-1+\mu\pa_G\Delta H_\ccirc\circ\la_I^4(\sigma)}e^{\pm i\wt\la_I^4(\sigma)}d\sigma\\
B_\inn^{\bb,\pm}(I)=&\mp\int_0^{-2\pi}\frac{\Delta
  H^{1,\pm}_\eell\circ\la_I^3(\sigma)}{-1+\mu\pa_G\Delta
  H_\ccirc\circ\la_I^3(\sigma)}e^{\pm i\wt\la_I^3(\sigma)}d\sigma. \nonumber
   \end{align}
 \end{lemma}
\begin{proof}
  Recall that the outer maps are the composition of
  two maps. Indeed, as explained in Section
  \ref{sec:Circular:Outer}, they are defined as
\[
 \begin{split}
  \FF_{e_0}^{\out,\ff}=\PP_{e_0}^6\circ\SSS_{e_0}^\ff: \wt\Lambda_0^3\longrightarrow \wt\Lambda_0^3\\
  \FF_{e_0}^{\out,\bb}=\SSS_{e_0}^\bb\circ\PP_{e_0}: \wt\Lambda_0^3\longrightarrow \wt\Lambda_0^3.
 \end{split}
\]
Thus, we study  both maps perturbatively and then their composition leads to
the proof of the lemma. We only deal with  $\FF_{e_0}^{\out,\ff}$ since the
proof for $\FF_{e_0}^{\out,\bb}$ is analogous.

To study $\SSS_{e_0}^\ff:\wt\Lambda_0^3\longrightarrow \wt\Lambda_0^4$  we use the Definition \ref{definition:OuterMap} of the (heteroclinic) outer map. Let us consider points $z\in \Gamma^\ast_{e_0}$, $x_+\in \wt\Lambda^4_{e_0}$ and $x_-\in \wt\Lambda^3_{e_0}$ such that
  \[
  \mathrm{dist}\left(\PP^n_{e_0}(z),\PP^n_{e_0}(x_\pm)\right)<C\lambda^{-|n|}\qquad\text{for }n\in\ZZ^\pm
  \]
  for certain constants $C>0$ and $\lambda>1$. Using the parameterizations of $\Gamma^\ff_{e_0}$ and $\wt\Lambda^j_{e_0}$, $j=3,4$, given in Theorem \ref{th:Elliptic:NHIM}, we write the points $z$ and $x_\pm$ in coordinates as $z=\CCC_{e_0}(I_0,t_0)$, $x_+=\GG^4_{e_0}(I_+,t_+)$ and $x_-=\GG^3_{e_0}(I_-,t_-)$. Then, the $I$-component of the outer map is just given by 
  \[
  \FF_{e_0}^{\out,I} (I_-,t_-)=I_+=I_-+(I_+-I_-).
  \]
  To measure $I_+-I_-$ we first deal with $I_0-I_\pm$. Consider the flow $\Phi_{e_0}$
  associated to the reduced elliptic problem \eqref{def:Reduced:ODE}.
Applying the
  Fundamental Theorem of Calculus,
  \[
   \begin{split}
  I_0-I_+=\lim_{T\rightarrow-\infty}\int_T^0\left(\frac{d}{ds}\Phi_{e_0}\left\{s,\CCC^\ff_{e_0}(I_0,t_0)\right\}-\frac{d}{ds}\Phi_{e_0}\left\{s,\GG^4_{e_0}(I_+,t_+)\right\}\right)ds\\
  I_0-I_-=\lim_{T\rightarrow+\infty}\int_T^0\left(\frac{d}{ds}\Phi_{e_0}\left\{s,\CCC^\ff_{e_0}(I_0,t_0)\right\}-\frac{d}{ds}\Phi_{e_0}\left\{s,\GG^3_{e_0}(I_-,t_-)\right\}\right)ds.
\end{split}
\]
  Note that the change of sign in the limit of integration comes from the fact that system \eqref{def:Reduced:ODE} has the time reversed.

  Using the perturbative expansions of $\CCC^\ff_{e_0}$ and $\Lambda^j_{e_0}$
given in Theorem \ref{th:Elliptic:NHIM}, equation \eqref{def:Reduced:ODE},
the perturbative expansion of the Hamiltonian \eqref{def:HamDelaunayNonRot}
given in Corollary \ref{coro:ExpansioHamiltonian} and the perturbation of the
flow $\Phi_{e_0}$ given in Lemma \ref{lemma:FLowExpansion}, we see that
  \[
  \begin{split}
    I_0-I_+&=-e_0\lim_{T\rightarrow-\infty}\int_T^0\Bigg(\left.\frac{\mu\pa_t \Delta H^1_\eell(L,\ell,G,g,t)}{-1+\mu\pa_G\Delta H_\ccirc(L,\ell,G,g) }\right|_{(L,\ell,G,g,t)=(\Phi_0^\ccirc,\Phi_0^t)\left\{s,\CCC^\ff_{0}(I_0,t_0)\right\}}\\
    &\qquad\qquad\qquad-\left.\frac{\mu\pa_t \Delta H^1_\eell(L,\ell,G,g,t)}{-1+\mu\pa_G\Delta H_\ccirc(L,\ell,G,g) }\right|_{(L,\ell,G,g,t)=(\Phi_0^\ccirc,\Phi_0^t)\left\{s,\GG^4_{0}\left(I_+,t_+\right)\right\}}\Bigg)ds+\OO\left(e_0^2\right)\\
 I_0-I_-&=-e_0\lim_{T\rightarrow+\infty}\int_T^0\Bigg(\left.\frac{\mu\pa_t \Delta H^1_\eell(L,\ell,G,g,t)}{-1+\mu\pa_G\Delta H_\ccirc(L,\ell,G,g) }\right|_{(L,\ell,G,g,t)=(\Phi_0^\ccirc,\Phi_0^t)\left\{s,\CCC^\ff_{0}(I_0,t_0)\right\}}\\
    &\qquad\qquad\qquad-\left.\frac{\mu\pa_t \Delta H^1_\eell(L,\ell,G,g,t)}{-1+\mu\pa_G\Delta H_\ccirc(L,\ell,G,g) }\right|_{(L,\ell,G,g,t)=(\Phi_0^\ccirc,\Phi_0^t)\left\{s,\GG^3_{0}\left(I_-,t_-\right)\right\}}\Bigg)ds+\OO\left(e_0^2\right),
  \end{split}
  \]
  where $\Phi_0^\ccirc$ and $\Phi^t_0$ are defined in \eqref{def:Flow:Circular} 
  and \eqref{eq:Flow:Circ:Time} respectively.

  Taking into account 
  that $\Delta H_\eell^1$ satisfies that $\NNN\left(\Delta
H_\eell^1\right)=\{\pm 1\}$ (Corollary \ref{coro:ExpansioHamiltonian}), 
  one can easily obtain the formula for $B_\out^{\ff,\pm}$ in \eqref{def:Omega:PlusMinus:Out:for}.

To obtain the formula for $B_\inn^{\ff,\pm}$ we proceed as in the study of 
the inner map in Section \ref{sec:Ell:NHIM}. Finally, to obtain the formula for 
$B^{\ff,\pm}$ it is enough to compose both maps $\PP_{e_0}^6$ and $\SSS_{e_0}^\ff$.
\end{proof}

\section{Existence of diffusing orbits}\label{sec:ProofDiffusion}

\subsection{Existence of a transition chain of whiskered tori}
The numerics of Appendix \ref{app:Comparison} support the following
ansatz, which is crucial to obtain the main theorem of this section, Theorem \ref{th:Transition}. The dynamical significance of this ansatz appears  in the averaging lemma~\ref{lemma:Averaging}, which is one of the steps in the proof of the theorem.

\begin{ansatz}\label{ans:B}
  The functions of $I$
  \[
  \wt B^{\ast,\pm} \left(I\right)=B^{\ast,\pm}
  \left(I\right)-\frac{e^{\pm i\mu\omega^\ast(I)}-1}{e^{\pm
      i\mu\TTT_0(I)}-1}A_1^\pm \left(I\right) 
  \]
  do not vanish over the domains $\DD^\ast$, $\ast=\ff,\bb$ (defined in
  Corollary~\ref{coro:NHIMCircular:Poincare}).
\end{ansatz}

Next is the main result of this section.

\begin{theorem} \label{th:Transition} Assume
  Ans{\"a}tze~\ref{ans:NHIMCircular}, \ref{ans:TwistInner}
  and~\ref{ans:B}. For every $\delta>0$ there exists $e_0^*>0$ and $C>0$
  such that for every $0<e_0<e_0^*$ the map $\PP_{e_0}$ in
  \eqref{def:Elliptic:Poincare} has a collection of invariant
  $1$-dimensional tori $\{\TT_i\}_{i=1}^N\subset \tilde
  \Lambda_{e_0}$ such that

  \begin{itemize}
  \item $\TT_1\cap \{I=I_-+\de\}\neq \emptyset$ and $\TT_N\cap \{I=I_+-\de\}\neq \emptyset$.

  \item Hausdorff $\mathrm{dist}(\TT_i,\TT_{i+1})< C e_0^{3/2}$.

  \item These tori form a  transition chain. Namely,
    $ W^u_{\TT_i}\pitchfork W^s_{\TT_{i+1}}\neq\emptyset$ for each $i=1,\dots, N-1$.
  \end{itemize}
\end{theorem}

\begin{proof}[Proof of Theorem \ref{th:Transition}] Once we have computed the first orders in $e_0$ of both
  the outer and the inner map, we want to understand their properties and
  compare their dynamics. To make this comparison we perform two
  steps of averaging \cite{ArnoldKN88}. This change of coordinates
  straightens the $I$-component of the inner map at order
  $\OO(e_0^3)$ in such a way that, in the new system of coordinates,
  the dynamics of both maps is easier to compare. Nevertheless,
  before averaging, we have to perform a preliminary change
  of coordinates to straighten the symplectic form $\Omega_{e_0}^3$ to
  deal with the canonical form $dI\wedge dt$.

\begin{lemma}\label{lemma:StraighteningOmega}
  Assume Ansatz~\ref{ans:NHIMCircular}. There exists an $e_0$-close to
  the identity change of variables
  \begin{equation}\label{def:SymplecticFormChange}
    (I,t)=\left(I',t'\right)+e_0\varphi_1 \left(I',t'\right),
  \end{equation}
  defined on $\wt \Lambda^3_{e_0}$, which transforms the symplectic 
  form $\Omega^3_{e_0}$ defined in \eqref{def:InnerDiffForm:Ell} into the canonical form
  \[
  \Omega_0=dI'\wedge dt'.
  \]
  In the new coordinates,
  \begin{itemize}
  \item The inner map $\FF_{e_0}^\inn$
    in \eqref{def:InnerMap:ell:th} reads
    \begin{equation}\label{def:InnerMap:ell:straightening}
      \FF_{e_0}^{\inn'}:\left(\begin{array}{c} I'\\
          t'
        \end{array}\right)\mapsto \left(\begin{array}{c}  I'+e_0 A_1\left(I',t'\right)+e_0^2 A_2'\left(I',t'\right)+\OO\left(\mu e_0^3\right)\\
          t'+\mu\TTT_0\left(I'\right)+e_0 \TTT_1'\left(I',t'\right)+e_0^2 \TTT_2'\left(I',t'\right)+\OO\left(\mu e_0^3\right)
        \end{array}\right)
    \end{equation}
    where $A_1$ is the function given in Lemma \ref{lemma:InnerMap:Elliptic} and $A_2'$, $\TTT_1'$ and $\TTT_2'$ satisfy
    \[
    \NNN\left(A_2'\right)=\{0,\pm 1,\pm 2\},\,\,\,\NNN\left(\TTT_1'\right)=\{\pm 1\}\,\,\text{ and }\,\,\NNN\left(\TTT_2'\right)=\{0,\pm 1,\pm 2\}.
    \]
  \item The outer maps $\FF_{e_0}^{\out,\ff}$ and $\FF_{e_0}^{\out,\bb}$ in \eqref{def:OuterMap:Elliptic:th} read
    \begin{equation}\label{def:OuterMap:ell:straightening}
      \FF_{e_0}^{\out,\ast'}:\left(\begin{array}{c} I'\\
          t'
        \end{array}\right)\mapsto \left(\begin{array}{c}  I'+e_0 B^\ast \left(I', t'\right)+\OO\left(\mu e_0^2\right)\\
          t'+\mu\omega^\ast(I')+\OO(\mu e_0)
        \end{array}\right),\,\,\,\ast=\ff,\bb,
    \end{equation}
    where $B^\ast$ are the functions given in Lemma \ref{lemma:Outer:Elliptic}.
  \end{itemize}
\end{lemma}
\begin{proof}
  We show that there exists a change of coordinates of the form
  \begin{equation}\label{eq:StraighteningOmega:Change}
    \left\{\begin{split}
        I&= I'+e_0^2 f_2\left(I',t'\right)+\OO\left(e_0^3\right)\\
        t&= t'+e_0 g_1\left(I',t'\right)+e_0^2 g_2\left(I',t'\right)+\OO\left(e_0^3\right)
      \end{split}\right.
  \end{equation}
  with
  \begin{equation}\label{eq:StraighteningOmega:Change:harmonics}
    \NNN(g_1)=\{\pm 1\},\,\,\, \NNN(g_2)=\{0,\pm 1,\pm 2\}\,\,\,\text{ and }\,\,\,\NNN(f_2)=\{0,\pm 1,\pm 2\},
  \end{equation}
  which straightens the symplectic form $\Omega_{e_0}^{3}$. In fact, we look
for the inverse change. Namely, we look for a change of coordinates of the form
  \begin{equation}\label{eq:StraighteningOmega:Invers}
    \left\{\begin{split}
        I'&= I+e_0^2 \wt f_2\left(I,t\right)+e_0^3\wt f_\geq (I,t)\\
        t'&= t+e_0 \wt g_1\left(I,t\right)
      \end{split}\right.
  \end{equation}
  such that the pullback of $\Omega_0=dI'\wedge dt'$ with respect to this
change is the symplectic form $\Omega_{e_0}^{3}$. Even though we do not write it explicitly, $\wt f_\geq$ depends on $e_0$. 
  To obtain this change, it is enough to solve the equations
  \[
    \pa_t \wt g_1=a_1^3,\qquad 
    \pa_I \wt f_2=a_2^3,\qquad 
    \pa_I \wt f_\geq =b,
  \]
  where
  \[
  b=a_{\geq}^3-\pa_t \wt g_1 \pa_I\wt f_2-e_0 \pa_t\wt g_1\pa_I\wt f_\geq+\pa_I\wt g_1 \pa_t \wt f_2+e_0\pa_I \wt g_I\pa_t \wt f_\geq
  \]
  and $a_1^3$, $a_2^3$ and $a_\geq^3$ are the functions introduced in
  \eqref{def:InnerDiffForm:Ell}. These equations can be solved
  iteratively. 

Recall that by Corollary \ref{coro:SymplecticForm} we have $\NNN(a_1^3)=\{\pm
1\}$. Then, we take $\wt g_1$ as the primitive of $a_1^3$ with zero average, which satisfies
\begin{equation}\label{eq:Harmonics:change:g1}
 \NNN\left(\wt g_1\right)=\{\pm 1\}.
\end{equation}
The other equations can be solved taking
\[
 \wt f_2(I,t)=\int_0^I a_2^3(J,t)dJ \qquad \wt f_\geq(I,t)=\int_0^I b(J,t)dJ.
\]
Note that $b$ depends on $\wt g_1$ and $\wt f_2$, which have been already obtained.
Since by Corollary \ref{coro:SymplecticForm} we have $\NNN(a_2^3)=\{0,\pm 1,\pm 2\}$, one can deduce that
\begin{equation}\label{eq:Harmonics:change:f2}
  \NNN\left(\wt f_2\right)=\{0,\pm 1,\pm 2\}.
\end{equation}
To obtain the change \eqref{eq:StraighteningOmega:Change} it is enough to invert the change \eqref{eq:StraighteningOmega:Invers}. Then,  formulas \eqref{eq:Harmonics:change:g1} and \eqref{eq:Harmonics:change:f2} imply \eqref{eq:StraighteningOmega:Change:harmonics}.

To finish the proof of the lemma it remains to check the properties of the
inner and outer maps in the new coordinates. They follow
from~\eqref{eq:StraighteningOmega:Change:harmonics}. 
\end{proof}
Once we have straightened the symplectic form, we perform two
steps of averaging of the inner map. 

\begin{lemma}\label{lemma:Averaging}
  Assume Ans{\"a}tze~\ref{ans:NHIMCircular} and \ref{ans:TwistInner}. There
exists a symplectic change
  of variables $e_0$-close to the identity,
  \begin{equation}\label{def:AveragingChange}
    \left(I',t'\right)=(\II,\tau)+e_0\varphi_2 (\II,\tau),
  \end{equation}
  defined on $\wt \Lambda^3_{e_0}$, that:
  \begin{itemize}
  \item Transforms the inner map $\FF_{e_0}^{\inn'}$ in
    \eqref{def:InnerMap:ell:straightening} into 
    \begin{equation}\label{def:InnerMap:ell:modified}
      \wt\FF_{e_0}^\inn:\left(\begin{array}{c} \II\\
          \tau
        \end{array}\right)\mapsto \left(\begin{array}{c}  \II+\OO\left(\mu e_0^3\right)\\
          \tau+\mu\TTT_0\left(\II\right)+e_0^2 \wt \TTT_2\left(\II\right)+\OO\left(\mu e_0^3\right)
        \end{array}\right).
    \end{equation}
  \item Transforms the outer maps $\FF_{e_0}^{\out,\ff'}$ and $\FF_{e_0}^{\out,\bb'}$ in \eqref{def:OuterMap:ell:straightening} into
    \begin{equation}\label{def:OuterMap:ell:modified}
      \wt\FF_{e_0}^{\out,\ast}:\left(\begin{array}{c} \II\\
          \tau
        \end{array}\right)\mapsto \left(\begin{array}{c}  \II+e_0 \wt B^\ast (\II, \tau)+\OO\left(\mu e_0^2\right)\\
          \tau+\mu\omega^\ast(\II)+\OO(\mu e_0)
        \end{array}\right),\,\,\,\ast=\ff,\bb,
    \end{equation}
    where
    \[
    \wt B^\ast \left(\II, \tau\right)=\wt B^{\ast,+} \left(\II\right) e^{i\tau}+\wt B^{\ast,-} \left(\II\right) e^{-i\tau}
    \]
    with
    \[
    \wt B^{\ast,\pm} \left(\II\right)=B^{\ast,\pm} \left(\II\right)-\frac{e^{\pm i\mu\omega^\ast(\II)}-1}{e^{\pm i\mu\TTT_0(\II)}-1}A_1^\pm \left(\II\right).
    \]
  \end{itemize}
\end{lemma}

With the functions introduced in this lemma, Ansatz~\ref{ans:B} can
be restated as $\wt B^{\ast,\pm} \left(\II\right)\neq
0$ over the domains $\DD^\ast$.

Note that we can do two steps of averaging globally in the whole
cylinder $\wt\Lambda_{e_0}$ due to the absence of resonances in the
first orders in $e_0$. Namely, there are no \emph{big gaps}.  This 
contrasts with the typical situation in Arnol'd diffusion (see
e.g. \cite{DelshamsLS06a}).

\begin{proof}
  We perform two steps of (symplectic) averaging. To this end we
  consider a generating function of the form 
  \[
  \SSS(\II, t')=\II t'+e_0\SSS_1(\II, t')+e_0^2 \SSS_2(\II, t'),
  \]
  which defines the change \eqref{def:AveragingChange} implicitly as
  \[
  \begin{split}
    I&=\II+e_0\pa_{t'}\SSS_1(\II, t')+e_0^2 \pa_{t'}\SSS_2(\II, t')\\
    \tau&=t'+e_0\pa_{\II}\SSS_1(\II, t')+e_0^2 \pa_{\II}\SSS_2(\II, t').
  \end{split}
  \]
  By Ansatz \ref{ans:TwistInner} we have  \eqref{eq:IntervalTwist} and by
  Theorem \ref{th:InnerAndOuter:Elliptic} we know the $t'$-harmonics of
  the functions $A_i$ and $\TTT_i$. Then, it follows that the
  functions $\SSS_i$ 
  corresponding to two steps of averaging are globally defined in
  $\wt\Lambda_{e_0}^3$. 
  In these new variables, taking into account that the inner
  map is exact symplectic, 
  one can see that the inner map is of the form
  \eqref{def:InnerMap:ell:modified}. 

  To obtain a perturbative expression for the outer maps
  $\wt\FF_{e_0}^{\out,\ast}$, we need 
  to compute $\SSS_1$ explicitly:
  \[
  \SSS_1(\II, t)=-\frac{iA_1^+(\II)}{e^{i\mu
      \TTT_0(\II)}-1}e^{it}+\frac{iA_1^-(\II)}{e^{-i\mu
      \TTT_0(\II)}-1}e^{-it}. 
  \]
  Applying this change to the outer maps $\FF_{e_0}^{\out,\ast}$ in
  \eqref{def:OuterMap:Elliptic}, we obtain
  \eqref{def:OuterMap:ell:modified}.
\end{proof}

In the new coordinates $(\II,\tau)$ the inner map $\wt\FF_{e_0}^\inn$
in \eqref{def:InnerMap:ell:modified} is a $e_0^3$-close to integrable
map. Moreover, thanks to Ansatz \ref{ans:TwistInner} it is twist.
Therefore we can apply KAM theory to prove the existence of invariant
curves in $\wt\Lambda^3_{e_0}$.  We use a version of the KAM Theorem
from \cite{DelshamsLS00} (see also \cite{Herman83c}).\\

  \noindent {\bf KAM theorem.}\ {\it Let $f:[0,1]\times\TT\rightarrow
    [0,1]\times\TT$ be an exact symplectic $\CCC^l$ map with $l>4$.
    Assume that $f=f_0+\de f_1$, where $f_0(I,\psi)=(I,\psi+A(I))$,
    $A$ is $\CCC^l$, $|\pa_I A|>M$ and $\|f_1\|_{\CCC^l}\leq 1$.
    Then, if $\de^{1/2} M^{-1}=\rho$ is sufficiently small, for a set
    of $\omega$ of Diophantine numbers of exponent $\theta=5/4$, we
    can find invariant tori which are the graph of $\CCC^{l-3}$
    functions $u_\omega$, the motion on them is $\CCC^{l-3}$ conjugate
    to the rotation by $\omega$, and
    $\|u_\omega\|_{\CCC^{l-3}}\leq C\de^{1/2}$. }\\

  Applying this theorem to the map $\wt\FF^\inn_{e_0}$ we obtain the
  KAM tori (see Remark~\ref{rmk:regularity} for the matter of
  their regularity). Moreover, this theorem ensures that the distance
  between these tori is no larger than $e_0^{3/2}$.  The results
  of Lemma \ref{lemma:Averaging} and the KAM theorem lead to the existence
  of a transition chain of invariant tori, as explained next.


  The transition chain is obtained comparing the outer and inner
  dynamics. We do this comparison in the coordinates $(\II,\tau)$
  given by Lemma \ref{lemma:Averaging} and thus we deal with the
  maps $\wt\FF_{e_0}^\inn$ and $\wt\FF_{e_0}^{\out,\ast}$ in
  \eqref{def:InnerMap:ell:modified} and
  \eqref{def:OuterMap:ell:modified} respectively.

  The KAM theorem ensures that there exists a torus $\TT_1$ such that
  $\TT_1\cap \{I=I_--\de\}\neq \emptyset$. Either
  $\wt\FF_{e_0}^{\out,\ff}$ or $\wt\FF_{e_0}^{\out,\bb}$ are defined
  for points in $\TT_1$. Assume without loss of generality that
  $\wt\FF_{e_0}^{\out,\ff}$ is defined for points in $\TT_1$. 
  Thanks to Ansatz \ref{ans:B},
  $\FF_{e_0}^{\out,\ff}(\TT_1)$ satisfies
  \[
  \mathrm{dist}\left(\TT_1, \FF_{e_0}^{\out,\ff}(\TT_1)\right)\geq C e_0
  \]
  for a constant $C>0$ independent of $e_0$. Then, the KAM theorem ensures
  that there exists a torus $\TT_2$ such that $\TT_2\cap
  \FF_{e_0}^{\out,\ff}(\TT_1)\neq 0$.  Iterating this procedure,
  choosing at each step either $\wt\FF_{e_0}^{\out,\ff}$ or
  $\wt\FF_{e_0}^{\out,\bb}$, we obtain the transition chain.
This completes the proof of Theorem~\ref{th:Transition}.
\end{proof}


\subsection{Shadowing}
To finish the proof of Theorem \ref{th:MainTheorem:detail} it remains to prove
the existence of a diffusing orbit using a Lambda Lemma. The study of the Lambda lemma, often also called Inclination Lemma, was initiated in the seminal work of Arnol'd \cite{Arnold64}. In the past decades there have been several works proving analogous results in more general settings \cite{ChierchiaG94, Marco96, Cresson97, FontichM98, Sabbagh13}.
Here, we use a version of the Lambda Lemma proven in \cite{FontichM98} (Theorem 7.1 of that paper).

\begin{lemma}\label{lemma:Lambdalemma}
  Let $f$ be a $\CCC^1$ symplectic map in a $2(d+1)$ symplectic
  manifold.  Assume that the map leaves invariant a $\CCC^1$
  $d$-dimensional torus $\TT$ and the motion on the torus is an
  irrational rotation. Let $\Gamma$ be a $d+1$ manifold intersecting
  $W^u_\TT$ transversally. Then,
  \[
  W^s_\TT\subset \ol{\bigcup_{i>0}f^{-i}(\Gamma)}.
  \]
\end{lemma}

An immediate consequence of this lemma is that any finite transtion
chain can be shadowed by a true orbit. The proof of
Theorem~\ref{th:MainTheorem:detail} follows from the following lemma.

\begin{lemma}\label{lemma:shadowing}
Assume   Ans{\"a}tze~\ref{ans:NHIMCircular}, \ref{ans:TwistInner}
  and~\ref{ans:B}. Consider the transition chain of invariant tori $\{\TT_i\}_{i=1}^N$ obtained in Theorem
 \ref{th:Transition} and a sequence of positive numbers $\{\eps_i\}_{i=1}^N$. Then, we can find a point $P=(L_0,\ell_0 , G_0, g_0, I_0)$  and a sequence of times $T_i$ such that 
\[
\Phi(T_i,P)\in B_{\eps_i}(\TT_i),
\]
where $\Phi$ is the flow associated to Hamiltonian \eqref{def:HamDelaunayRot} and $B_{\eps_i}(\TT_i)$ is a neighborhood of size $\eps_i$ of the torus $\TT_i$.
\end{lemma}

\begin{proof}
Consider $P'\in W_{\TT_1}^s$. Then, there exists a ball $B_1$ centered on $P'$ and a time $T_1>0$, such that 
\begin{equation}\label{def:shadowingInclusion}
 \Phi(T_1, B_1)\subset B_{\eps_1}(\TT_1).
\end{equation}
Since  $W_{\TT_1}^u$ and $W_{\TT_2}^s$ intersect transversally, by Lemma
\ref{lemma:Lambdalemma}, we know that $W_{\TT_2}^s\cap B_1\neq\emptyset$.
Thus, there exists a closed ball $B_2\subset B_1$ centered at a point in $W_{\TT_2}^s$ that satisfies  
\[
 \Phi(T_2, B_2)\subset B_{\eps_2}(\TT_2)
\]
for some time $T_2>0$. Hence, proceeding by induction, we obtain a sequence of nested closed balls 
\[
 B_N\subset B_{N-1}\subset\ldots\subset B_1
\]
and a sequence of times $\{T_i\}_{i=1}^N$, such that
\[
 \Phi(T_j, B_i)\subset B_{\eps_i}(\TT_j) \,\,\,\text{ for }i\leq j.
\]
Therefore,  the intersection $\cap_{i=1}^N B_i$ is non-empty and any point
belonging to it shadows the transition chain of tori.
\end{proof}

\bigskip In terms of the elliptical elements of the asteroid, such a
diffusing orbit can be described as follows. The orbit starts near
the resonant cylinder $\Lambda_{e_0}$. The eccentricity of the primaries is
small: this is an essential feature of both the proof above and the
qualitative behavior of the orbit. Over a time interval of length
$\ll 1/e_0$, the orbit closely follows a hyperbolic periodic orbit of
the circular problem. The semi major axis is roughly constant equal to
$7^{2/3}$ and the Jacobi constant to $\Jminus$. The asteroid turns
around the primaries, making one full turn over a time interval of
 $7$ periods of Jupiter. In the frame rotating approximately with the primaries,
the Keplerian ellipse of the asteroid precesses counterclockwise
with fast frequency approximately equal to $-1$; in the inertial frame of reference, it
rotates only $\mu$-slowly (see e.g.~\cite{ArnoldKN88, Fejoz:2002}),
while the eccentricity $e$ slowly oscillates around $e=\eminus$.

At some point (as soon as we can if we want to save time), the orbit
undergoes a heteroclinic excursion, during which a heteroclinic orbit
is shadowed over a time interval of size $\OO(-\ln(\mu e_0)/\sqrt\mu)$.
During this excursion, the semi major axis itself undergoes an
oscillation of magnitude $\OO(\sqrt\mu)$, eventually coming back to its
initial approximate value $7^{2/3}$. On the other hand, the Jacobi
constant and the eccentricity have increased by $\OO(\mu e_0)$.

This process is repeated, and the increments in the eccentricity accumulate
to reach the value $e=\eplus$ in finite time.

\appendix
\section{Numerical study of the normally hyperbolic invariant  cylinder of the circular
  problem. }\label{app:NHIMCircular}

We devote this appendix to the numerical study of the hyperbolic
invariant manifold of the circular problem given in Corollary
\ref{coro:NHIMCircular} and its invariant manifolds. In other words,
we show numerical results which justify the properties of the circular
problem stated in Ansatz~\ref{ans:NHIMCircular}.

Numerical analysis has several sources of error: mainly round-off
errors in computer arithmetics, and approximations of ideal
mathematical objects (e.g. linear approximation of local
stable/unstable manifolds). In our analysis, we have tried to evaluate
such errors, and check that they are appropriately small. We do
\emph{not} claim to give a fully rigorous proof of
Ansatz~\ref{ans:NHIMCircular}, which would require Computer-Assisted
techniques as in \cite{WZ03}. Indeed, we have focused our efforts to
keep the numerics relatively simple and, hopefully, convincing. One
could think of several possible numerical computations to prove our
result. The most numerically demanding one would be to check directly
that some given orbit has an adequate drift of eccentricity. This
computation would not bring much light to the mechanism of
instability, and moreover it would involve formidable numerical
analysis problems, due to the necessarily very long time of
integration. On the contrary, our line of proof allows us to use
numerical verifications involving only orbits of the circular problem
--a dramatic simplification, as we will see below.

Let us make a few more specific comments on the strategy of our
numerical analysis. As mentioned in Section~\ref{Section:Circular},
the circular problem has a conserved quantity, the Jacobi constant
which we denote by $J$ (see \eqref{def:Jacobi}), which corresponds to
energy when the system is expressed in rotating coordinates.  Thus it
is natural to fix the Jacobi constant $J=J_0$ and perform our analysis
for a given $J_0$.  This allows us to reduce the dimension of the
computations by one.  Finally, we let $J$ vary and repeat the
computations for all $J$ in the range of interest $J\in[J_-,J_+]$.

Another important comment is on the choice of coordinates.  For
numerics we prefer Cartesian coordinates, since the equations of
motion are explicit in these coordinates.  Thus we carry out our
computations of the hyperbolic structure of the circular problem in
Cartesian (Appendix~\ref{app:NHIMCircular}).

On the other hand, for perturbative analysis we have used Delaunay
coordinates throughout this paper.
Thus, in Appendix~\ref{sec:NumericalStudyInnerOuter} we explain how
to change coordinates from Cartesian to Delaunay, and
we carry out our computations of the inner and outer maps of the
circular and elliptic problems in Delaunay.

Regarding the integration method, we use a variable-order Taylor method
specially generated to integrate the equations of motion and
variational equations of the circular problem.
The Taylor method has been generated using the ``taylor'' package of
{\`A}.~Jorba and M.~Zou
(see~\url{http://www.maia.ub.es/~angel/taylor/}).
The main advantage of using a Taylor method is that it is very fast
for long-time integrations (without sacrificing accuracy).

\subsection{Computation of the periodic
orbits}\label{sec:computation_periodic_orbits}
Consider the circular problem in rotating Cartesian coordinates
\begin{equation}\label{eq:RTBP:RotCartesian}
 J(x,y,p_x,p_y)=\frac{1}{2}(p_x^2+p_y^2)+yp_x-xp_y-\frac{1-\mu}{r_1}-\frac{\mu}{r_2},
\end{equation}
where
\begin{align*}
r_1^2 &= (x-\mu_2)^2 + y^2, \\
r_2^2 &= (x+\mu_1)^2 + y^2.
\end{align*}
Recall that the energy of the circular problem in rotating
coordinates coincides with the Jacobi constant $J$ in \eqref{def:Jacobi}.
From now on in this appendix we will refer to $J$ as the energy of the system.

We follow the convention to place the large mass (Sun) to the left of
the origin, and the small mass (Jupiter) to the right.
(This is opposite to the astrodynamics convention).
Thus we choose $\mu_1=\mu$ as the small mass, and $\mu_2=1-\mu$ as the
large mass.
Notice that equation~\eqref{eq:RTBP:RotCartesian} is reversible with
respect to the involution
\begin{equation} \label{involutionCartesian}
R(x,p_x,y,p_y) = (x,-p_x,-y,p_y).
\end{equation}
Thus, a solution of the system is symmetric if and only if it
intersects the symmetry axis $\mathrm{Fix}(R) = \{y=0,\ p_x=0\}$.
This symmetry will facilitate our numerical computations. Note that
the involution $R$ is just the involution \eqref{def:involution}
expressed in rotating Cartesian coordinates.

Let the energy be fixed to $J=J_0$.
We look for a resonant periodic orbit $\la_{J_0}$
of~\eqref{eq:RTBP:RotCartesian} in the level of energy $J_0$.
As a first approximation to $\la_{J_0}$, we look for a resonant
periodic orbit of the two-body problem, i.e. of the
Hamiltonian~\eqref{def:HamDelaunayCirc} with $\mu=0$.
Let us denote the approximate periodic orbit by
$\tilde \la_{J_0}=(L,\ell,G,g)$.
The actions $L$ and $G$ are determined by the resonant condition
$L^3=7$ and energy condition $-\frac{1}{2L^2}-G=J_0$.
To determine $\tilde \la_{J_0}$ completely, we choose that the
asteroid is initially at the perihelion, i.e. we impose an initial
condition $\tilde \la_{J_0}(0)=(L^0,\ell^0,G^0,g^0)$ with
$\ell^0=0$ and  $g^0=0$.
Switching to Cartesian coordinates, we obtain an initial condition
$(x^0,p_x^0,y^0,p_y^0)$ with $p_x^0=0$ and $y^0=0$.

Next we refine the trajectory $\tilde \la_{J_0}$ into a true periodic
orbit $\la_{J_0}$ for the system~\eqref{eq:RTBP:RotCartesian} with
$\mu=10^{-3}$.
Consider the Poincar{\'e} section
\[ \Sigma^+ = \{ y=0,\ p_y>0 \} \]
in the circular problem~\eqref{eq:RTBP:RotCartesian}, and let $P\colon\
\Sigma^+\to\Sigma^+$ be the associated Poincar\' e map. Since we are
in rotating coordinates, this section corresponds to collinear configurations
of the three bodies.

\begin{remark}
In numerical integrations, we use a variable-order Taylor method with
local error tolerance $10^{-16}$.
Moreover, a point is considered to be on the Poincar{\'e} section whenever
$|y|<10^{-16}$ and $p_y>0$.
\end{remark}
Furthermore, the momentum variable $p_y$ can be eliminated. Indeed, since $\partial_{p_y} J \neq 0$, $p_y$ in the region of the phase space we deal with, it can be recovered from the
other variables using the energy condition $J(x,p_x,y,p_y)=J_0$.
Hence, the Poincar{\'e} map is a two-dimensional symplectic map
at each energy level, acting only on $(x,p_x)$.

Notice that, in the rotating frame, a 1:7 resonant periodic orbit
makes $6$ turns around the origin. See Figure~\ref{fig:trtbp}.
In principle, we could look for the periodic orbit as a periodic point
$p=(x,p_x)$ of the Poincar{\'e} map: $p=P^6(p)$. This would imply solving
a system of two equations.
Thanks to the reversibility~\eqref{involutionCartesian},
in fact it is only necessary to solve one equation.
Notice that our initial condition $(x,p_x)$ is at the symmetry
section $\{y=0,\ p_x=0\}$, so the periodic orbit must be symmetric.
Thus it is enough to impose the condition that the trajectory
$\la_{J_0}(t)$ after \emph{half} the period is again at the symmetry
section.
Hence we set up the problem as simple one-dimensional root finding:
we look for a point $p=(x,0)$, such that its third iterate $P^3(p)$
has momentum $p_x=0$:
\[ \pi_{p_x}(P^3(p)) = 0. \]
(Here, $\pi_{p_x}\colon \RR^2\to \RR$ is the projection onto the $p_x$
component).

\begin{figure}[h]
\begin{center}
\psfrag{x}{$x$}
\psfrag{y}{$y$}
\psfrag{z}{$p$}
\psfrag{P1}{$P(p)$}
\psfrag{P2}{$P^2(p)$}
\psfrag{P3}{$P^3(p)$}
\includegraphics[width=10cm]{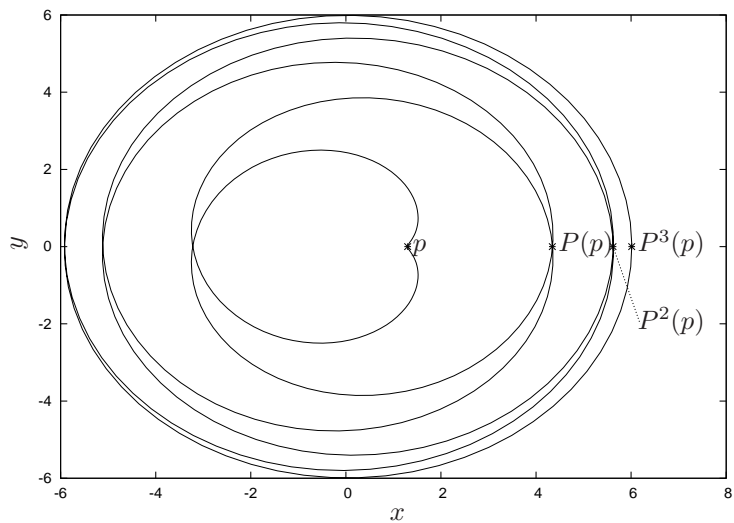}
\end{center}
\caption{Resonant periodic orbit $\lambda_{-1.6}$ of the circular
problem in rotating Cartesian coordinates.}
\label{fig:trtbp}
\end{figure}

In order to solve this problem, we use a a Newton-like method.
Specifically, we use a modified version of Powell's Hybrid method (see
the GSL manual~\cite{GSL} for details) without scaling.
In our computations, the Newton method converges in less than 5
iterations.
As a test of the software, we have checked that the rate of
convergence of the Newton method is quadratic.

\begin{remark}
We ask for an accuracy of $10^{-14}$ in the Newton method,
i.e. a point $p=(x,0)$ is accepted as a root if and only if its third
iterate $P^3(p)$ has momentum $|p_x|<10^{-14}$.
\end{remark}

For the Newton method, we need to compute the derivative of the Poincar{\'e} map.
For each $\xi\in \RR^4$, let $u(t,\xi)$ be the solution of the system with
initial condition $u(0,\xi)=\xi$.
Let $T:\Sigma^+\to \RR$ be the Poincar{\'e} return time.
The derivative of the Poincar{\'e} map at a point $p\in \RR^4$ is given by
the partial derivative $DP(p)=u_\xi(T(p),p)$.
It is well-known that $u_\xi(t,p)$ is the matrix solution of the
variational equation
\[ \dot W = Df(u(t,p))W, \]
where $f$ is the vector field of the circular problem.
We compute $DP(p)$ by numerically integrating the variational
equation using the Taylor method mentioned above.

\begin{figure}[h]
\begin{center}
\psfrag{H}{$J$}
\psfrag{T}{$T_J - 14\pi$}
\psfrag{L}{$L_{\max}$}
\psfrag{T2XXXXXXX}{$T_J - 14\pi$}
\psfrag{L2XXXXXXX}{$L_{\max}$}
\includegraphics[width=10cm]{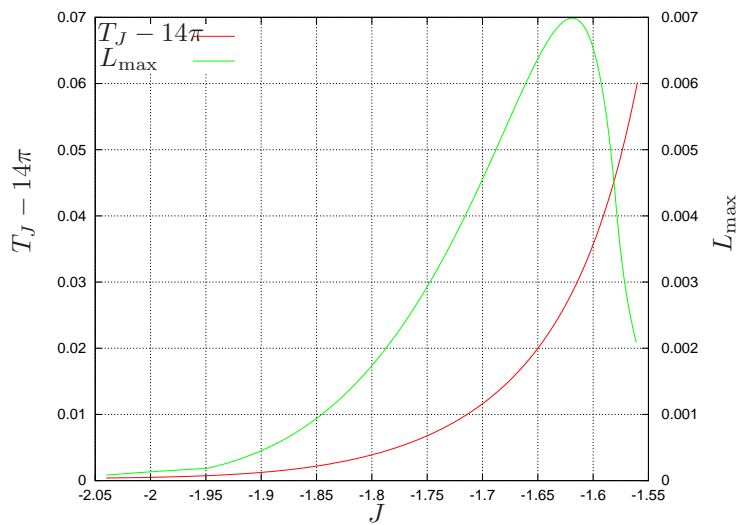}
\end{center}
\caption{Resonant family of periodic orbits.
We show normalized period $T_J - 14\pi$, and maximum deviation of $L$ component with respect to the resonant value $7^{1/3}$ (see equation~\eqref{eq:Ldeviation}).}\label{fig:porbits}
\end{figure}

For illustration, let us show some numerical results corresponding to
the energy value $J=-1.6$.
The first approximation $\tilde \la_{-1.6}$ from the two-body
problem has initial condition $p^0=(x^0,p_x^0)=(1.30253\cdots,0)$.
After refining this initial condition via the Newton method, we obtain
a resonant periodic orbit $\la_{-1.6}$ of the circular problem passing
through the point $p=(x,p_x)=(1.29858\cdots,0)$, with period
$T_{-1.6}=44.01796\cdots \sim 14\pi$. See Figure~\ref{fig:trtbp}.
The periodic orbit $\la_{-1.6}$ is symmetric, with the points $p$
and $P^3(p)$ located at the symmetry section (they have $y=0$ and $p_x=0$).
Notice that, in rotating coordinates, the trajectory of the asteroid
makes 6 turns around the origin before closing up at the point $p$.

Finally, we let $J$ change and, using this procedure, we are able to
obtain the resonant periodic orbit for energy levels
\begin{equation}\label{def:EnergyRange:PeriodicOrbit}
J \in [\bar J_-, \bar J_+] = [-2.04,-1.56].
\end{equation}
See Figure~\ref{fig:porbits}.
This family of resonant periodic orbits constitutes the normally
hyperbolic invariant manifold $\Lambda_0$ given in
Corollary~\ref{coro:NHIMCircular}.
Notice that the period $T_J$ stays close to the resonant period $14\pi$
of the unperturbed system. From Figure~\ref{fig:porbits}, we obtain the bound
\[ |T_J - 14\pi| < 60\mu, \]
which is the first bound given in Ansatz~\ref{ans:NHIMCircular}.

\begin{figure}[h]
\begin{center}
\psfrag{H}{$J$}\psfrag{lu}{$\ln(\lambda)$}
\includegraphics[width=10cm]{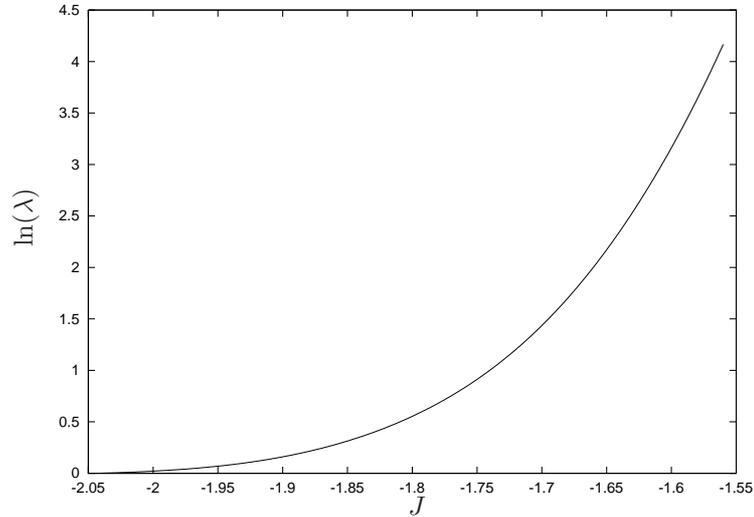}
\end{center}
\caption{Characteristic exponent $\ln(\lambda)$ as a function of energy level $J$ (the other exponent is $-\ln(\lambda)$).}
\label{fig:hypers}
\end{figure}

To determine the stability of the periodic orbit $\la_{J_0}$, we
compute the eigenvalues $\lambda$ and $\lambda^{-1}$ of the matrix
$DP^6(p)$, where $DP^6(p)$ is the linearization of the iterated
Poincar{\'e} map $P^6$ about the fixed point $p$.

Figure~\ref{fig:hypers} shows the characteristic exponents
$\ln(\lambda)$, $\ln(\lambda^{-1})$ as a function of energy.
The family of periodic orbits is strongly hyperbolic as $J \to \bar
J_+$, and weakly hyperbolic as $J\to \bar J_-$.
Note that one would expect that we are in a  nearly integrable regime
since $\mu$ is small. Then one would expect the eigenvalues to be
close to 1.
Nevertheless, in this problem the non-integrability is very noticeable
when one increases $\mu$ to $\mu=10^{-3}$. This is due to the effect of
the perturbing body (Jupiter) on the asteroid, as the asteroid passes
close to it.

Furthermore, we verify that (the square of) the semi-major axis $L$ stays
close to the resonant value $7^{1/3}$.
Integrating the periodic orbit in Delaunay coordinates
$\la_J(t)=(L_J(t),\ell_J(t),G_J(t),g_J(t))$ over one period $T_J$, we compute
the quantity
\begin{equation} \label{eq:Ldeviation}
 L_{\max}(J) = \max_{t \in [0,T_J)} |L_J(t)-7^{1/3}|.
\end{equation}
The function $L_{\max}(J)$ is plotted in Figure~\ref{fig:porbits}.
Notice that we obtain the bound
\[ |L_J(t)-7^{1/3}| < 7\mu \]
for all $t\in \RR$ and $J \in [\bar J_-, \bar J_+]$, which is the second bound given in Ansatz~\ref{ans:NHIMCircular}.

\begin{figure}[h]
\begin{center}
\psfrag{x}{$x$}
\psfrag{y}{$y$}
\includegraphics[width=10cm]{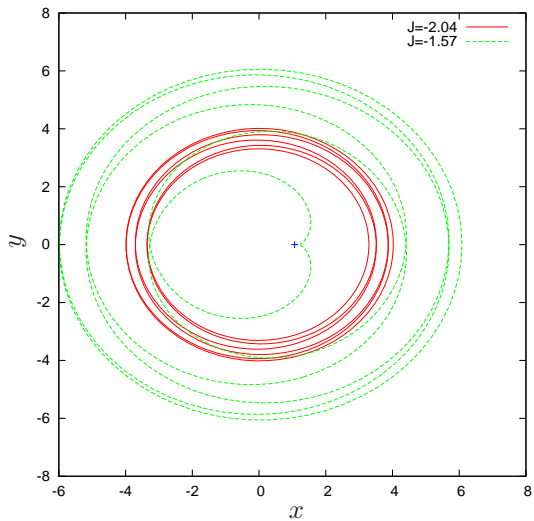}
\end{center}
\caption{Extremal periodic orbits of the family: circular periodic
orbit with $J=\bar J_-$ (in red), elliptical periodic orbit with
$J=\bar J_+$ (in green). The Lagrange equilibrium point $L_2$ is
marked with a '+' symbol.}
\label{fig:trtbp_porbits}
\end{figure}

Let us briefly describe the family of periodic orbits $\la_J$.
For illustration, see Figure~\ref{fig:trtbp_porbits}.
At one endpoint of the family, as $J \to \bar J_-$, the periodic orbit
$\la_J$ tends to a circular orbit of period $14\pi$ centered at the
origin and passing far away from the primaries (Sun and Jupiter).
Moreover, $\la_J$ looses hyperbolicity when $J \to \bar J_-$.
For instance, the periodic orbit $\tilde \la_{\bar J_-}$ of the two-body
problem approximation has eccentricity $e(\bar J_-)=0.09989\cdots$.

At the other endpoint of the family, as $J \to \bar J_+$, the periodic
orbit $\la_J$ tends to a homoclinic loop of the Lagrangian
equilibrium point $L_2$ that makes $6$ turns around the Sun-Jupiter
system.
(In rotating Cartesian coordinates, $L_2$ is located on the $x$ axis at
the point $x_2 \simeq 1.068$).\
This explains the fact that the period $T_J$ ``explodes'' as $J \to
\bar J_+$. Since we are interested in working close to the resonance,
we avoid energies $J > \bar J_+$ where the period explodes.

\subsection{Computation of invariant manifolds}
\label{sec:invariant_manifolds}

In this appendix, we compute the stable and unstable invariant
manifolds associated to the periodic orbits found in the previous
section.

Consider first a fixed energy level $J=J_0$.
Let $\la_{J_0}$ be the resonant periodic orbit of the circular problem
found in the previous section.
To compute the invariant manifolds of the periodic orbit, we continue
using the iterated Poincar{\'e} map.
Thus we look for (one dimensional) invariant manifolds of a hyperbolic
fixed point at each energy level.
Let $p \in \la_{J_0}$ be a hyperbolic fixed point of the iterated
Poincar{\'e} map $\sixmap = P^6$.
Let $\lambda, \ \lambda^{-1}$ be the eigenvalues of $D\sixmap(z)$ with
$\lambda>1$, and $v_u, v_s$ be the associated eigenvectors.

\begin{figure}[h]
\begin{center}
\psfrag{x}{$x$}
\psfrag{px}{$p_x$}
\psfrag{p0}{$p_0$}
\psfrag{p1}{$p_1$}
\psfrag{p2}{$p_2$}
\psfrag{p3}{$p_3$}
\psfrag{p4}{$p_4$}
\psfrag{p5}{$p_5$}
\includegraphics[width=10cm]{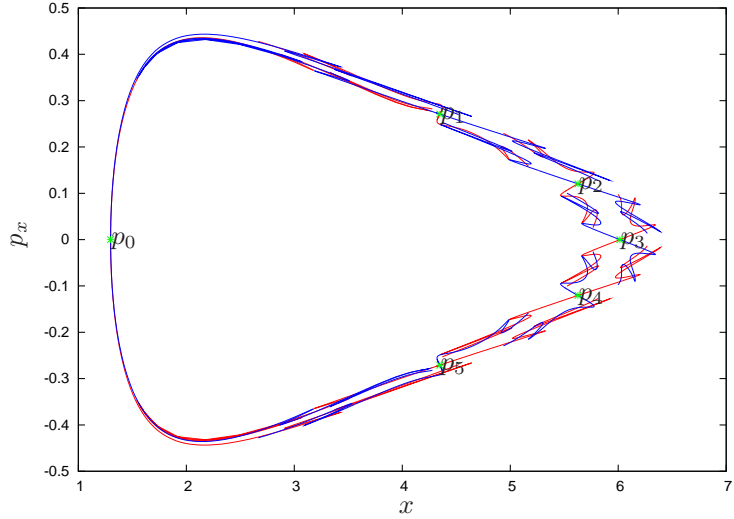}
\end{center}
\caption{Invariant manifolds of the fixed points $p_0, p_1,\dots,p_5$
on the section $\Sigma^+$.
Unstable manifolds are plotted in red, stable in blue.
The fixed points are marked in green.}
\label{fig:invmfld_it}
\end{figure}

Assume that we want to compute the unstable manifold $W^u(p)$.
Let $\eta$ be a small displacement in the unstable direction $v_u$.
We approximate a piece of the local manifold by the linear segment
between the points $p+\eta v_u$ and $\sixmap(p+\eta v_u)$.
We call this segment a \emph{fundamental domain}.
We discretize the fundamental domain into an array of points, and
iterate them by $\sixmap$ to globalize the manifold.
(The stable manifold is computed analogously using the inverse map
$\sixmap^{-1}$.)

The error commited in the local approximation
$ \sixmap(p+\eta v_u) = p + \lambda \eta v_u + \OO(\eta^2) $
of the manifold is given by
\[ \mathrm{err}(\eta) = \left\|\sixmap(p+\eta v_u) - p - \lambda \eta v_u \right\| \in \OO\left(\eta^2\right).
\]

\begin{remark} \label{rem:displacement}
For each energy level $J$, we choose a displacement $\eta=\eta(J)$
such that the local error is $\mathrm{err}(\eta) < 10^{-12}$.
\end{remark}

One can think of $p$ as a fixed point of the iterated Poincar{\'e} map
$\sixmap = P^6$, or as a $6$-periodic point of the Poincar{\'e} map $P$.
If $p_i = P^i(p)$ are the iterates of $p$ for $i=0,\dots,5$, then
$p_i$ are also fixed points of $\sixmap$.
They have associated unstable and stable manifolds, which can be obtained from
$W^{u,s}(p)$ by iteration.

\begin{figure}[h]
\begin{center}
\psfrag{x}{$x$}
\psfrag{px}{$p_x$}
\psfrag{p0}{$p_0$}
\psfrag{p1}{$p_1$}
\psfrag{p2}{$p_2$}
\psfrag{p3}{$p_3$}
\psfrag{p4}{$p_4$}
\psfrag{p5}{$p_5$}
\includegraphics[width=10cm]{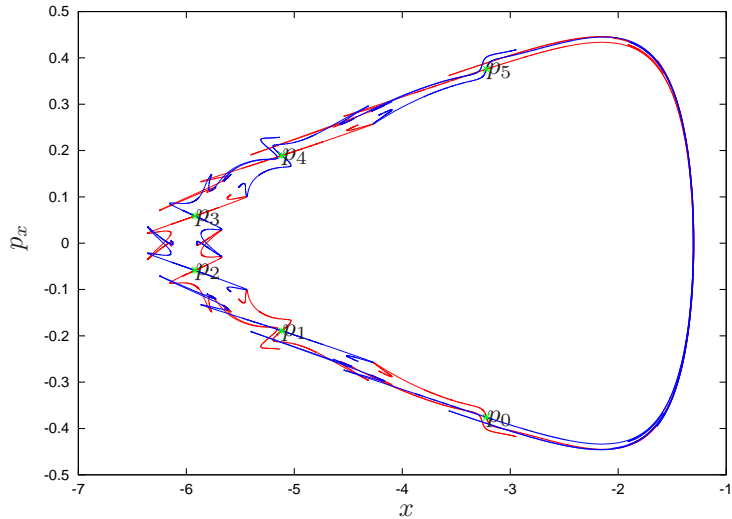}
\end{center}
\caption{Invariant manifolds on the section $\Sigma^-$.}
\label{fig:invmfld2}
\end{figure}

For illustration, let us show some numerical results corresponding to
the energy value $J=-1.6$.
Figure~\ref{fig:invmfld_it} shows the manifolds of all iterates
$\{p_i\}_{i=0,\dots,5}$.
Notice that the dynamics in Figure~\ref{fig:invmfld_it} is reversible
with respect to the symmetry section $\{y=0,\ p_x=0\}$, as discussed
in the previous section (see~\eqref{involutionCartesian}).
Figure~\ref{fig:invmfld_it} shows that the manifolds do intersect
transversally at different homoclinic points.
We are interested in measuring the splitting angle between the
manifolds.
Unfortunately, the homoclinic points do not lie on the symmetry axis, which
would be very useful in order to compute them.

In order to have the homoclinic points lie on the symmetry axis, we
recompute the manifolds on the new Poincar{\'e} section
\[ \Sigma^- = \{y=0,\ p_y<0\}. \]
Numerically, we just transport points on the unstable manifold from
section $\Sigma^+$ to section $\Sigma^-$ by the forward flow, and
points in the stable manifold by the backward flow.
See Figures~\ref{fig:invmfld2} and~\ref{fig:invmfld2_it23}.
Now the points that lie on the symmetry line $p_x=0$ are homoclinic
points.

\begin{figure}[h]
\begin{center}
\psfrag{x}{$x$}
\psfrag{px}{$p_x$}
\psfrag{p2}{$p_2$}
\psfrag{p3}{$p_3$}
\psfrag{Wu1p3}{$W^{u,1}(p_3)$}
\psfrag{Ws2p3}{$W^{s,2}(p_3)$}
\psfrag{Ws1p2}{$W^{s,1}(p_2)$}
\psfrag{Wu2p2}{$W^{u,2}(p_2)$}
\includegraphics[width=10cm]{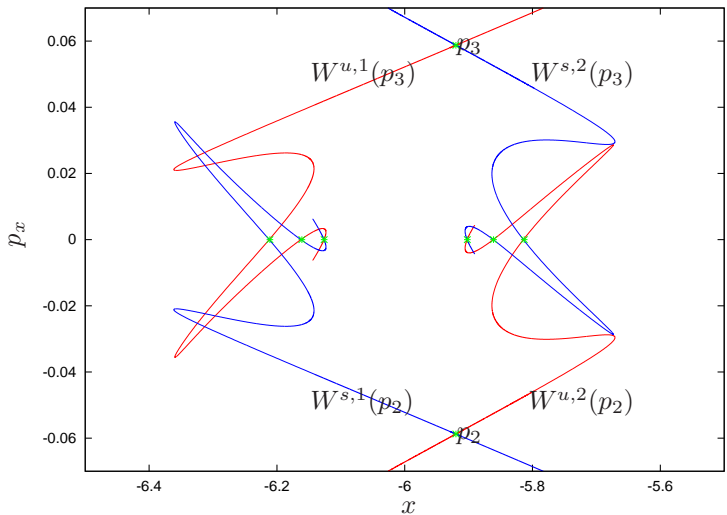}
\end{center}
\caption{Invariant manifolds of the points $p_2$ and $p_3$ on the
section $\Sigma^-$. Due to the symmetry, points that lie on the line
$p_x=0$ (marked in green) are intersection points.}
\label{fig:invmfld2_it23}
\end{figure}

\subsection{Computation of transversal homoclinic points and splitting
angle}
\label{sec:homoclinic_points}

In this appendix, we compute the angle between the invariant manifolds
at one of the  transversal intersections.
We will restrict the range of energy values to
\begin{equation}\label{def:EnergyRange:Splitting}
J\in[J_-,J_+]=[-1.81,-1.56],
\end{equation}
or equivalently the range of eccentricities to
$ e\in[e_-,e_+]=[0.48,0.67]$.
This is the range where we can validate the accuracy of our
computations (see Appendix~\ref{sec:accuracy_computations}).
Below $e_-=0.48$, the splitting size becomes comparable to the
numerical error that we commit in double precision arithmetic.

\begin{remark}\label{rem:maxrange}
In this paper we concentrate on proving the existence of global
instabilities in the Restricted three-body problem; we are not so much
concerned with finding the \emph{maximal} range of eccentricities
along which the asteroid drifts.
Thus we do not investigate the transversality of the splitting below
$e_-$.
However, we are convinced that the maximal range of eccentricities is
larger than $[e_-,e_+]$, in particular that the lower bound can be
pushed well below $e_-$.
We think that our mechanism of instability applies to this larger
range of eccentricities.
In fact, it is possible to study such exponentially small splitting
using more sophisticated numerical methods, such as multiple-precision
arithmetic, and high-order approximation of local invariant manifolds,
see for instance~\cite{FontichSimo1990, DelshamsRamirez1999,
GelfreichSimo2008}.
\end{remark}

\begin{figure}
\begin{center}
\psfrag{x}{$x$}
\psfrag{px}{$p_x$}
\psfrag{p2}{$p_2$}
\psfrag{p3}{$p_3$}
\psfrag{z1}{$z_1$}
\psfrag{z2}{$z_2$}
\psfrag{Wu1p3}{$W^{u,1}(p_3)$}
\psfrag{Ws2p3}{$W^{s,2}(p_3)$}
\psfrag{Ws1p2}{$W^{s,1}(p_2)$}
\psfrag{Wu2p2}{$W^{u,2}(p_2)$}
\includegraphics[width=10cm]{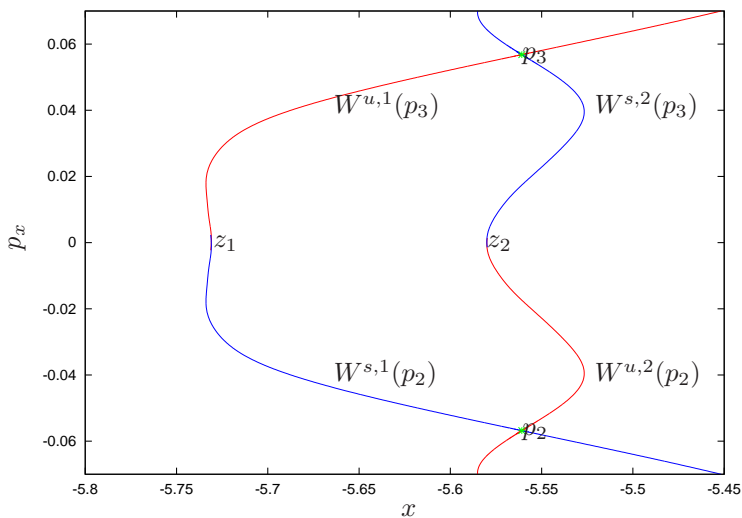}
\end{center}
\caption{Invariant manifolds of the points $p_2$ and $p_3$ for the energy
level $J=-1.74$.}
\label{fig:invmfld2_it23_H174}
\end{figure}

Consider first a fixed energy level $J=J_0$ that is close to the
unperturbed situation, e.g. $J=-1.74$.
The corresponding manifolds are given in
Figure~\ref{fig:invmfld2_it23_H174}.
In general, there are uncountably many intersection points.
For instance, in Figure~\ref{fig:invmfld2_it23} we show six intersections on
the symmetry line.
However, when the perturbation is small, there is one distinguished
intersection point located ``in the middle'' of the homoclinic.
We call it the \emph{primary} intersection point.

Let us compute the primary intersection point $z_1$ corresponding to
the ``outer'' splitting of the manifolds $W^{u,1}(p_3)$ and
$W^{s,1}(p_2)$.
For $J=-1.74$, the \emph{primary} intersection $z_1$ corresponds to
the \emph{first} intersection of the manifolds with the $p_x=0$ line,
as we grow the manifolds from the fix points.
Thanks to the symmetry, it is enough to look for the intersection of
$W^{u,1}(p_3)$ with the $p_x=0$ axis, because $W^{s,1}(p_2)$ must also
intersect the axis at the same point.

To compute the intersection point $z_1$, we continue using a linear
approximation of the local manifold, and propagate a fundamental
domain in the local manifold by iteration.
Let $v_u$ be the unstable eigenvector associated to the point $p_3$.
Consider the fundamental segment $l_u$ between the points $p_3+\eta
v_u$ and $\sixmap(p_3+\eta v_u)$, as in the previous section.
First we look for the \emph{smallest} natural $n$ such that
$\sixmap^n(l_u)$ intersects the $p_x=0$ axis.
Then we use a standard numerical method (bisection-like
one-dimensional root finding) to find a point $z_u$ in the fundamental
segment $l_u$ such that
\[ \pi_{p_x}(\sixmap^{n}(z_u))=0. \]
Thus we obtain the homoclinic point $z_1 = \sixmap^{n}(z_u)$ in
Figure~\ref{fig:invmfld2_it23_H174}.
Numerically, we verify that $z_1$ is in the the $p_x=0$ axis within
$10^{-10}$ tolerance.

\begin{figure}[h]
\begin{center}
\psfrag{H}{$J$}
\psfrag{x}{$x$}
\includegraphics[width=10cm]{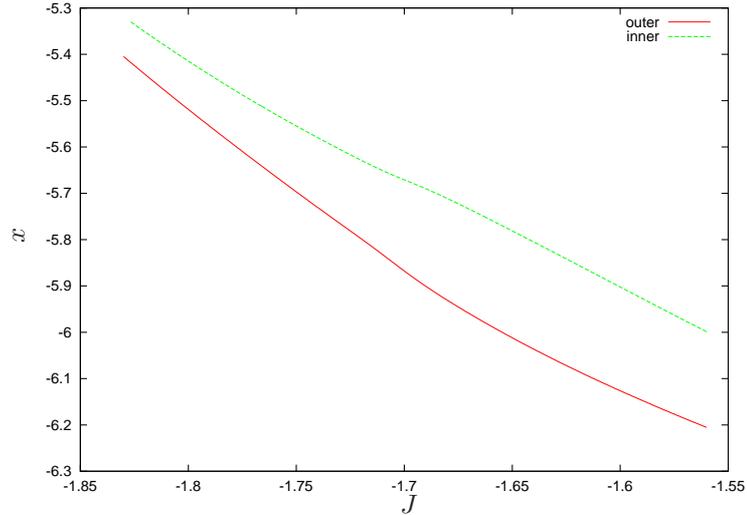}
\end{center}
\caption{Family of primary intersection points corresponding to outer
and inner splitting.
For every energy level $J$, we plot the $x$ coordinate of the
intersection point $z_1$ and $z_2$ (the $p_x$ coordinate is equal to
zero). Notice that both families are continuous.}
\label{fig:intersec}
\end{figure}

Finally, we vary energy $J$ and use a continuation method to obtain
the family of primary intersections $\{z_1\}_J$, using as seed the
primary intersection $\left.{z_1}\right|_{J=-1.74}$ found above.
See Figure~\ref{fig:intersec}.

\begin{remark} \label{rem:primary_intersection}
For low energy levels (such as $J=-1.74$), corresponding to weak
hyperbolicity, the invariant manifolds behave as if they were close to
integrable, and the primary intersection corresponds to the
\emph{first} intersection of the manifolds with the $p_x=0$ axis.
For high energy levels (such as $J=-1.6$), corresponding to strong
hyperbolicity, the manifolds develop some folds, and thus the primary
intersection may not correspond to the \emph{first} intersection of
the manifolds with the $p_x=0$ axis. See
Figure~\ref{fig:invmfld2_it23}.

In practice, we first identify the primary intersection at low energy
levels, and then use a continuation method to obtain the primary
family of intersections up to high energy levels.
\end{remark}

\begin{figure}[h]
\begin{center}
\psfrag{x}{$x$}
\psfrag{px}{$p_x$}
\psfrag{z}{$z_1$}
\psfrag{wu}{$w_u$}
\psfrag{ws}{$w_s$}
\includegraphics[width=10cm]{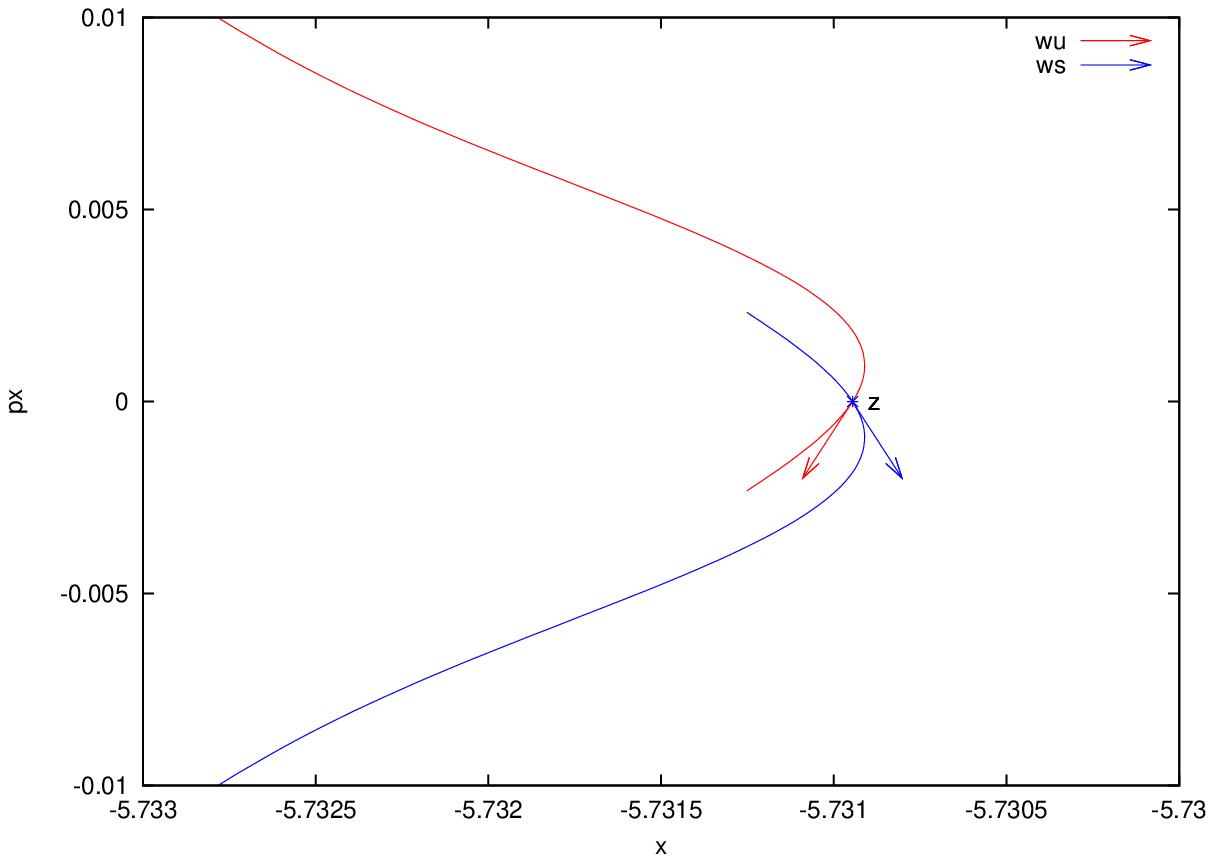}
\end{center}
\caption{Outer splitting of the manifolds for energy level $J=-1.74$.
This is a magnification of Figure~\ref{fig:invmfld2_it23_H174} at the
intersection point~$z_1$.
We show the vectors $w_u, w_s$ tangent to the unstable and stable
manifolds at $z_1$.
The splitting angle $\sigma$ is the angle between $w_u$ and $w_s$.}
\label{fig:splitangle}
\end{figure}

Analogously, we compute the family of primary intersections
$\{z_2\}_J$ corresponding to the inner splitting.
See Figure~\ref{fig:intersec}.

Let us now compute the splitting angle between the manifolds $W^{u,1}(p_3)$
and $W^{s,1}(p_2)$ at the point $z_1$.
For illustration, we show some numerical results corresponding
to the energy value $J=-1.74$.
First we need the tangent vectors $w_u$ and $w_s$ to the manifolds at
$z_1$. See Figure~\ref{fig:splitangle}.
As found above, let $z_u$ be the point in the unstable fundamental
segment that maps to $z_1$, i.e. $\sixmap^{n}(z_u) = z_1$.
Consider the tangent vector $v_u$ to the manifold $W^{u,1}(p_3)$ at
the point $z_u$.
(Recall that at this point the linear approximation is good enough, so
we can use as $v_u$ the unstable eigenvector.)
Multiply $v_u$ by the Jacobian of $\sixmap$ at the successive iterates
$\sixmap^{i}(p_u)$, for $i=0,...,n-1$.
This way, we obtain the tangent vector to the unstable manifold at
$z_1$. Let us denote this vector $w_u=(w_1,w_2)$.
We normalize it to $\|w_u\|=1$.

Due to reversibility, the vector $w_s$ tangent to the stable manifold
at $z_1$
is $w_s=(w_1,-w_2)$. See Figure~\ref{fig:splitangle}. Notice that we
choose the tangent vectors with the appropriate orientation, i.e. with
the same orientation as the trajectories on the manifolds.

Thus the oriented splitting angle between $w_u$ and $w_s$ is
\[ \sigma= 2\arctan_2(-w_1,-w_2), \]
where $\arctan_2$ is the arctangent function of two variables, which
uses the signs of the two arguments to determine the sign of the
result.

\begin{figure}[h]
\begin{center}
\psfrag{H}{$J$}
\psfrag{s}{$\sigma$ (radians)}
\includegraphics[width=10cm]{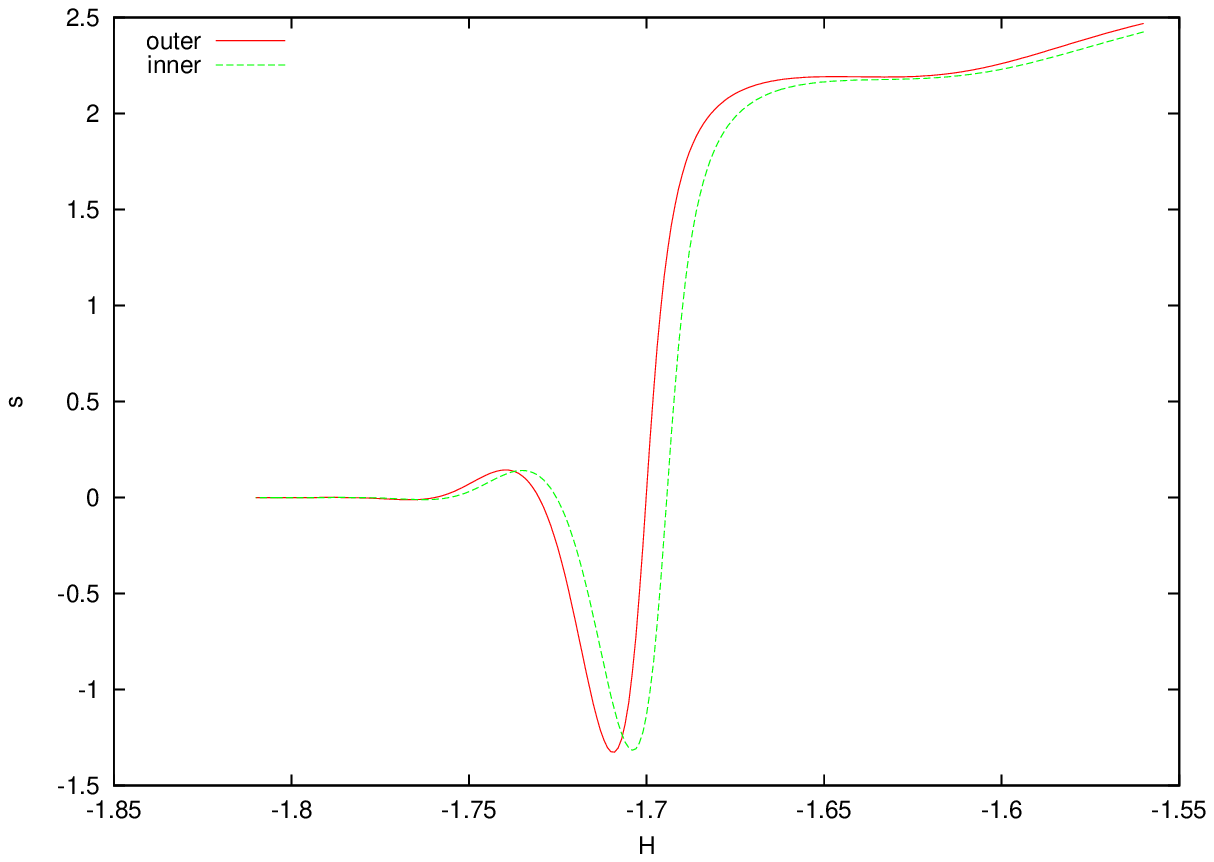}
\end{center}
\caption{Splitting angle associated to inner and outer splitting.}
\label{fig:splittings}
\end{figure}

\begin{table}
\begin{center}
\begin{tabular}{|c|c|}
\hline
inner & outer \\
\hline
$ ( -1.695,-1.694 ) $ & $ ( -1.701,-1.700 ) $ \\
$ ( -1.726,-1.725 ) $ & $( -1.731,-1.730 ) $ \\
$ ( -1.756,-1.755 ) $ & $ ( -1.760,-1.759 ) $ \\
$ ( -1.781,-1.780 ) $ & $ ( -1.784,-1.783 ) $ \\
$ ( -1.802,-1.801 ) $ & $ ( -1.805,-1.804 ) $ \\
\hline
\end{tabular}
\end{center}
\caption{Subintervals of $J\in[J_-,J_+]$ containing the zeros of inner
splitting (left column) and outer splitting (right column).}
\label{tab:zeros_inner_outer}
\end{table}

Finally, we let $J$ change and, using this procedure, we are able to
obtain the splitting angle for energy levels $J \in [J_-, J_+]$.
See Figure~\ref{fig:splittings}.
The splitting angle is nonzero for all energy values except for a
discrete set of them.
The splitting angle oscillates around zero with decreasing amplitude
as $J\to J_-$.
Numerically, we find that the zeros of the splitting angle are
contained in the intervals listed in
Table~\ref{tab:zeros_inner_outer}.

Notice that the inner and outer splittings behave similarly.
However, they become zero at different values of $J$, as seen in
Table~\ref{tab:zeros_inner_outer}.
Thus, when one of the intersections becomes tangent, the other one is
still transversal, and we can always use one of them for diffusion.

\subsection{Accuracy of computations}
\label{sec:accuracy_computations}

For small eccentricities, the splitting angle~$\sigma$ becomes very
small.
We need to check the validity of $\sigma$, making sure that the size
of (accumulated) numerical errors in the computation is smaller than
the size of $\sigma$.

The smallest splitting angle in Figure~\ref{fig:splittings},
corresponding to $J_-=-1.81$, is
\[ \sigma(J_-)= -1.777970294158603 \times 10^{-5}. \]
We check the validity of $\sigma(J_-)$ by recomputing this angle using
an alternative numerical method.
First we compute the intersection of the manifolds $W^{u,1}(p_3)$ and
$W^{s,1}(p_2)$ with the horizontal axis defined by
\[ p_x = \frac{j}{10^{5}} \]
for $j\in (-2,-1,1,2)$.

\begin{table}
\begin{center}
\begin{tabular}{|c|c|c|c|}
\hline
$p_x$ & $x^u$ & $x^s$ & $x^u-x^s$ \\
\hline
$-0.00002$ & $-5.481541931871417$ & $-5.481541932226887$ &
$0.000000000355470$ \\
$-0.00001$ & $-5.481541931790012$ & $-5.481541931967703$ &
$0.000000000177691$ \\
$0.00000$ & $-5.481541931822124$ & $-5.481541931822124$ &
$0.000000000000000$ \\
$0.00001$ & $-5.481541931967703$ & $-5.481541931790012$ &
$-0.000000000177691$ \\
$0.00002$ & $-5.481541932226887$ & $-5.481541931871417$ &
$-0.000000000355470$\\
\hline
\end{tabular}
\end{center}
\caption{Sampling of the manifolds $W^{u,1}(p_3)$ and $W^{s,1}(p_2)$ at
different values of $p_x$, and their difference (last column).}
\label{tab:differentiation}
\end{table}

In Table~\ref{tab:differentiation} we tabulate the $x$ coordinate of
$W^{u,1}(p_3)$ and $W^{s,1}(p_2)$ on these axes, and their difference
$d=x^u-x^s$ gives the distance between the manifolds.
We apply numerical differentiation to the last column of this table,
using central differences centered at $z_1$ with step sizes $0.00002$
and $0.00004$, and obtain the values:
\[ d_1 = \frac{d(0.00001) - d(-0.00001)}{0.00002} =
-0.0000177691.\]
\[ d_2 = \frac{d(0.00002) - d(-0.00002)}{0.00004} =
-0.0000177735.\]
Finally, we use Richardson extrapolation and obtain:
\[ d = \frac{4 d_1-d_2}{3} = -0.00001776763333333333. \]
Thus, using this alternative method, we obtain the splitting angle
\[ \sigma(J_-) = \mathrm{atan}(-0.00001776763333333333) =
-0.00001776763333146364. \]
Compare the splitting angle computed using the two methods. They
differ by approximately $10^{-8}$.
This gives an estimate of the numerical error commited in our
computation of the splitting angle.

\begin{figure}[h]
\begin{center}
\psfrag{H}{$J$}
\psfrag{srad}{$\sigma$ (radians)}
\psfrag{s}{$\sigma$}
\psfrag{err}{$\mathrm{err}$}
\includegraphics[width=10cm]{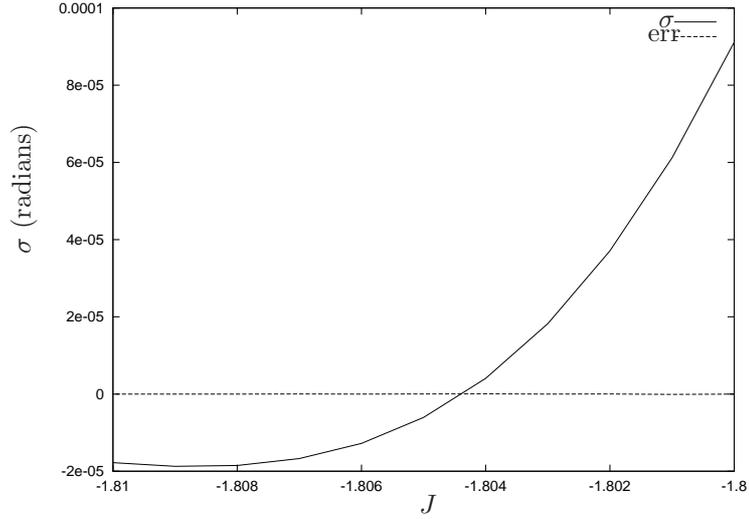}
\caption{Splitting angle $\sigma(J)$ and estimate of the numerical
error $\mathrm{err}(J)$ as a function of energy level $J$.}
\label{fig:diffs}
\end{center}
\end{figure}

We repeat this test for a range of energies $J\in[-1.81,-1.8]$.
In Figure~\ref{fig:diffs}, we compare the splitting angle $\sigma(J)$
and the estimate of the numerical error $\mathrm{err}(J)$. This error stays
below $10^{-7}$, and it is several orders of magnitude smaller than
the splitting angle.
For higher energy values $J\in[-1.8,-1.56]$, the splitting angle is
large, so the numerical error is certainly smaller.
Therefore we are confident that the splitting angle has been accurately
computed in the range of eccentricities considered,
$[J_-,J_+]$.

\section{Numerical study of the inner and outer dynamics}
\label{sec:NumericalStudyInnerOuter}

In Appendix \ref{app:NHIMCircular} we have studied the periodic orbits
and the invariant manifolds in rotating Cartesian coordinates
$(x,y,p_x,p_y)$. Nevertheless, the study of the inner and outer maps
are done in rotating Delaunay coordinates. Indeed since these
coordinates are action-angle coordinates for the two body problem, it is
much more convenient to use them to study the mean motion resonance.

\begin{figure}[h]
\begin{center}
\psfrag{x}{$x$}\psfrag{px}{$p_x$}\psfrag{p0}{$p_0$}\psfrag{p1}{$p_1$}\psfrag{p2}{$p_2$}\psfrag{p3}{$p_3$}\psfrag{p4}{$p_4$}\psfrag{p5}{$p_5$}
\includegraphics[width=10cm]{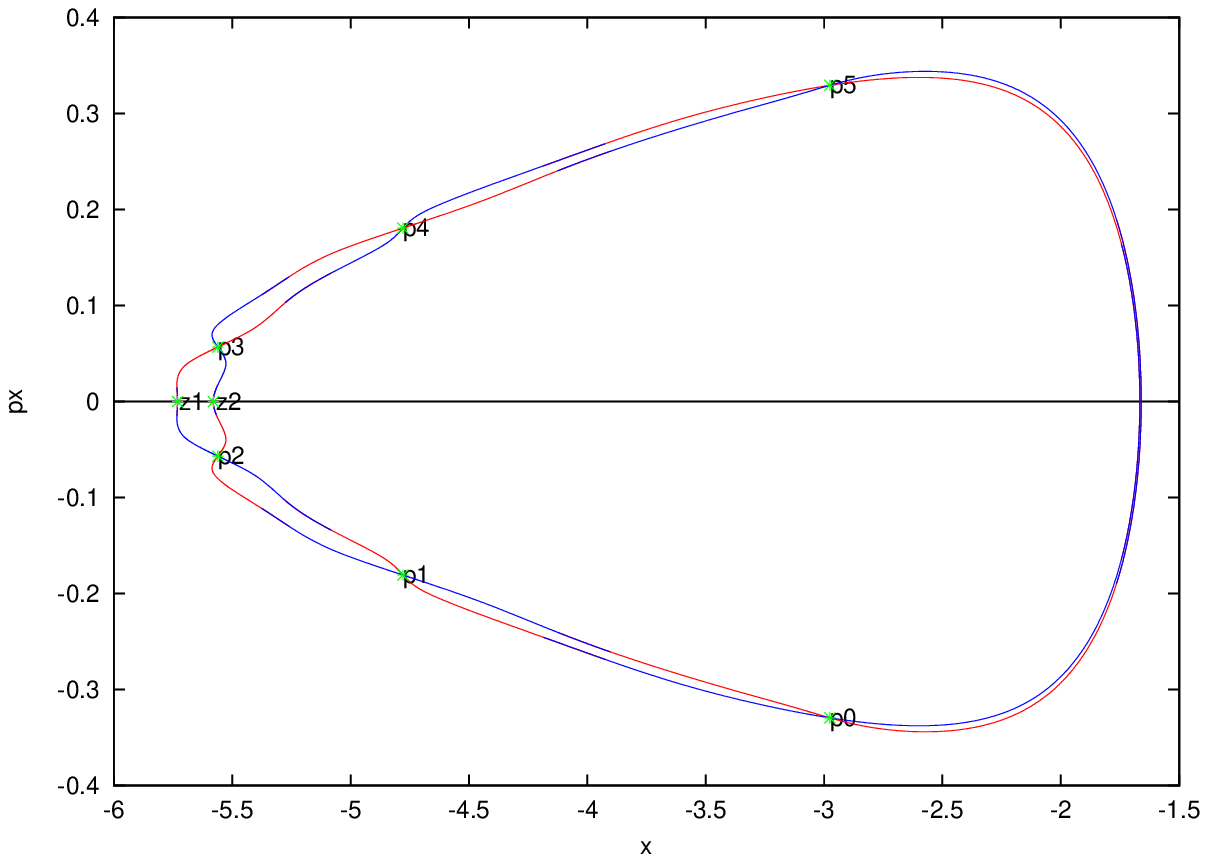}
\end{center}
\caption{Energy $J=-1.74$. Resonance structure in Cartesian coordinates. The axis of symmetry is marked with a horizontal line.}\label{fig:invmfld2_H174}
\end{figure}

The Poincar{\'e} section $\{y=0\}$ is completely different from the section
$\{g=0\}$ which will be used from now on (see \eqref{def:PoincareMap}). In
particular, the periodic orbits $\{\lambda_J\}_{J\in [\bar J_-,\bar
  J_+]}$ obtained in Appendix \ref{sec:computation_periodic_orbits}
intersect the section $\{y=0\}$ six times whereas they intersect $\{g=0\}$ seven
times. However, we remark that the homoclinic points $z_1$ and $z_2$ lie on the symmetry axis both in Cartesian and in Delaunay variables. See Figures~\ref{fig:invmfld2_H174} and~\ref{fig:orbitpdel_H174}.

To obtain the intersection of these periodic orbits with
$\{g=0\}$ we just need to express the $6$-periodic points of the
Poincar{\'e} map $P$ obtained in Appendix
\ref{sec:computation_periodic_orbits} in Delaunay coordinates and then
iterate them by the flow of the circular problem expressed in Delaunay
coordinates until they hit the section
$\{g=0\}$. We do the same with the homoclinic points. In
Appendix \ref{sec:RotatingToDelaunay} we explain how to compute the
change of coordinates and the vector field in Delaunay coordinates. Then, in Appendices \ref{app:InnerOuterCircular} and \ref{app:InnerOuterElliptic} we study the inner and outer maps of the circular and elliptic problems respectively. Finally, in Appendix \ref{app:Comparison} we compare the inner and outer maps of the elliptic problem, which leads to Ansatz \ref{ans:B}.

\begin{figure}[h]
\begin{center}
\psfrag{l}{$l$}\psfrag{L}{$L$}\psfrag{p0}{$p_0$}\psfrag{p1}{$p_1$}\psfrag{p2}{$p_2$}
\psfrag{p3}{$p_3$}\psfrag{p4}{$p_4$}\psfrag{p5}{$p_5$}\psfrag{p6}{$p_6$}\psfrag{z1}{$z_1$}\psfrag{z2}{$z_2$}
\includegraphics[width=10cm]{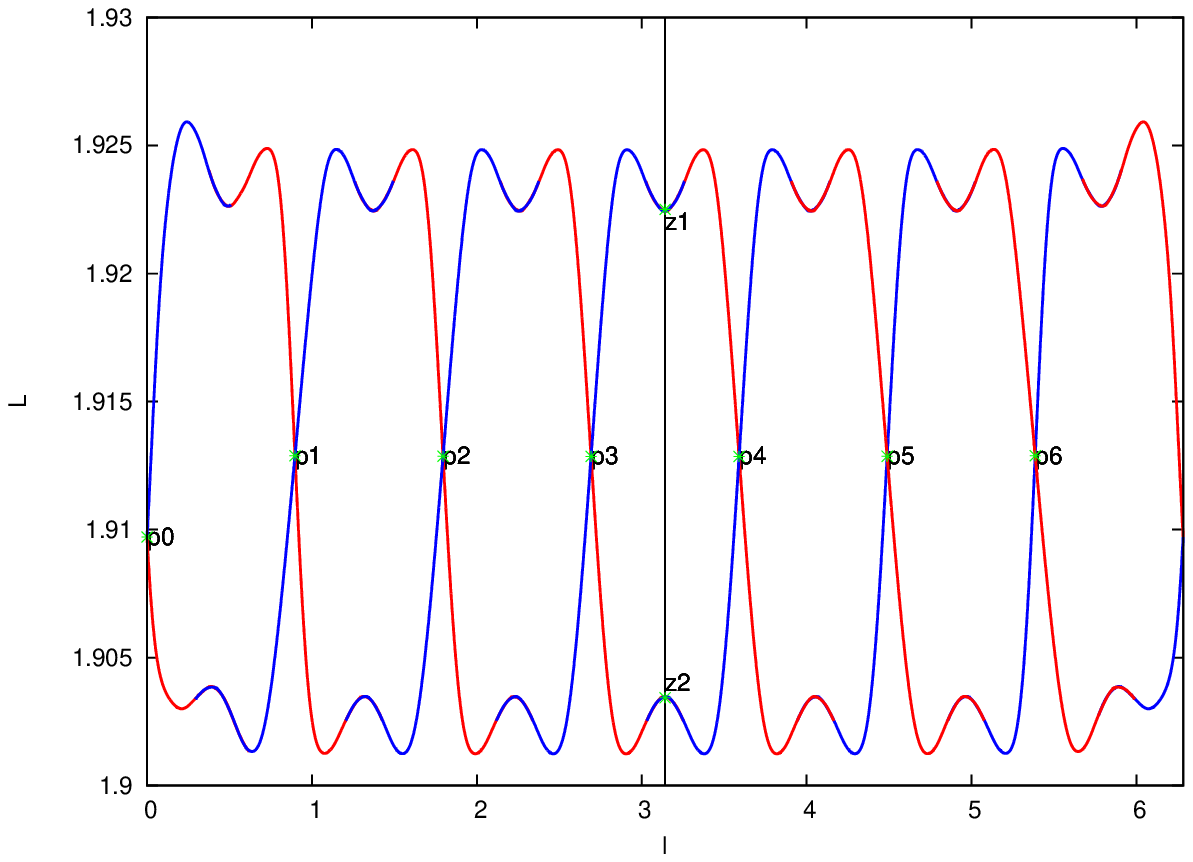}
\end{center}
\caption{Energy $J=-1.74$. Resonance structure in Delaunay coordinates. The
symmetry corresponds to $l=0$ and $l=\pi$ and is marked with a vertical
line.}\label{fig:orbitpdel_H174}
\end{figure}

\subsection{From Cartesian to Delaunay and computation of
  $\pa_G\Delta H_\ccirc$}\label{sec:RotatingToDelaunay}

We explain an easy way to obtain the rotating Delaunay coordinates
from rotating Cartesian (or polar) coordinates in the circular
problem. First recall that $G$ can be computed as
\[
G=r\left(-p_x\sin\phi+p_y\cos\phi\right).
\]

The potential $\mu\Delta H_\ccirc$ in Cartesian coordinates only
depends on the position $(x,y)$ of the asteroid, and can be easily
computed. Then, one can use the equation
\[
J=-\frac{1}{2L^2}-G+\mu\Delta H_\ccirc
\]
to obtain $L$.  Knowing $L$ and $G$ we can obtain the eccentricity $e$ by
\[
e=\sqrt{1-\frac{G^2}{L^2}}.
\]
Using that $r=L^2(1-e\cos u)$, one can obtain $u$ and from here $\ell$
using Kepler's equation $u-e\sin u=\ell$. On the other hand, from $u$
we can obtain $v$ using
\[
\tan\frac{v}{2}=\sqrt{\frac{1+e}{1-e}}\tan\frac{u}{2}.
\]
Finally, we can deduce $g$ using that $\phi=v+g$.

\medskip We devote the rest of this appendix to compute $\pa_G\Delta
H_{\ccirc}$. The other derivatives of $\Delta H_{\ccirc}$ can be
computed analogously. Let us define
$$D[r_0] = D[r_0](r,v,g) =
\left(r^2+r_0^2-2rr_0\cos(v+g)\right)^{-1/2}.$$
Then
$$\Delta H_{circ}(L,\ell,G,g) = -(1-\mu)D[-\mu] - \mu D[1-\mu] -
D[0].$$
Thus by the chain rule
there only remains to compute $\pa_G r$ and $\pa_G v$. First, let us
point out that
\[
\pa_G e=-\frac{G}{eL^2}=\frac{e^2-1}{eG}.
\]
On the other hand, using that $\ell=u-e\sin u$, one has that
\[
\pa_e u=\frac{\sin u}{1-e\cos u}.
\]
Then, since $r(L,e,\ell)=L^2(1-e\cos u(e,\ell))$,  using that
\begin{equation}\label{eq:FromUtoV}
  \cos v=\frac{\cos u -e}{1-e\cos u},
\end{equation}
we have that
\[
\pa_e r(L,e,\ell)=L^2\cos v
\]
and therefore,
\[
\pa_G r(L,\ell,G)=-\frac{G\cos v}{e}.
\]
To compute $\pa_G v$, let us point out that $\pa_G v=\pa_e v\pa_G e$.
Therefore it only remains to compute $\pa_e v$, we obtain it using
formula \eqref{eq:FromUtoV} and
\[
\sin v=\frac{\sqrt{1-e^2}\sin u}{1-e\cos u}.
\]
Then,
\[
\pa_e v=\frac{\sin v}{1-e^2}\left(2+e\cos v\right).
\]
and therefore,
\[
\pa_G v=-\frac{\sin v}{eG}\left(2+e\cos v\right).
\]

\subsection{Inner and outer dynamics of the circular problem}
\label{app:InnerOuterCircular}

In this appendix, we numerically compute the inner map $\FF_0^\inn$ and
the outer maps $\FF_0^{\out,\ast}$ of the circular problem, given in
Section~\ref{Section:Circular}. Recall that to compute these maps we deal with the extended system given by the Hamiltonian $H$ in  \eqref{def:HamDelaunayRot} with $e_0=0$ restricted to the energy level $H=0$ and thus, we have that $I=-J$. Then, we consider $I\in [I_-, I_+]=[- J_+,- J_-]$, where the range $[- J_+,- J_-]$ is given in  \eqref{def:EnergyRange:Splitting}.

As seen in Section~\ref{sec:Circular:Inner}, the inner map has the
form
\begin{equation}\label{def:InnerMap:Circular:Numerics}
  \FF_0^\inn:\left(\begin{array}{c} I\\
      t
    \end{array}\right)\mapsto \left(\begin{array}{c} I\\
      t+\mu\TTT_0(I)
    \end{array}\right),
\end{equation}
where $T_J = 14\pi + \mu\TTT_0(I)$ is the period of the periodic orbit
obtained in Ansatz~\ref{ans:NHIMCircular} on the corresponding level
of energy $J$, which now corresponds to an invariant hyperplane
$I=\mathrm{constant}$.

Recall that we computed the periodic orbit $\la_J$ as well as its
period $T_J$ in Appendix~\ref{sec:computation_periodic_orbits}.
In particular, Figure~\ref{fig:porbits} shows a plot of the function
$T_J - 14\pi = \mu\TTT_0(I)$.
Notice that the derivative of the function $\TTT_0(I)$ is nonzero for
the whole range $[I_-,  I_+]$.
This shows that the inner map is twist.
Moreover, Figure~\ref{fig:porbits} shows that
\[ 0<\mu\TTT_0(I)<60\mu<\pi. \]
Therefore, the function $\TTT_0(I)$ satisfies the properties
stated in Lemma~\ref{ans:TwistInner}.

As a test, we have computed the same function $\TTT_0(I)$ using two
different methods. First by computing the period of the periodic
orbit, as above. Then by computing the integral
expression~\eqref{def:T0:Integral} using numerical integration. The
difference in $\TTT_0(I)$ using both methods is of the order
$10^{-12}$.

\begin{figure}[h]
\begin{center}
\psfrag{H}{$J$}
\psfrag{wf}{$\omega^{\ff}$}
\psfrag{wb}{$\omega^{\bb}$}
\includegraphics[width=10cm]{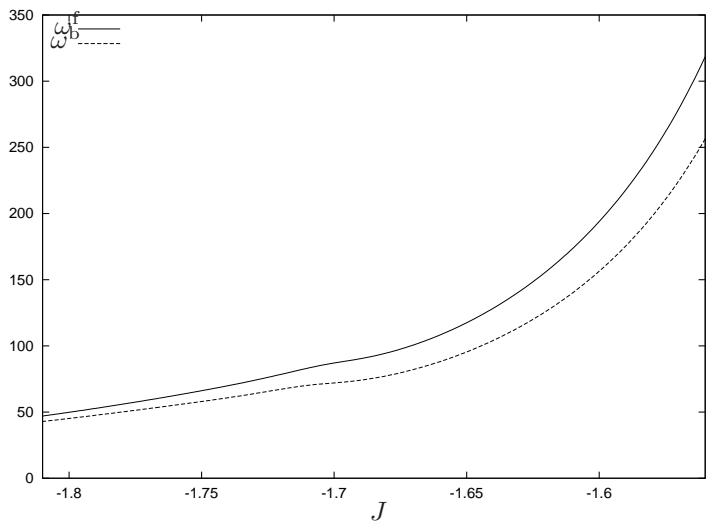}
\end{center}
\caption{Functions $\omega^{\ff}(I)$ and $\omega^{\bb}(I)$ involved in the
definition of the outer map~\eqref{def:OuterMap:Circular:Numerics} of
the circular problem as a function of the Jacobi constant $J$ (recall
that in the circular problem $I=-J$).}
\label{fig:outer_circular}
\end{figure}

As seen in Section~\ref{sec:Circular:Outer}, the outer maps have the
form
\begin{equation}\label{def:OuterMap:Circular:Numerics}
  \FF_0^{\out,\ast}:\left(\begin{array}{c} I\\
      t
    \end{array}\right)\mapsto \left(\begin{array}{c} I\\
      t+\mu\omega^\ast(I)
    \end{array}\right),\,\,\,\ast=\ff,\bb.
\end{equation}
For simplicity, let us only discuss the computation of $\omega^\ff(I)$
($\omega^\bb(I)$ is computed analogously).
Recall from Lemma~\ref{lem:Omega0} that the function $\omega^\ff(I)$
is defined as
  \[
  \omega^\ff(I)= \omega^\ff_\out(I)+\omega_\inn^\ff (I),
  \]
where, taking into account that the homoclinic orbit is symmetric with respect to the involution \eqref{def:involution},
\begin{equation}\label{def:Omega0:OuterPart:Numerics}
 \omega^\ff_\out(I)=\omega_+^\ff(I)-\omega^\ff_-(I) = 2\omega_+^\ff(I)
\end{equation}
  with
  \begin{equation}\label{def:Omega0PlusMinus:Numerics}
 \begin{split}
 \omega_+^\ff(I)&=\lim_{N\rightarrow+\infty}\left(\int_0^{ 14N\pi
      }\frac{(\pa_G\Delta
        H_\ccirc) \circ \gamma_I^\ff(\sigma)}{-1+\mu(\pa_G\Delta
        H_\ccirc) \circ \gamma_I^\ff(\sigma)} \,
d\sigma+N\TTT_0(I)\right),\,\,\,
\end{split}
\end{equation}
\begin{equation}\label{def:Omega0:InnerPart:Numerics}
\begin{split}
 \omega_\inn^\ff(I)&=\int_0^{ -12\pi
      }\frac{(\pa_G\Delta
        H_\ccirc) \circ \la_I^4(\sigma)}{-1+\mu(\pa_G\Delta
        H_\ccirc) \circ \la_I^4(\sigma)} \, d\sigma.
\end{split}
\end{equation}
To obtain $\omega^\ff(I)$, we compute the
integrals~\eqref{def:Omega0PlusMinus:Numerics}
and~\eqref{def:Omega0:InnerPart:Numerics} numerically, using a standard
algorithm from the GSL library~\cite{GSL}.
The integrals are computed within a relative error limit $10^{-9}$.

The function $\pa_G\Delta H_\ccirc$ involved in both integrals is
given explicitly in Appendix~\ref{sec:RotatingToDelaunay}.
The integral $\omega_\inn^\ff(I)$ is evaluated on a periodic
trajectory $\la_I^4(\sigma)$ of the reduced circular problem (namely,
with reparameterized time, see \eqref{def:Reduced:ODE:Circ}) with initial
condition $p_4$, a fixed point of the Poincar{\'e} map $\PP_0^7$.
The integral $\omega_+^\ff(I)$ is evaluated on a homoclinic trajectory
$\gamma_I^\ff(\sigma)$ of the reduced circular problem with initial condition
$z_2$, the primary homoclinic point corresponding to the inner
splitting found in Appendix~\ref{sec:homoclinic_points}.

\begin{figure}[h]
\begin{center}
\psfrag{N}{$N$}
\psfrag{d}{$dist^+$}
\includegraphics[width=10cm]{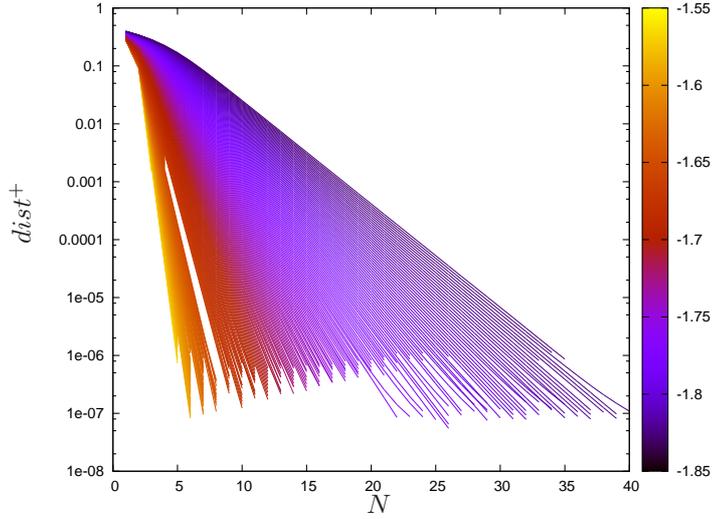}
\end{center}
\caption{Exponential decay of the function $\mathrm{dist}^+$ as a function of
$N$ (multiples of the period) for different energy levels.
The energy levels $J\in[J_-,J_+]$ are color-coded.}
\label{fig:outer_circ_test}
\end{figure}

Next we make a couple of important remarks about the numerical
computation of the integral $\omega_+^\ff(I)$.
The key point is that the homoclinic orbit $\gamma_I^\ff$ was already
computed in Appendix~\ref{sec:homoclinic_points} with high accuracy,
and we can exploit this information here.
Recall that the primary homoclinic point $z_2$ was obtained as the
$n$-th iterate of a point $z_u$ in the local fundamental segment $l_u$
under the Poincar{\'e} map:
\begin{equation} \label{eq:zu_to_z2}
 z_2 = \{\PP_0^7\}^n (z_u).
\end{equation}
Moreover, recall that the point $z_u$ was chosen to be suitably close
to the fixed point $p_3$ for each energy level $J$.
See Remark~\ref{rem:displacement}.

Notice that the integral $\omega_+^\ff(I)$ is defined by a limit as
$N\to\infty$, i.e. as the homoclinic orbit $\gamma_I^\ff(\sigma)$
asymptotically approaches the periodic orbit $\la_I^3(\sigma)$ in
forward time (see equation~\eqref{def:OmegaSmall}).
Numerically, of course, we should stop integrating at an upper
endpoint $N$ large enough such that the integral converges.
In practice, we choose the upper endpoint $N=N(I)$ to be the number of
iterates $n=n(I)$ in~\eqref{eq:zu_to_z2}.
This means that we evaluate the integral along the homoclinic
trajectory $\gamma_I^\ff(\sigma)$ until it reaches the point $z_u$,
which is suitably close to the periodic orbit.

Notice also that integrating the homoclinic trajectory
$\gamma_I^\ff(\sigma)$ forwards in the reduced system means
integrating it backwards along the unstable manifold in the original
system.
This is numerically unstable, since numerical errors grow
exponentially.
In practice, we rewrite the
integral~\eqref{def:Omega0PlusMinus:Numerics} using the change of
variables $\hat \sigma = \sigma - 14N\pi$ so that the homoclinic
trajectory is integrated forwards along the unstable manifold,
starting from the point $z_u$.

The computed values of the functions $\omega^{\ff}(I)$ and
$\omega^{\bb}(I)$ are shown in Figure~\ref{fig:outer_circular}. Note
that they are plotted as a function of the Jacobi constant $J$ instead
of as a function of $I$, so that they can be compared with Figure
\ref{fig:porbits}, where we have plotted $\mu \TTT_0(I)=T_J-14\pi$
as a function of $J$.

To test the computation of the function $\omega_+^\ff$, we directly
verify the definition of the outer map in~\ref{definition:OuterMap}.
Let $z_2 = (L_h,\ell_h,G_h,0)$ be the primary homoclinic point, and
let $p_3 = (L_p,\ell_p,G_p,0)$ be the periodic point.
Given a point $(L_h,\ell_h,G_h,0, I, t)$ in the extended circular
problem, we check that it is forward asymptotic (in the
reparametrized time) to the point $(L_p,\ell_p,G_p,0, I,
t+\omega_+^\ff(I))$, where $t\in\TT$ is arbitrary.
Thus we check that the distance
\[ \mathrm{dist}^+(s) = |\Phi_0\{s,(L_h,\ell_h,G_h,0, I, t)\} -
\Phi_0\{s,(L_p,\ell_p,G_p,0, I, t+\omega_+^\ff(I))\}|
\xrightarrow{s\to\infty} 0 \]
with exponential decay.

The result of the test is shown in Figure~\ref{fig:outer_circ_test}
for values of the energy $J\in[J_-,J_+]$ (recall that $J=-I$).
Notice that the vertical axis is in logarithmic scale.
Let $s=14N\pi$.
We plot the distance $\mathrm{dist}^+$ as a function of $N$ (multiples of the
period).
The test shows exponential decay of the distance function for all
energy values, i.e. straight lines in the plot.

Recall that the periodic orbits $\la_I^{3,4}(s)$ become more hyperbolic as
the energy $I$ decreases. Thus, the rate of exponential convergence
between the homoclinic and the periodic trajectory also increases,
i.e. the straight lines have increasing slope in the plot.
As explained above, the length of integration $N=N(I)$ along the
homoclinic orbit is suitably chosen for each energy level.
For  $I\to I_-$, there is exponential decay up to
time $s=40\cdot (14\pi) \approx 1760$.

\subsection{Inner and outer dynamics of the elliptic problem}
\label{app:InnerOuterElliptic}

In this appendix, we numerically compute the first orders in $e_0$ of
the inner map $\FF_{e_0}^\inn$ and the outer maps
$\FF_{e_0}^{\out,\ast}$ of the elliptic problem, given in
Section~\ref{sec:Elliptic}.
In order to compare the inner and outer dynamics of the elliptic
problem through Lemma~\ref{lemma:Averaging}, only some specific terms
in the expansions of the inner and outer maps are necessary.
Namely, we only need to compute the term $A_1$ in the expansion of the
inner map~\eqref{def:InnerMap:ell}, and the terms $B^*$ in the
expansion of the outer maps~\eqref{def:OuterMap:Elliptic}.

\begin{figure}[h]
\begin{center}
\psfrag{H}{$\hat H$}\psfrag{A}{$A_1^+(I)$}\psfrag{reXXXXXX}{$\Re(A_1^+)$}\psfrag{imXXXXXX}{$\Im(A_1^+)$}
\includegraphics[width=10cm]{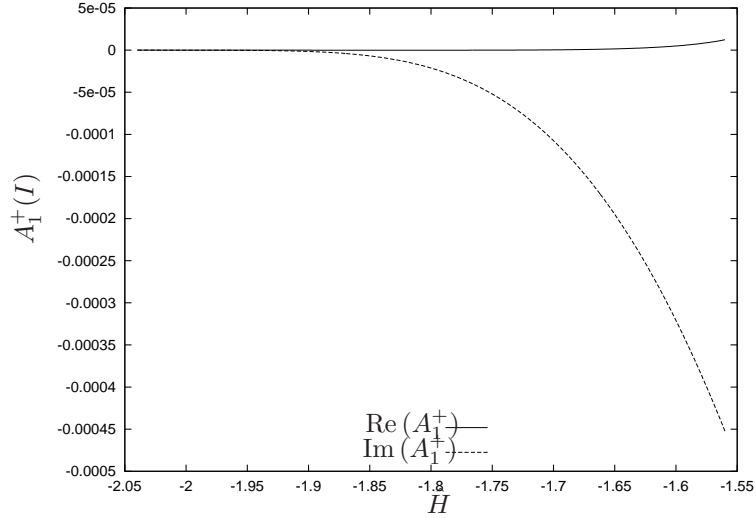}
\end{center}
\caption{Function $A_1^+(I)$ (real and imaginary parts) involved in the definition of the inner map~\eqref{def:InnerMap:ell} of the elliptic problem as a function of the energy of the system in rotating coordinates $\hat H$. Recall that $\hat H = -I$.}
\label{fig:inner_ell}
\end{figure}

Recall from Section~\ref{sec:CylinderExpansion} that $A_1$ can be
split as
  \[
  A_1(I,t)=A_1^+(I)e^{it}+A_1^-(I)e^{-it}.
  \]
Since $A_1^+$ and $A_1^-$ are complex conjugate, it is only necessary
to compute one of them. Let us compute the positive harmonic,
  \begin{equation}\label{def:A:plus:Numerics}
    A_1^+(I)=-i\mu\int_0^{-14\pi}
      \frac{\Delta H_{\eell}^{1,+}\circ \la_I^3(\sigma)}
      {-1+\mu\pa_G\Delta H_\ccirc\circ \la_I^3(\sigma)}
      e^{i\wt\la_I^3(\sigma)}d\sigma.
  \end{equation}
Notice that the denominator is the same one used in the previous
section for the inner and outer dynamics of the circular problem.
Next we give the numerator $i\Delta H_{\eell}^{1,+}$ explicitly.
Let
  \[
    \begin{split}
      \Delta H^1_\eell(L,\ell,G,g,t) =&
-\frac{1-\mu}{\mu}\BB_1\left(-\frac{r(L,\ell,G)}{\mu},v(L,\ell,G),g,t\right)\\
      &-\frac{\mu}{1-\mu}\BB_1\left(\frac{r(L,\ell,G)}{1-\mu},v(L,\ell,G),g,t\right),
    \end{split}
  \]
  where $\BB_1$ is the function defined in Lemma
\ref{lemma:ExpansionB}.
Then it is straightforward to see that
\begin{equation}
  \begin{split}
\Delta H_{\eell}^{1,+}(l,L,g,G)
=&
-\frac{1-\mu}{\mu}\BB_1^+\left(-\frac{r(L,\ell,G)}{\mu},v(L,\ell,G),g\right)
\\
      &-\frac{\mu}{1-\mu}\BB_1^+\left(\frac{r(L,\ell,G)}{1-\mu},v(L,\ell,G),g\right),
  \end{split}
\end{equation}
where
\[ \BB_1^+(r,v,g) =
-\frac{1-r\cos(v+g)-i2r\sin(v+g)}{2\Delta^3(r,v+g)}. \]

\begin{figure}[h]
\begin{center}
\psfrag{H}{$\hat H$}\psfrag{reBfXXXXXX}{$\Re(B^{\ff,+})$}\psfrag{imBfXXXXXX}{$\Im(B^{\ff,+})$}\psfrag{reBbXXXXXX}{$\Re(B^{\bb,+})$}\psfrag{imBbXXXXXX}{$\Im(B^{\bb,+})$}
\includegraphics[width=10cm]{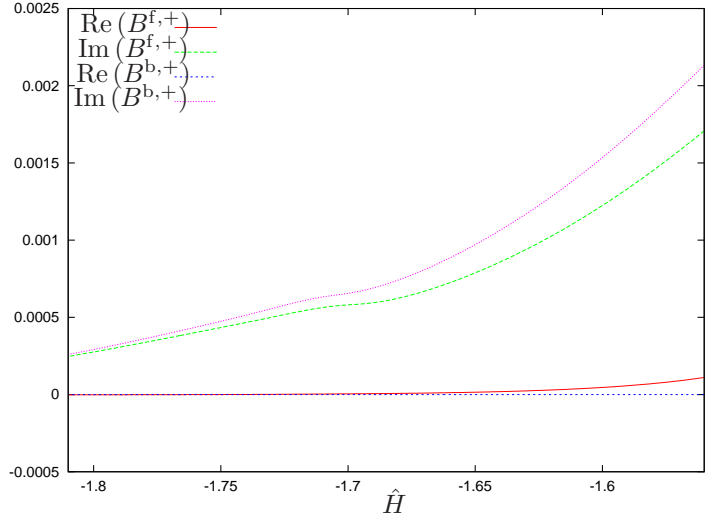}
\caption{Functions $B^{\ff,+}$ and $B^{\bb,+}$ (real and imaginary parts) involved in the definition of the outer map~\eqref{def:OuterMap:Elliptic} of the elliptic problem.}
\label{fig:B_fb}
\end{center}
\end{figure}

The computed value of the function $A_1^+$ is shown in
Figure~\ref{fig:inner_ell}. We plot it as a function of
the energy of the elliptic problem in rotating coordinates
$\hat H$ in \eqref{def:HamDelaunayNonRot}. Recall that since
we are working in the energy level $H=0$ of the extended
Hamiltonian $H$ in \eqref{def:HamDelaunayRot}, we have that $I=-\hat H$.

For the outer map, we compute the functions $B^*(I)$.
Similarly to $A_1$, it is only necessary to compute the positive
harmonics $B^{*,+}$.
Recall from Lemma~\ref{lemma:Outer:Elliptic} that the positive
harmonics $B^{\ff,+}(I)$ and $B^{\bb,+}(I)$ are defined as
\begin{equation}\label{def:Omega:PlusMinus:Numerics}
\begin{split}
B^{\ff,+}(I)&=B_\out^{\ff,+}(I)+B_\inn^{\ff,+}(I)e^{i\mu\omega_\out^\ff(I)}\\
B^{\bb,+}(I)&=B_\inn^{\bb,+}(I)+B_\out^{\bb,+}(I)e^{i\mu\omega_\inn^\bb(I)},
\end{split}
\end{equation}
where $\omega_\out^\ff$ and $\omega_\inn^\bb$ were obtained in
Appendix~\ref{app:InnerOuterCircular}.
To obtain $B_\out^{*,+}$ and $B_\inn^{*,+}$, we compute the
integrals~\eqref{def:Omega:PlusMinus:Out:for}--\eqref{def:Omega:PlusMinus:Inn}
numerically, using the same techniques as in the previous
Appendix~\ref{app:InnerOuterCircular}. In particular, the integrands
of the Melnikov integrals \eqref{def:Omega:PlusMinus:Out:for} and
\eqref{def:Omega:PlusMinus:back}, by construction, decay exponentially
as $T\rightarrow \pm\infty$ and we take the same approximate limits of
integration $\pm 14\pi N$ where $N=N(I)$ is the constant considered in
Appendix \ref{app:InnerOuterCircular}.

The computed values of the functions $B^{\ff,+}(I)$ and $B^{\bb,+}(I)$ are
shown in Figure~\ref{fig:B_fb}.

\subsection{Comparison of the inner and outer dynamics of the elliptic problem}\label{app:Comparison}

Finally, we verify the non-degeneracy
condition
 \begin{equation}\label{eq:NonVanishing:Outer:Numerics}
   \wt B^{\ast,\pm} \left(\II\right)\neq 0\qquad\text{for }\II\in
\DD^\ast
 \end{equation}
stated in Lemma~\ref{lemma:Averaging}, which implies the existence of
a transition chain of tori.
Since $B^{\ast,+}$ and $B^{\ast,-}$ are complex-conjugate, it is only
necessary to compute one of them. Let us compute the positive
harmonic,
 \[
 \wt B^{\ast,+} \left(\II\right)=B^{\ast,+}
\left(\II\right)-\frac{e^{i\mu\omega^\ast(\II)}-1}{e^{
i\mu\TTT_0(\II)}-1}A_1^+ \left(\II\right).
 \]
All the functions involved in the expression above are known: $\TTT_0$
and $\omega^\ast$ are obtained in
Appendix~\ref{app:InnerOuterCircular} and $A_1^+$ and $B^{\ast,+}$ are
obtained in Appendix~\ref{app:InnerOuterElliptic}.

\begin{figure}[h]
\begin{center}
\psfrag{H}{$\hat H$}\psfrag{reBfXXXXXX}{$\Re(\wt B^{\ff,+})$}\psfrag{imBfXXXXXX}{$\Im(\wt B^{\ff,+})$}\psfrag{reBbXXXXXX}{$\Re(\wt B^{\bb,+})$}\psfrag{imBbXXXXXX}{$\Im(\wt B^{\bb,+})$}
\includegraphics[width=10cm]{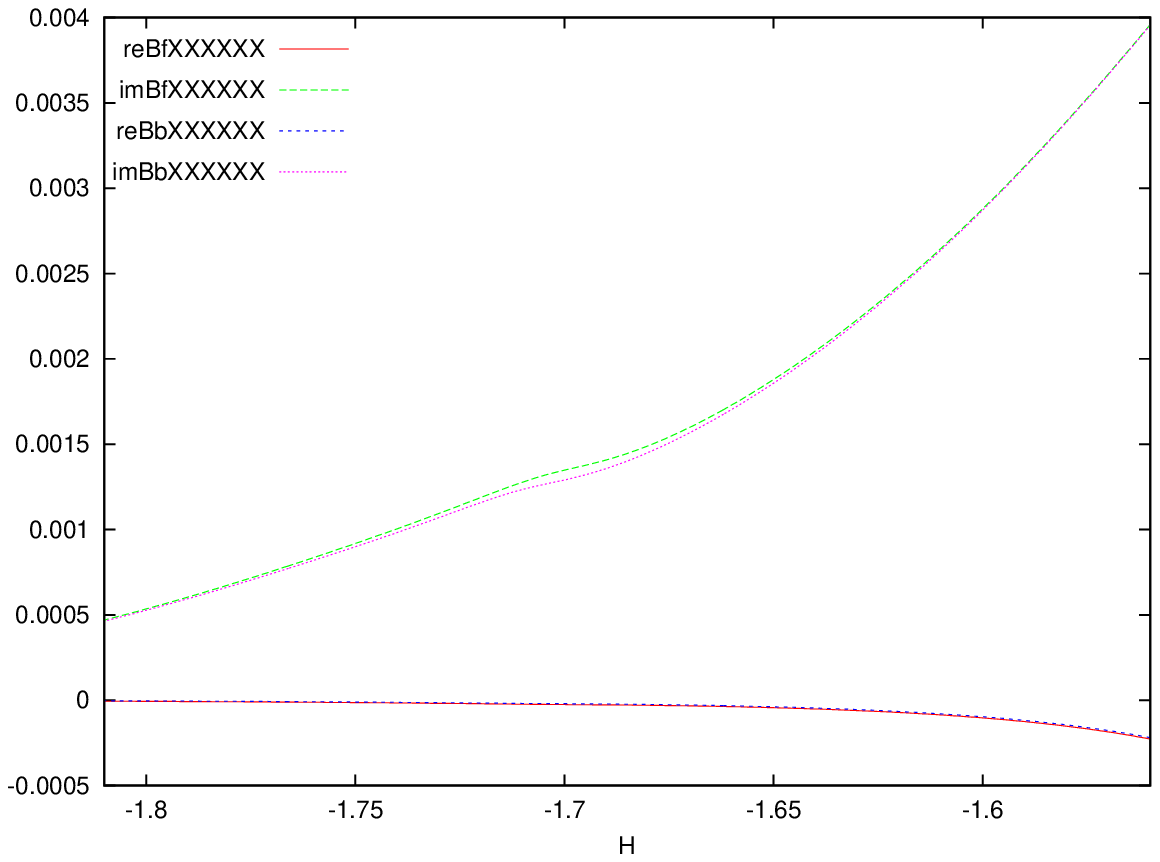}
\end{center}
\caption{Functions $\wt B^{\ff,+}$ and $\wt B^{\bb,+}$ (real and imaginary parts).}
\label{fig:tildeB}
\end{figure}

The computed values of the functions $\wt B^{\ff,+}$ and $\wt
B^{\bb,+}$ are shown in Figure~\ref{fig:tildeB}.
Therefore, we see that the functions $\wt B^{\ast,+}$ are not
identically zero.
This justifies Ansatz \ref{ans:B}.

\begin{remark}
Figure~\ref{fig:tildeB} also shows that $\wt B^{\ff,+}$ and $\wt
B^{\bb,+}$ are almost identical, which is surprising for the authors.
However, this fact is not relevant for the argument in
Lemma~\ref{lemma:Averaging}; we only need that these functions do not
vanish identically.
\end{remark}

\section{The Main Result for the $3:1$ resonance: instabilities in the Kirkwood gaps}
\label{App:Resonance1:3}

We devote this appendix to show how the proof of Theorem
\ref{th:MainTheorem:detail} in Sections
\ref{Section:Circular}--\ref{sec:ProofDiffusion} can be adapted to
deal with the $3:1$ resonances. 
First, we state a more rigorous version of Main Result $(3:1)$.

\begin{theorem}\label{th:MainTheorem:detail:2}
Assume Ans{\"a}tze~\ref{ans:NHIMCircular:2}, \ref{ans:TwistInner:2}
and~\ref{ans:B:2}. Then there exists $e_0^\ast>0$ such that for
$0<e_0<e_0^\ast$, there exist $T>0$ and an orbit of the
Hamiltonian~\eqref{def:HamDelaunayRot} which satisfy
\[
G(0)>\Gminusnew\text{ and }G(T)<\Gplusnew
\]
whereas
\[
\left| L(t)-3^{-1/3}\right|\leq 100\mu.
\]
\end{theorem}

Ans{\"a}tze~\ref{ans:NHIMCircular:2} and \ref{ans:TwistInner:2} are stated in Appendix \ref{sec:Circular:2} and Ansatz \ref{ans:B:2} is stated in Appendix \ref{sec:Elliptic:2}. They are analogous to Ans{\"a}tze \ref{ans:NHIMCircular}, \ref{ans:TwistInner} and \ref{ans:B} but referred to the $3:1$ resonance instead of the $1:7$ one. To prove this theorem, we consider the Hamiltonian
\eqref{def:HamDelaunayRot} and we study the resonance
\begin{equation}\label{def:Resonance:2}
\dot\ell \sim 3\quad\text{ and }\quad\dot g\sim -1.
\end{equation}
 As for
the $1:7$ resonance, without loss of generality, we take $H=0$ and we
look for a large drift in $I$, which being $L$ almost constant,
implies a big drift in $G$.

\subsection{The circular problem}
\label{sec:Circular:2}

We first study the circular problem \eqref{def:HamDelaunayCirc}, close
to the resonance $3\ii\dot \ell+\dot g\sim 0$. We assume the following
ansatz. It has been verified numerically (see
Appendix~\ref{app:3:1:numerics}). It replicates Ansatz
\ref{ans:NHIMCircular}.

\begin{ansatz}\label{ans:NHIMCircular:2}
  Consider the Hamiltonian \eqref{def:HamDelaunayCirc} with
  $\mu=10^{-3}$.  Then, in each energy level $J\in [J_-, J_+]=[-1.6,-1.3594]$,
  there exists a hyperbolic periodic orbit $\lb_J=(L_J(t), \ell_J(t),
  G_J(t), g_J(t))$ of period $T_J$ which satisfies
  \[
  \left| T_J-2\pi\right|<15 \mu,
  \]
  and is smooth with respect to $J$, and
  \[
  \left| L_J(t)-3^{-1/3}\right|< 100\mu
  \]
  for all $t\in\RR$.

  Each $\lb_J$ has two branches of stable and
  unstable invariant manifolds $W^{s,j}(\lb_J)$
  and $W^{u,j}(\lb_J)$, $j=1,2$. Then, for each
  $J\in [J_-,J_+]$ either $W^{s,1}(\lb_J)$ and
  $W^{u,1}(\lb_J)$ or $W^{s,2}(\lb_J)$ and
  $W^{u,2}(\lb_J)$ intersect transversally.
\end{ansatz}
Note that since now Jupiter is slower than the Asteriod, the period of the these periodic orbits is approximately the period of Jupiter instead of the period of the asteroid. For the \emph{Extended Circular Problem} given by the Hamiltonian
\eqref{def:HamDelaunayRot} with $e_0=0$, the periodic orbits obtained
in Ansatz \ref{ans:NHIMCircular:2} become invariant two-dimensional tori
which belong to hyperplanes $I=\text{constant}$ for any
\[
I\in [I_-,I_+]=[-J_+,-J_-]=[1.3594,1.6].
\]

\begin{corollary}\label{coro:NHIMCircular:2}
Assume Ansatz \ref{ans:NHIMCircular:2}. Then, the Hamiltonian \eqref{def:HamDelaunayRot} with $\mu=10^{-3}$ and
$e_0=0$ has an analytic normally hyperbolic invariant
$3$-dimensional manifold $\Lambda_0$, which is foliated by
two-dimensional invariant tori.

Moreover, $\Lambda_0$ has two branches of stable and unstable
invariant manifolds, which we call $W^{s,j}(\Lambda_0)$ and
$W^{u,j}(\Lambda_0)$, $j=1,2$. Then, in the invariant planes
$I=\text{constant}$,  for each $I\in [I_-,I_+]$ either
$W^{s,1}(\Lambda_0)$ and $W^{u,1}(\Lambda_0)$ or
$W^{s,2}(\Lambda_0)$ and $W^{u,2}(\Lambda_0)$ intersect transversally.
\end{corollary}

For the analysis of the $3:1$ resonance is more convenient to consider the global 
Poincar{\'e} section $\{\ell=0\}$ instead of the section $\{g=0\}$ considered 
in Section \ref{Section:Circular}, since now the asteroid moves faster than Jupiter. 
We consider the map
\begin{equation}\label{def:PoincareMap:2}
  \PP_0:\{\ell=0\}\longrightarrow \{\ell=0\},
\end{equation}
induced by the flow associated to the Hamiltonian \eqref{def:HamDelaunayRot} with $e_0=0$. 
Now the intersection of the cylinder $\Lambda_0$
with the section $\{\ell=0\}$ is formed by three cylinders $\wt \Lambda_0^j$,
$j=0,1, 2$. Namely,
\begin{equation}\label{def:Cylinder:Poincare:2}
\Lambda_0\cap \{\ell=0\}=\wt\Lambda_0=\bigcup_{j=0}^2\wt\Lambda_0^j.
\end{equation}
As a whole $\cup_{j=0}^2\wt\Lambda_0^j$ is a normally hyperbolic
invariant manifold for the Poincar{\'e} map $\PP_0$ whereas each $\wt\Lambda_0^j$ is 
a normally hyperbolic invariant manifold for $\PP_0^3$. These cylinders have a natural
system of coordinates, which we use to study the inner and outer dynamics on them. 
We work with $\wt \Lambda_0^1$ and $\wt\Lambda_0^2$ since in each invariant plane $I=\text{constant}$ 
they are connected by at least one heteroclinic connection (of $\PP_0^3$) which is symmetric
with respect to the involution \eqref{def:involution}. As before, we call it a
forward heteroclinic orbit if it is asymptotic to $\wt \Lambda_0^1$ in
the past and $\wt\Lambda_0^2$ in the future and a backward
heteroclinic orbit if it is asymptotic to $\wt \Lambda_0^2$ in the
past and to $\wt\Lambda_0^1$ in the future. We denote by $\DD^{\ff}
\subset [I_-,I_+]$, where $\ff$ stands for forward, the subset of
$[I_-,I_+]$ where $W^u(\wt \Lambda_0^1)$ and $W^s(\wt \Lambda_0^2)$
intersect transversally and by $\DD^{\bb} \subset [I_-,I_+]$, where
$\bb$ stands for backward, the subset of $[I_-,I_+]$ where $W^s(\wt
\Lambda_0^1)$ and $W^u(\wt \Lambda_0^2)$ intersect transversally.  By
Corollary \ref{coro:NHIMCircular:2} we have that $\DD^\ff\cup\DD^\bb=
[I_-,I_+]$.


\begin{corollary}\label{coro:NHIMCircular:Poincare:2}
Assume Ansatz \ref{ans:NHIMCircular:2}.  Then, the Poincar{\'e} map $\PP^3_0$ defined in \eqref{def:PoincareMap:2}, which
  is induced by the Hamiltonian \eqref{def:HamDelaunayRot} with
  $\mu=10^{-3}$ and $e_0=0$, has three analytic normally hyperbolic
  invariant manifolds $\wt\Lambda_0^j$, $j=0,1,2$. They are
  foliated by one-dimensional invariant curves. Moreover, there exist
  analytic functions $\GG^j_0: [I_-,I_+]\times\TT\rightarrow (\RR
  \times \TT)^3$,
\[
    \GG_0^j(I,t)=\left(\GG_0^{j,L}(I),0,
      \GG_0^{j,G}(I),\GG_0^{j,g}(I),I,t\right),
 \]
  that parameterize $\wt \Lambda_0^j$, namely,
  \[
  \wt\Lambda^j_0=\left\{ \GG_0^j(I,t):(I,t)\in[I_-,I_+]\times \TT\right\}.
  \]
  The associated invariant manifolds $W^{u}(\wt\Lambda^1_0)$ and
  $W^{s}(\wt\Lambda^3_0)$ intersect transversally provided $I\in
  \DD^\ff$; and $W^{s}(\wt\Lambda^1_0)$ and $W^{u}(\wt\Lambda^2_0)$
  intersect transversally provided $I\in \DD^\bb$.  Moreover, one of
  the points of these intersections belongs to the symmetry axis of
  \eqref{def:involution}. Let us denote by $\Gamma^{\ast}_0$,
  $\ast=\ff,\bb$, these intersections. Then, there exist analytic
  functions
  \[
  \CCC^{*}_0: \DD^j\times\RR\rightarrow
  \left(\RR\times\TT\right)^3, \quad (I,t) \mapsto \CCC^*_0(I,t),\,\,\,\ast=\ff,\bb,\]
  which parameterize them:
  \[
  \Gamma^*_0=\left\{ \CCC^*_0(I,t)=(\CCC_0^{*,L}(I),0,
    \CCC_0^{*,G}(I),\CCC_0^{*,g}(I),I,t):(I,t)\in \DD^\ast\times\TT\right\},\,\,\,\ast=\ff,\bb.
  \]
\end{corollary}

Corollary \ref{coro:NHIMCircular:2} gives global coordinates $(I,t)$ for
each cylinder $\wt\Lambda^j_0$. These coordinates are symplectic with
respect to the canonical symplectic form $\Omega_0=dI\wedge dt$. We consider the inner and the two outer maps in the
cylinder $\wt\Lambda_0^1$.

\paragraph{The inner map}
The inner map $\FF_0^\inn:\wt\Lambda_0^1 \rightarrow \wt\Lambda_0^1$ is defined as the 
Poincar{\'e} map $\PP^3_0$ restricted to the symplectic invariant
submanifold $\wt\Lambda_0^1$. It is of the form
\begin{equation}\label{def:InnerMap:Circular:2}
  \FF_0^\inn:\left(\begin{array}{c} I\\
      t
    \end{array}\right)\mapsto \left(\begin{array}{c} I\\
      t+\mu\TTT_0(I)
    \end{array}\right),
\end{equation}
where the function $\TTT_0$ is such that $2\pi+\mu\TTT_0(I)$ is the period of
the periodic orbit obtained in Ansatz~\ref{ans:NHIMCircular:2} on the
corresponding energy surface. We assume the following ansatz, which asserts
that this map is twist (see the corresponding Ansatz \ref{ans:TwistInner}).
It has been verified numerically (see Appendix~\ref{app:3:1:numerics}).
\begin{ansatz}\label{ans:TwistInner:2}
The function $\TTT_0(I)$ satisfies
  \[
  \pa_I \TTT_0(I)\neq 0\qquad \text{for }I\in [I_-,I_+].
  \]
Therefore, the analytic symplectic inner map $\FF_0^\inn$ is twist.   Moreover, the function $\TTT_0(I)$ satisfies
  \[
    0<\mu\TTT_0(I)<\pi.
  \]
\end{ansatz}

\paragraph{The outer map}
Proceeding as in Section \ref{sec:Circular:Outer}, we define the  outer map for the circular problem at the $3:1$ resonance. Recall that it has been defined as a composition of the map given by Definition \ref{definition:OuterMap} and a suitable power of the Poincar{\'e} map $\PP_0$ restricted to the cylinders $\wt\Lambda_0^j$. For the $3:1$ resonance, we consider outer maps $\FF_0^{\out,*}$, $*=\ff,\bb$, which 
connect $\wt\Lambda_0^1$ to itself and are defined as
\[
 \begin{split}
  \FF_0^{\out,\ff}=\PP_0^2\circ\SSS^\ff: \wt\Lambda_0^1\longrightarrow \wt\Lambda_0^1\\
  \FF_0^{\out,\bb}=\SSS^\bb\circ\PP_0: \wt\Lambda_0^1\longrightarrow \wt\Lambda_0^1,
 \end{split}
\]
where $\SSS^\ff$ is the outer map which connects $\wt\Lambda_0^1$ and
$\wt\Lambda_0^2$ through $W^u(\wt\Lambda_0^1)\cap W^s(\wt\Lambda_0^2)$
and $\SSS^\bb$ is the outer map which connects $\wt\Lambda_0^2$ and
$\wt\Lambda_0^1$ through $W^u(\wt\Lambda_0^2)\cap
W^s(\wt\Lambda_0^1)$. Recall that  we are abusing notation since the
forward and backwards outer maps are only defined provided $I\in
\DD^\ff$ and $I\in\DD^\bb$ respectively and not in the whole cylinder
$\wt\Lambda_0^1$.

As for $1:7$ case, these maps are of the form
\begin{equation}\label{def:OuterMap:Circular:2}
  \FF_0^{\out,\ast}:\left(\begin{array}{c} I\\
      t
    \end{array}\right)\mapsto \left(\begin{array}{c} I\\
      t+\mu\omega^\ast(I)
    \end{array}\right),\,\,\,\ast=\ff,\bb.
\end{equation}
Since we want to compute these outer maps using flows, we need to reparameterize time in the vector field associated to the Hamiltonian
\eqref{def:HamDelaunayRot} with $e_0=0$, so that it preserves the section
$\{\ell=0\}$. We consider the following vector field, which  corresponds to identifying the variable $\ell$ with
time,
\begin{equation}\label{def:Reduced:ODE:Circ:2}
  \begin{array}{rlcrl}
    \dps\frac{d}{ds} \ell=&1&\text{   }&\dps\frac{d}{ds} L=&\dps-\frac{\pa_\ell
      H}{L^{-3}+\mu\pa_L\Delta H_\ccirc}\\
    \dps\frac{d}{ds} g=&\dps\frac{-1+\mu\pa_G\Delta H_\ccirc}{L^{-3}+\mu\pa_L\Delta H_\ccirc}&\text{   }&\dps\frac{d}{ds} G=&\dps-\frac{\pa_g
      H}{L^{-3}+\mu\pa_L\Delta H_\ccirc}\\
    \dps\frac{d}{ds} t=&\dps\frac{1}{L^{-3}+\mu\pa_L\Delta H_\ccirc}&\text{
    }&\dps\frac{d}{ds} I=&\dps0
  \end{array}
\end{equation}
where $H$ is Hamiltonian \eqref{def:HamDelaunayRot} with $e_0=0$.
Notice that now we are not changing time direction, as happened in the $1:7$ resonance. We  refer to this system as a \emph{reduced circular problem}. Recall that we denote by $\Phi_0^\ccirc$ the flow associated to the $(L,\ell,G,g)$ components of equation \eqref{def:Reduced:ODE:Circ} (which are independent of $t$ and $I$). We use it to derive the formulas for the outer map. Let
\begin{equation}\label{def:OrbitsForOuterMap:2}
\begin{split}
\gamma_I^\ast(\sigma) &=
\Phi^\ccirc_0\{\sigma,(\CCC_0^{\ast,L}(I),0,\CCC_0^{\ast,G}(I),\CCC_0^{\ast,g}(I))\},\,\,\,\ast=\ff,\bb\\
\la_I^j(\sigma) &=
\Phi^\ccirc_0\{\sigma,(\GG_0^{j,L}(I),0,\GG_0^{j,G}(I),\GG_0^{j,g}(I))\}\\
\end{split}
\end{equation}
be trajectories of the circular problem. Then, one can see that the functions $\omega^{\ff,\bb}(I)$ involved in the definition of the outer
  maps in \eqref{def:OuterMap:Circular:2} can be defined as
  \[
  \omega^\ast(I)= \omega^\ast_\out(I)+\omega_\inn^\ast (I),
  \]
where
\begin{equation}\label{def:Omega0:OuterPart:2}
 \omega^\ast_\out(I)=\omega_+^\ast(I)-\omega^\ast_-(I)
\end{equation}
  with
  \begin{equation}\label{def:Omega0PlusMinus:2}
 \begin{split}
 \omega_+^\ast(I)&=\lim_{N\rightarrow+\infty}\left(\int_0^{ 6N\pi
      }\frac{\left(\mu\ii(3-L^{-3})-\pa_L\Delta
        H_\ccirc\right) \circ \gamma_I^\ast(\sigma)}{3(L^{-3}+\mu\pa_L\Delta
        H_\ccirc) \circ \gamma_I^\ast(\sigma)} \, d\sigma+ N\TTT_0(I)\right)\\
\omega^\ast_-(I)&=\lim_{N\rightarrow-\infty}\left(\int_0^{6N\pi
      }\frac{\left(\mu\ii(3-L^{-3})-\pa_L\Delta
        H_\ccirc\right) \circ \gamma_I^\ast(\sigma)}{3(L^{-3}+\mu\pa_L\Delta
        H_\ccirc) \circ \gamma_I^\ast(\sigma)} \, d\sigma+ N\TTT_0(I)\right),\,\,\,\ast=\ff,\bb
\end{split}
\end{equation}
and
\begin{equation}\label{def:Omega0:InnerPart:2}
\begin{split}
 \omega_\inn^\ff(I)&=\int_0^{ 4\pi
      }\frac{\left(\mu\ii(3-L^{-3})-\pa_L\Delta
        H_\ccirc\right)  \circ \la_I^2(\sigma)}{3(L^{-3}+\mu\pa_L\Delta
        H_\ccirc) \circ \la_I^2(\sigma)} \, d\sigma\\
\omega_\inn^\bb(I)&=\int_0^{ 2\pi
      }\frac{\left(\mu\ii(3-L^{-3})-\pa_L\Delta
        H_\ccirc\right) \circ \la_I^1(\sigma)}{3(L^{-3}+\mu\pa_L\Delta
        H_\ccirc) \circ  \la_I^1(\sigma)} \, d\sigma,
\end{split}
\end{equation}
where  $\TTT_0(I)$ is the function in \eqref{def:InnerMap:Circular:2}. Recall that along the periodic and homoclinic orbits $(3-L^{-3})\sim\mu$.

\subsection{The elliptic problem}\label{sec:Elliptic:2}
We study now the elliptic problem. Reasoning as for the $1:7$ resonance, for $e_0$ small enough the system associated to the Hamiltonian
\eqref{def:HamDelaunayRot} has a normally hyperbolic invariant
cylinder $\Lambda_{e_0}$, which is $e_0$-close to $\Lambda_0$ given in
Corollary \ref{coro:NHIMCircular}. Analogously, the Poincar{\'e} map
\[  \PP_{e_0}:\{\ell=0\}\longrightarrow \{\ell=0\}
\]
associated to the flow of \eqref{def:HamDelaunayRot} and the section $\{\ell=0\}$ has a
normally hyperbolic invariant cylinder
$\wt\Lambda_{e_0}=\Lambda_{e_0}\cap \{\ell=0\}$. Moreover, it is formed
by three connected components $\wt\Lambda_{e_0}^j$, $j=0,1,2$,
which are $e_0$-close to the cylinders $\wt\Lambda^j_{e_0}$ obtained
in Corollary \ref{coro:NHIMCircular:Poincare:2} and have natural coordinates $(I,t)$ as happened for the circular case.

We look for perturbative expansions of the inner and outer maps. For the $3:1$ resonances they are computed using the new reduced elliptic
problem
\begin{equation}\label{def:Reduced:ODE:2}
  \begin{array}{rlcrl}
    \frac{d}{ds} \ell=&1&\text{   }&\frac{d}{ds} L=&\dps-\frac{\pa_\ell H}{L^{-3}+\mu\pa_L\Delta H_\ccirc +\mu e_0\pa_L\Delta H_\eell}\\
    \frac{d}{ds} g=&\dps\frac{\pa_G H}{L^{-3}+\mu\pa_L\Delta H_\ccirc +\mu e_0\pa_L\Delta H_\eell}&\text{   }&\frac{d}{ds} G=&\dps-\frac{\pa_g H}{L^{-3}+\mu\pa_L\Delta H_\ccirc +\mu e_0\pa_L\Delta H_\eell}\\
    \frac{d}{ds} t=&\dps\frac{1}{L^{-3}+\mu\pa_L\Delta H_\ccirc +\mu e_0\pa_L\Delta H_\eell}&\text{   }&\frac{d}{ds} I=&\dps-\frac{\mu e_0\pa_t \Delta H_\eell}{L^{-3}+\mu\pa_L\Delta H_\ccirc +\mu e_0\pa_L\Delta H_\eell},
  \end{array}
\end{equation}
which is  a perturbation of \eqref{def:Reduced:ODE:Circ:2}.  

For the elliptic problem, the coordinates $(I,t)$ are symplectic not
with respect to the canonical symplectic form $dI\wedge dt$ but whith respect to a 
symplectic form
\begin{equation}\label{def:InnerDiffForm:Ell:2}
 \Omega^j_{e_0}=\left(1+e_0 a^j_1(I,t)+e_0^2a^j_2 (I,t)+e_0^3
   a^j_\geq(I,t)\right)dI\wedge dt,
\end{equation}
with certain functions $a_k^j:[I_-,I_+]\times\TT\rightarrow \RR$ which satisfy
\[
 \NNN\left(a_1^3\right)=\{\pm 1\}\,\,\,\text{ and }\,\,\,\NNN\left(a_2^3\right)=\{0,\pm 1,\pm 2\},
\]
(see \eqref{eq:Harmonics:change:f2} for the definition of $\NNN$ and  Corollary \ref{coro:SymplecticForm} for the corresponding result for the $1:7$ resonance).

In the invariant cylinder $\wt\Lambda_{e_0}^1$, one can define inner and outer
maps as we have done in $\wt\Lambda_0^1$ for the circular problem. We proceed as in Section \ref{sec:Elliptic} for the $1:7$ resonance.

\paragraph{The inner map}
We study first the inner map. As for the circular problem, it is
defined the map $\PP^3_{e_0}$ in \eqref{def:Elliptic:Poincare}
restricted to the normally hyperbolic invariant manifold
$\wt\Lambda^1_{e_0}$. For $e_0$ small enough, proceeding as in the proof of Lemma \ref{lemma:InnerMap:Elliptic}, one can see that it is of the form
    \begin{equation}\label{def:InnerMap:Elliptic:2}
      \FF_{e_0}^\inn:\left(\begin{array}{c} I\\
          t
        \end{array}\right)\mapsto \left(\begin{array}{l} I+ e_0 A_1(I,t)+e_0^2 A_2(I,t)+\OO\left(e_0^3\right)\\
          t+\mu\TTT_0(I)+e_0 \TTT_1(I,t)+e_0^2 \TTT_2(I,t)+\OO\left( e_0^3\right)
        \end{array}\right),
    \end{equation}
        with functions $A_1$, $A_2,$ $\TTT_1,$ and $\TTT_2$ satisfying
   \begin{align}
      \NNN\left(A_1\right)&=\{\pm 1\},\,\,\, \NNN\left(A_2\right)=\{0,\pm 1,\pm 2\}\label{eq:InnerMap:I:1,2:Harmonics:0:2}\\
      \NNN\left(\TTT_1\right)&=\{\pm 1\},\,\,\, \NNN\left(\TTT_2\right)=\{0,\pm 1,\pm 2\}.\label{eq:InnerMap:t:1,2:Harmonics:0:2}
\end{align}    
 Thus, $A_1$ can be split as,
  \[
  A_1(I,t)=A_1^+(I)e^{it}+A_1^-(I)e^{-it}.
  \]
Moreover, the Fourier coefficients are defined as
  \[
    A_1^\pm(I)=\mp i\mu \int_0^{6\pi}\frac{\Delta H_{\eell}^{1,\pm}\circ \la_I^1(\sigma)}{L^{-3}+\mu \pa_G\Delta H_\ccirc\circ\la_I^1(\sigma)}e^{\pm i\wt\la_I^1(\sigma)}d\sigma,
  \]
  where the functions $\Delta H_{\eell}^{1,\pm}$ are defined as
  \[
  \Delta H_{\eell}^1(L,\ell,G,g,t)=\Delta H_{\eell}^{1,+}(L,\ell,G,g)e^{it}+\Delta H_{\eell}^{1,\pm}(L,\ell,G,g)e^{-it},
  \]
 and $\la_I^1(\sigma)$ has been defined in \eqref{def:OrbitsForOuterMap:2}. Finally, $ \wt\la_I^1(\sigma)$ is defined as 
\begin{equation}\label{def:Orbit:Gamma3:time:2}
 \wt\la_I^1(\sigma)=
\wt\Phi_0\{\sigma,(\GG_0^{1,L}(I),\GG_0^{1,\ell}(I),\GG_0^{1,G}(I),0)\},
\end{equation}
where $\GG^3_0$ has been introduced in Corollary \ref{coro:NHIMCircular:Poincare:2} and 
\begin{equation}\label{def:Flow:t:2}
    \wt\Phi_0 \{s,(L,\ell,G,g)\}=t+\int_0^s \frac{1}{L^{-3}+\mu\pa_L\Delta
        H_\ccirc\left(\Phi^\ccirc_0\{\sigma,(L,\ell,G,g)\}\right)} \,
      d\sigma.
\end{equation}
The function $\wt\Phi_0 $ is analogous to the corresponding function for the $1:7$ resonance, defined in \eqref{eq:Flow:Circ:Time}.

\paragraph{The outer map}
We study now the outer maps
\begin{equation}\label{def:OuterMap:Elliptic:0:2}
  \FF_{e_0}^{\out,\ast}:\wt\Lambda^1_{e_0}\longrightarrow \wt\Lambda^1_{e_0},\,\,\,\ast=\ff,\bb
\end{equation}
for $e_0>0$. Thanks to Ansatz \ref{ans:NHIMCircular:2}, we know that for $e_0$ small enough, there exist transversal intersections of the
invariant manifolds of $\wt\Lambda^1_{e_0}$ and $\wt\Lambda^2_{e_0}$.
Thus, we can proceed as in Section \ref{sec:Circular:Outer} to define
the outer maps $\FF_{e_0}^\out$ for the $3:1$ resonance and we study them as a
perturbation of the outer maps of the circular problem given in
\eqref{def:OuterMap:Circular:2}. We use the reduced elliptic
problem defined in \eqref{def:Reduced:ODE:2} and we compute their first order. To this end, we use the
notation $\gamma_I^{\ff,\bb}(\sigma)$ and $\la_I^{1,2}(\sigma)$ defined in \eqref{def:OrbitsForOuterMap:2}. Analogously we define their corresponding $t$-component of the flow as
\begin{equation}\label{def:OrbitsForOuterMap:t:2}
\begin{split}
\wt\gamma_I^\ast(\sigma) &=
\wt\Phi_0\{\sigma,(\CCC_0^{\ast,L}(I),\CCC_0^{\ast,\ell}(I),\CCC_0^{\ast,G}(I),0)\},\,\,\,\ast=\ff,\bb\\
\wt\la_I^j(\sigma) &=
\wt\Phi_0\{\sigma,(\GG_0^{j,L}(I),\GG_0^{j,\ell}(I),\GG_0^{j,G}(I),0)\},\,\,\, j=1,2\\
\end{split}
\end{equation}
where  $\CCC_0^\ast$ and $\GG_0^j$ have been given in Corollary \ref{coro:NHIMCircular:Poincare:2} and $\wt\Phi_0$ in \eqref{def:Flow:t:2}.


\begin{lemma}\label{lemma:Outer:Elliptic:2}
  The outer map defined in \eqref{def:OuterMap:Elliptic:0:2} has the following expansion with respect to $e_0$,
  \begin{equation}\label{def:OuterMap:Elliptic:2}
    \FF_{e_0}^{\out,\ast}:\left(\begin{array}{c} I\\
        t
      \end{array}\right)\mapsto \left(\begin{array}{c} I+ e_0\left(B^{\ast,+} (I)e^{it}+B^{\ast,-} (I)e^{-it}\right) +\OO\left(e_0^2\right)\\
        t+\mu\omega^{\ast}(I)+\OO(e_0)
      \end{array}\right),\,\,\,\ast=\ff,\bb.
  \end{equation}
  Moreover, the functions $B^{\ast,\pm}(I)$ can be defined as
\begin{equation}\label{def:Omega:PlusMinus:2}
\begin{split}
B^{\ff,\pm}(I)&=B_\out^{\ff,\pm}(I)+B_\inn^{\ff,\pm}(I)e^{\pm i\mu\omega_\out^\ff(I)}\\
B^{\bb,\pm}(I)&=B_\inn^{\bb,\pm}(I)+B_\out^{\bb,\pm}(I)e^{\pm i\mu\omega_\inn^\bb(I)},
\end{split}
\end{equation}
where $\omega_\out^\ff(I)$ and $\omega_\inn^\bb(I)$ are the functions defined in \eqref{def:Omega0:OuterPart} and \eqref{def:Omega0:InnerPart} respectively and
  \begin{align}
    B^{\ff,\pm}_\out(I)=&\pm i\mu\lim_{T\rightarrow+\infty}\int_0^T\left(\frac{\Delta H^{1,\pm}_\eell\circ\gamma_I^\ff(\sigma)}{L^{-3}+\mu\pa_L\Delta H_\ccirc\circ\gamma_I^\ff(\sigma)}e^{\pm i\wt\gamma_I^\ff(\sigma)}\right.\nonumber\\
    &\qquad\qquad\qquad\left.-\frac{\Delta H_\eell^{1,\pm}\circ\la_I^1(\sigma)}{L^{-3}+\mu\pa_L\Delta H_\ccirc\circ\la_I^1(\sigma)}e^{\pm i\left(\wt\la_I^1(\sigma)+\mu\omega_+^\ff(I)\right)}\right) d\sigma\\
    &\mp i\mu\lim_{T\rightarrow-\infty}\int_0^T\left(\frac{\Delta H_\eell^{1,\pm}\circ\gamma_I^\ff(\sigma)}{L^{-3}+\mu\pa_L\Delta H_\ccirc\circ\gamma_I^\ff(\sigma)}e^{\pm i\wt\gamma_I^\ff(\sigma)}\right.\nonumber\\
    &\qquad\qquad\qquad\left.-\frac{\Delta H_\eell^{1,\pm}\circ\la_I^2(\sigma)}{L^{-3}+\mu\pa_L\Delta H_\ccirc\circ\la_I^2(\sigma)}e^{\pm i\left(\wt\la_I^2(\sigma)+\mu\omega_-^\ff(I)\right)}\right)d\sigma,\nonumber
\end{align}
\begin{align}
  B_\out^{\bb,\pm}(I)=&\pm i\mu\lim_{T\rightarrow+\infty}\int_0^T\left(\frac{ \Delta H^{1,\pm}_\eell\circ\gamma_I^\bb(\sigma)}{L^{-3}+\mu\pa_L\Delta H_\ccirc\circ\gamma_I^\bb(\sigma)}e^{\pm i\wt\gamma_I^\bb(\sigma)}\right.\nonumber\\
    &\qquad\qquad\qquad-\left.\frac{\Delta H_\eell^{1,\pm}\circ\la_I^2(\sigma)}{L^{-3}+\mu\pa_L\Delta H_\ccirc\circ\la_I^2(\sigma)}e^{\pm i\left(\wt\la_I^2(\sigma)+\mu\omega_+^\bb(I)\right)}\right)d\sigma\\
    &\mp i\mu\lim_{T\rightarrow-\infty}\int_0^T\left(\frac{ \Delta H_\eell^{1,\pm}\circ\gamma_I^\bb(\sigma)}{L^{-3}+\mu\pa_L\Delta H_\ccirc\circ\gamma_I^\bb(\sigma)}e^{\pm i\wt\gamma_I^\bb(\sigma)}\right.\nonumber\\
    &\qquad\qquad\qquad-\left.\frac{ \Delta H_\eell^{1,\pm}\circ\la_I^1(\sigma)}{L^{-3}+\mu\pa_L\Delta H_\ccirc\circ\la_I^1(\sigma)}e^{\pm i\left(\wt\la_I^1(\sigma)+\mu\omega_-^\bb(I)\right)}\right)d\sigma,\nonumber
\end{align}
\begin{align}
  B_\inn^{\ff,\pm}(I)=&\mp i\mu\int_0^{4\pi}\frac{\Delta H^{1,\pm}_\eell\circ\la_I^2(\sigma)}{L^{-3}+\mu\pa_L\Delta H_\ccirc\circ\la_I^2(\sigma)}e^{\pm i\wt\la_I^2(\sigma)}d\sigma\\
B_\inn^{\bb,\pm}(I)=&\mp\int_0^{2\pi}\frac{\Delta H^{1,\pm}_\eell\circ\la_I^1(\sigma)}{L^{-3}+\mu\pa_L\Delta H_\ccirc\circ\la_I^1(\sigma)}e^{\pm i\wt\la_I^1(\sigma)}d\sigma\nonumber
   \end{align}
  where
  \[
  \Delta H_\eell^{1,\pm}(\ell,L,g,G,t)=\Delta H_\eell^{1,\pm}(\ell,L,g,G)e^{it}+\Delta H_\eell^{1,\pm}(\ell,L,g,G)e^{-it}
  \]
  has been defined in Corollary \ref{coro:ExpansioHamiltonian} and $\omega_\pm^\ast$ have been defined in \eqref{def:Omega0PlusMinus:2}.
\end{lemma}

\paragraph{Existence of diffusing orbits} 
The last step to prove the existence of diffusing orbits can be done analogously to what has been done in Section \ref{sec:ProofDiffusion} for the $1:7$ resonance. Namely, we just need to obtain a change of coordinates  $(I,t)=(\II,\tau)+e_0\varphi (\II,\tau)$ which 
\begin{enumerate}
 \item Straightens the symplectic form $ \Omega^1_{e_0}$ (see \eqref{def:InnerDiffForm:Ell:2}) into $\Omega_0=d\II\wedge d\tau$.
\item Flattens the inner map in the $I$-direction.
\end{enumerate}
This is summarized in the next lemma, which merges the corresponding  Lemmas \ref{lemma:StraighteningOmega} and \ref{lemma:Averaging} for the $1:7$ resonance.

\begin{lemma}\label{lemma:Averaging:2}
There exists a $e_0$-close to the identity  change of variables
\[
\left(I,t\right)=(\II,\tau)+e_0\varphi (\II,\tau) 
\]
defined on $\wt \Lambda^1_{e_0}$, which:
\begin{itemize}
\item  Transforms the symplectic form $\Omega^1_{e_0}$  into the canonical form
$ \Omega_0=d\II\wedge d\tau$.
\item Transforms the inner map $\FF_{e_0}^{\inn}$ in \eqref{def:InnerMap:Elliptic:2} into
    \begin{equation}\label{def:InnerMap:ell:modified:2}
      \wt\FF_{e_0}^\inn:\left(\begin{array}{c} \II\\
          \tau
        \end{array}\right)\mapsto \left(\begin{array}{c}  \II+\OO\left(\mu e_0^3\right)\\
          \tau+\mu\TTT_0\left(\II\right)+e_0^2 \wt \TTT_2\left(\II\right)+\OO\left(\mu e_0^3\right).
        \end{array}\right)
    \end{equation}
\item Transforms the outer maps $\FF_{e_0}^{\out,\ff}$ and $\FF_{e_0}^{\out,\bb}$ in \eqref{def:OuterMap:Elliptic:2} into
    \begin{equation}\label{def:OuterMap:ell:modified:2}
      \wt\FF_{e_0}^{\out,\ast}:\left(\begin{array}{c} \II\\
          \tau
        \end{array}\right)\mapsto \left(\begin{array}{c}  \II+e_0 \wt B^\ast (\II, \tau)+\OO\left(\mu e_0^2\right)\\
          \tau+\mu\omega^\ast(\II)+\OO(\mu e_0)
        \end{array}\right),\,\,\,\ast=\ff,\bb,
    \end{equation}
    where
    \[
    \wt B^\ast \left(\II, \tau\right)=\wt B^{\ast,+} \left(\II\right) e^{i\tau}+\wt B^{\ast,-} \left(\II\right) e^{-i\tau}
    \]
    with
    \[
    \wt B^{\ast,\pm} \left(\II\right)=B^{\ast,\pm} \left(\II\right)-\frac{e^{\pm i\mu\omega^\ast(\II)}-1}{e^{\pm i\mu\TTT_0(\II)}-1}A_1^\pm \left(\II\right).
    \]
\end{itemize}
  \end{lemma}
To be able to ensure the existence of transition chains of tori, we need to assume the following ansatz.
\begin{ansatz}\label{ans:B:2}
The functions $\wt B^{\ast,\pm}$ defined in Lemma \ref{lemma:Averaging:2} satisfy
\[      \wt B^{\ast,\pm} \left(\II\right)\neq 0\qquad\text{for }\II\in \DD^\ast,
  \]
where $\DD^\ast$ are the domains considered in Corollary \ref{coro:NHIMCircular:Poincare:2}. 
\end{ansatz}

With this ansatz, and also Ans\"atze \ref{ans:NHIMCircular:2} and \ref{ans:TwistInner:2}, we can proceed as in Section \ref{sec:ProofDiffusion} to prove the existence of a transition chain of tori and of orbits shadowing such chain.

\subsection{Numerical study of the $3:1$ resonance}
\label{app:3:1:numerics}

In this section, we briefly describe our numerical analysis of the $3:1$
resonance. In particular, we verify Ans{\"a}tze~\ref{ans:NHIMCircular:2},
\ref{ans:TwistInner:2} and~\ref{ans:B:2} numerically.

The numerical methodology used for the $3:1$ resonance is analogous to the
$1:7$ resonance. Cartesian rotating coordinates are used for the computation
of the hyperbolic structure of the circular problem (normally hyperbolic
invariant cylinder, stable and unstable manifolds, and their homoclinic
intersection). We now consider the Poincar\'e section
\[ \tilde\Sigma^+=\{y=0,\ \dot y>0\}, \]
and the  associated $2$-dimensional symplectic Poincar\'e map $P\colon\
\tilde\Sigma^+\to\tilde\Sigma^+$ acting on $(x,p_x)$.
We now look for $3:1$ resonant periodic orbits as $2$-periodic points of the
Poincar\'e map, i.e. letting $p=(x,p_x)$, we need to solve the equation
\[ P^2(p)=p. \]
In fact, exploiting the symmetry of the problem, it is enough to use
$1$-dimensional root finding:
\[ \pi_{p_x}(P^2(p))=0, \]
since we impose that the point $p$ lies on the symmetry section $\{y=0,\
p_x=0\}$.

\begin{figure}[h]
\begin{center}
\psfrag{H}{$J$}
\psfrag{T}{$T_J - 2\pi$}
\psfrag{L}{$L_{\max}$}
\psfrag{T2XXXXXXX}{$T_J - 2\pi$}
\psfrag{L2XXXXXXX}{$L_{\max}$}
\includegraphics[width=10cm]{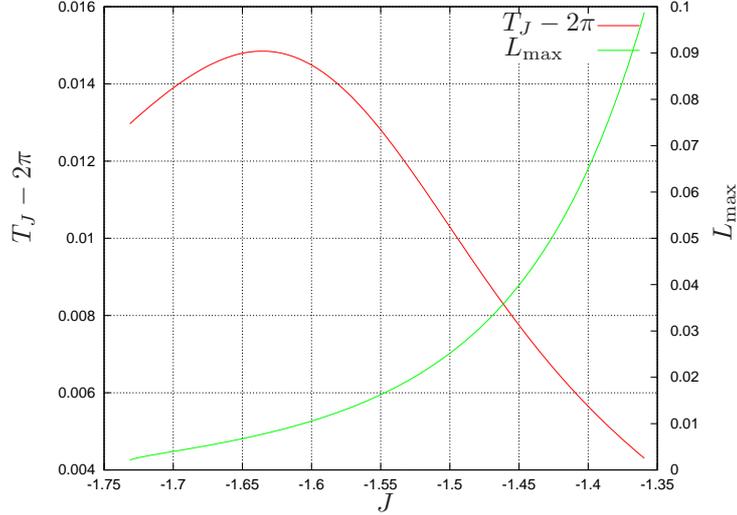}
\end{center}
\caption{Resonant family of periodic orbits.
We show normalized period $T_J - 2\pi$, and maximum deviation of $L$
component with respect to the resonant value $3^{-1/3}$ (see
equation~\eqref{eq:Ldeviation:2}).}\label{fig:porbits:2}
\end{figure}

Thus we obtain the family of resonant periodic orbits for energy levels 
\[ J\in [\bar J_-, \bar J_+] = [-1.7314, -1.3594].\]
See Figure~\ref{fig:porbits:2}. 
Notice that the period $T_J$ stays close to the resonant period $2\pi$
of the unperturbed system. 
From Figure~\ref{fig:porbits:2}, we obtain
the bound 
\[ |T_J - 2\pi| < 15\mu, \]
which is the first bound given in Ansatz~\ref{ans:NHIMCircular:2}.

\begin{figure}[h]
\begin{center}
\psfrag{H}{$J$}\psfrag{lu}{$\ln(\lambda)$}
\includegraphics[width=10cm]{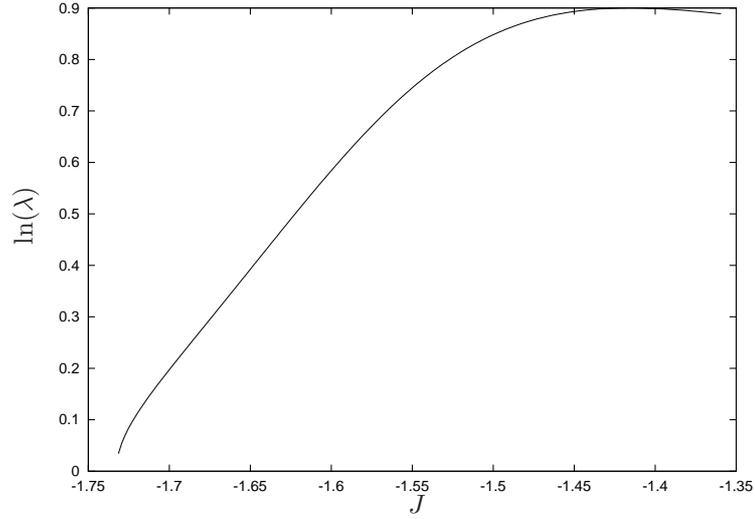}
\end{center}
\caption{Characteristic exponent $\ln(\lambda)$ as a function of energy level $J$ (the other exponent is $-\ln(\lambda)$).}
\label{fig:hypers:2}
\end{figure}

Furthermore, we verify that (the square of) the semi-major axis $L$ stays
close to the resonant value $3^{-1/3}$.
Integrating the periodic orbit in Delaunay coordinates
$\la_J(t)=(L_J(t),\ell_J(t),G_J(t),g_J(t))$ over one period $T_J$, we compute
the quantity
\begin{equation} \label{eq:Ldeviation:2}
 L_{\max}(J) = \max_{t \in [0,T_J)} |L_J(t)-3^{-1/3}|.
\end{equation}
The function $L_{\max}(J)$ is plotted in Figure~\ref{fig:porbits:2}.
Notice that we obtain the bound
\[ |L_J(t)-3^{-1/3}| < 100\mu \]
for all $t\in \RR$ and $J \in [\bar J_-, \bar J_+]$, which is the second bound
given in Ansatz~\ref{ans:NHIMCircular:2}.

To determine the stability of the periodic orbits, we now compute the
eigenvalues $\lambda$ and $\lambda^{-1}$ of $DP^2(p)$.
Figure~\ref{fig:hypers:2} shows the characteristic exponents
$\ln(\lambda)$, $\ln(\lambda^{-1})$ as a function of energy.
The family of periodic orbits is hyperbolic in the interval $[\bar J_-,\bar
J_+]$, although the strength of hyperbolicity is weaker than in the $1:7$
resonance. Compare with Figure~\ref{fig:hypers}.

\begin{figure}[h]
\begin{center}
\psfrag{x}{$x$}
\psfrag{px}{$p_x$}
\psfrag{p0}{$p_0$}
\psfrag{p1}{$p_1$}
\psfrag{z1}{$z_1$}
\psfrag{z2}{$z_2$}
\includegraphics[width=10cm]{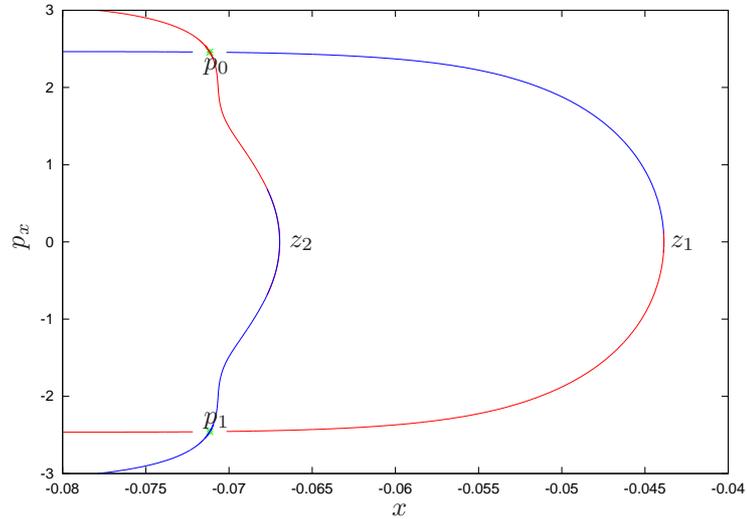}
\end{center}
\caption{Invariant manifolds of the fixed points $p_0$ and $p_1$ for energy
level $J=-1.3594$.}
\label{fig:invmfld2:2}
\end{figure}

The stable and unstable invariant manifolds of the periodic orbits are
computed using the same methodology explained for the $1:7$ resonance. In
particular, we switch to the new Poincar\'e section
\[ \tilde\Sigma^-=\{y=0,\ \dot y<0\} \]
in order to have the homoclinic points lying on the symmetry axis. 
For illustration, we show the result corresponding to the energy value
$J=-1.3594$ in Figure~\ref{fig:invmfld2:2}. The manifolds intersect
transversally at the homoclinic points $z_1$ (outer splitting) and $z_2$
(inner splitting), as we will show below.

\begin{figure}[h]
\psfrag{H}{$J$}
\psfrag{s}{$\sigma$ (radians)}
\begin{center}
\includegraphics[width=10cm]{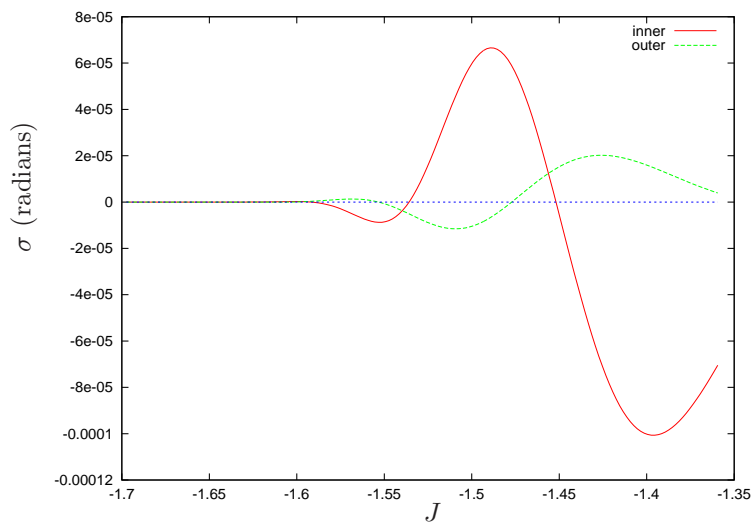} 
\end{center}
\caption{Splitting angle associated to inner and outer splitting.}
\label{fig:splittings:2}
\end{figure}

Next we compute the splitting angle between the invariant manifolds at the
homoclinic points. We will restrict the range of energy values to
\begin{equation}
\label{def:EnergyRange:Splitting:2}
J\in[J_-,J_+]=[\Jminusnew,\Jplusnew],
\end{equation}
or equivalently the range of eccentricities to
$e\in[e_-,e_+]=[\eminusnew,\eplusnew]$.
This is the range where we can validate the accuracy of our
computations (see Appendix~\ref{sec:accuracy_computations}).
Below $e_-=\eminusnew$, the splitting size becomes comparable to the
numerical error that we commit in double precision arithmetic.

\begin{remark}
In contrast with the $1:7$ resonance, now the manifolds stay close to the
integrable situation \emph{for the whole range of energies}, i.e. they meet
with small splitting angle as we will show below. Thus for the $3:1$
resonance there is no difficulty in identifying the \emph{primary} family of
homoclinic points. Compare with Remark~\ref{rem:primary_intersection}.
\end{remark}

\begin{table}
\begin{center}
\begin{tabular}{|c|c|}
\hline
inner & outer \\
\hline
$ ( -1.453,-1.451 ) $ & $ ( -1.477,-1.475 ) $ \\
$ ( -1.537,-1.535 ) $ & $( -1.553,-1.551 ) $ \\
$ ( -1.593,-1.591 ) $ & \\
\hline
\end{tabular}
\end{center}
\caption{Subintervals of $J\in[J_-,J_+]$ containing the zeros of inner
splitting (left column) and outer splitting (right column).}
\label{tab:zeros_inner_outer:2}
\end{table}

Using the same methodology as for the $1:7$ resonance, we are able to obtain
the splitting angle for energy levels $J\in[J_-,J_+]$. See
Figure~\ref{fig:splittings:2}.
Numerically, we find that the zeros of the splitting angle are
contained in the intervals listed in
Table~\ref{tab:zeros_inner_outer:2}.
As seen from the table, the inner and outer splittings become zero at
different values of $J$.
Thus, when one of the intersections becomes tangent, the other one is
still transversal, and we can always use one of them for diffusion.

\begin{figure}[h]
\begin{center}
\psfrag{g}{$g$}
\psfrag{G}{$G$}
\psfrag{p0}{$p_0$}
\psfrag{p1}{$p_1$}
\psfrag{p2}{$p_2$}
\psfrag{z1}{$z_1$}
\psfrag{z2}{$z_2$}
\includegraphics[width=10cm]{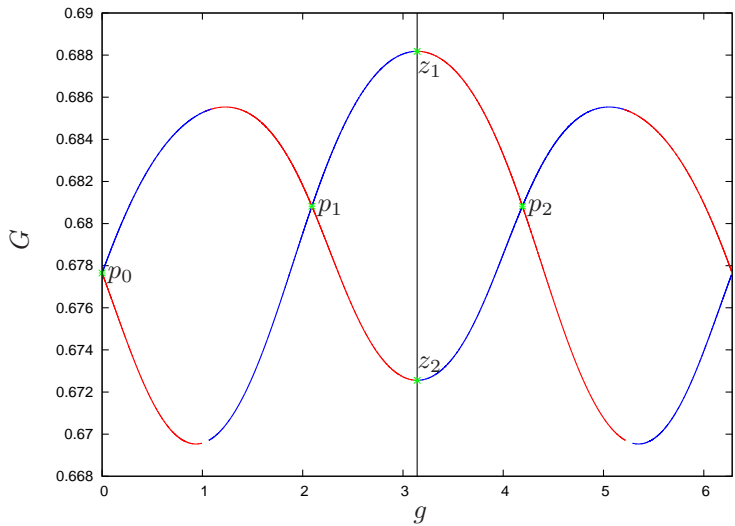}
\end{center}
\caption{Energy $J=-1.7194$. Resonance structure in Delaunay coordinates. The
symmetry corresponds to $g=0$ and $g=\pi$ and is marked with a vertical
line.}\label{fig:resonancedel:2}
\end{figure}

Again, we check the validity of $\sigma(J_-)$ by computing this splitting
angle using two different numerical methods and comparing the results. They
differ by less than $10^{-10}$, which gives an estimate of the total
numerical error.

Recall that the study of the inner and outer maps is done in rotating
Delaunay coordinates. As explained in Appendix~\ref{App:Resonance1:3}, for
the analysis of the $3:1$ resonance it is convenient to consider the
Poincar\'e section $\{\ell=0\}$.
Thus, we transform the hyperbolic structure of the
circular problem from Cartesian to Delaunay coordinates as explained in
Appendix~\ref{sec:RotatingToDelaunay}. See
Figure~\ref{fig:resonancedel:2}.

First we compute the inner map $\FF_0^\inn$ and the outer maps
$\FF_0^{\out,\ast}$ of the circular problem, given in
Appendix~\ref{sec:Circular:2}. 
We consider $I\in [I_-, I_+]=[- J_+,- J_-]$, where the range $[- J_+,- J_-]$
is given in \eqref{def:EnergyRange:Splitting:2}.
For the inner map, Figure~\ref{fig:porbits:2} shows a plot of the function
$T_J - 2\pi = \mu\TTT_0(I)$.
Notice that the derivative of the function $\TTT_0(I)$ is nonzero for
the whole range $[I_-,  I_+]$.
This shows that the inner map is twist.
Moreover, Figure~\ref{fig:porbits:2} shows that
\[ 0<\mu\TTT_0(I)<15\mu<\pi. \]
Therefore, the function $\TTT_0(I)$ satisfies the properties
stated in Ansatz~\ref{ans:TwistInner:2}.

Then we compute the first orders in $e_0$ of the inner map
$\FF_{e_0}^\inn$ and the outer maps $\FF_{e_0}^{\out,\ast}$ of the elliptic
problem, given in Appendix~\ref{sec:Elliptic:2}.
For brevity, we do not show the results here, since the plot of the functions
$A_1^+$, $B^{\ff,+}$ and $B^{\bb,+}$ does not convey much information.

\begin{figure}[h]
\begin{center}
\psfrag{H}{$J$}
\psfrag{reBfXXXXXX}{$\Re(\wt B^{\ff,+})$}
\psfrag{imBfXXXXXX}{$\Im(\wt B^{\ff,+})$}
\psfrag{reBbXXXXXX}{$\Re(\wt B^{\bb,+})$}
\psfrag{imBbXXXXXX}{$\Im(\wt B^{\bb,+})$}
\includegraphics[width=10cm]{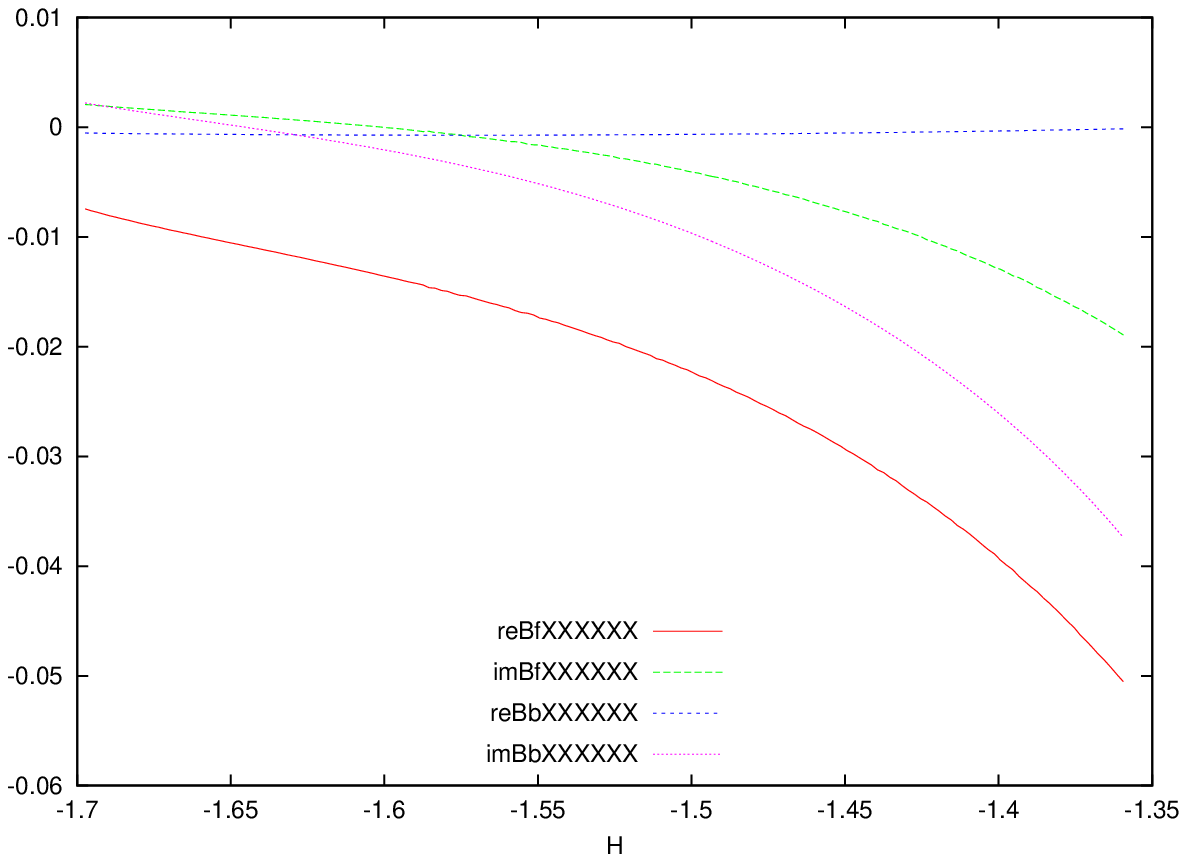}
\end{center}
\caption{Functions $\wt B^{\ff,+}$ and $\wt B^{\bb,+}$ (real and imaginary
parts).}
\label{fig:tildeB:2}
\end{figure}

Finally, we verify the non-degeneracy condition
\[      \wt B^{\ast,\pm} \left(\II\right)\neq 0\qquad\text{for }\II\in
\DD^\ast,
  \]
stated in Ansatz~\ref{ans:B:2}, which implies the existence of a transition
chain of tori.
The computed values of the functions $\wt B^{\ff,+}$ and $\wt B^{\bb,+}$ are
shown in Figure~\ref{fig:tildeB:2}.
Therefore, we see that the functions $\wt B^{\ast,+}$ are not
identically zero.
This justifies Ansatz \ref{ans:B:2}.

\section{Conjectures on the speed of diffusion}\label{speed}

Instabilities for nearly integrable systems are often called 
{\it Arnol'd diffusion}. As far as we know, this term was coined 
by Chirikov \cite{Chirikov:1979}. In this section we state two conjectures
about random behavior of orbits near resonances, where randomness
is coming from initial condition. 

A nearly integrable Hamiltonian systems of two degrees of freedom
in the region of interest often can be reduced to a two-dimensional
area-preserving twist map. To construct instability regions of these
maps physicists often use {\it a resonance overlap criterion} (see e.g.
\cite{Sagdeev:1988}, ch. 5, sect. 2). This criterion for nearby rational
numbers $p/q$ and $p'/q'$ compares ``sizes'' of averaged potentials.
If the sum of square roots of maxima of those potentials exceeds
$|p/q-p'/q'|$, this is a strong indication that the corresponding
periodic orbits can be connected. Doing this for an interval of
rational numbers gives an approximation for a so-called Birkhoff
Region of Instability (BRI).

If non-integrability is small, then most of the space is laminated
by KAM invariant curves. In order to find channels outside of KAM
curves, one considers a neighborhood of a resonance and computes
size of a so-called {\it stochastic layer}. Heuristic formulas can
be found e.g. in \cite{Chirikov:1979}, ch. 6.2 or in \cite{Sagdeev:1988}, ch. 5, sect. 3.
Treschev \cite{Treschev:2010} estimated width of stochastic layer
in a fairly general setting.

It turns out that Arnol'd's example and the elliptic problem near mean motion resonances
can be viewed as a perturbation of a product of two area-preserving
twist maps. In loose terms, for the first map we study orbits
located near a resonance inside the corresponding stochastic layer.
Width of a stochastic layer gives an approximation for time $\mathcal T$
it takes for many orbits to go around the layer. Stochastic behavior for the other twist map occurs because it takes
place ``over'' stochastic layer with random behavior.  This randomness
gives rise to ``random compositions'' of  twist maps. Numerical
experiments show behavior similar to a diffusion process (see e.g.
~\cite[Figure 6.3]{Lichtenberg:2010} or \cite{Gu07}). Its diffusion coefficient is proportional to square of the properly 
averaged perturbation divided by $\mathcal T$ (see e.g. \cite{Chirikov:1979}, 
ch. 7.2, \cite{Sagdeev:1988}, ch.5, sect. 7).

However, mathematically such randomness is a dark realm since there are many phenomena competing with the diffusive behavior. 
For example, for twist maps here are a few serious obstacles:

\begin{itemize}
\item inside of a BRI there are elliptic islands, where orbits are 
confined and do not diffuse (see e.g. \cite{Chirikov:1979}, ch. 5.5 
for a heuristic discussion of their size);

\item even if elliptic islands occupy not a dominant part of the phase space,
there exist  the so-called, in mathematical literature, {\it Aubry-Mather sets}.
In physics literature they are called {\it Cantor-Tori}. 
Orbits can stick to these sets for long periods of time
(see e.g. \cite{Sagdeev:1988}, ch.5, sect. 7);

\item similarly to sticking to Aubry-Mather sets orbits can stick 
to elliptic islands. 
\end{itemize}

For systems of two and a half degrees of freedom near a resonance 
 is also quite complicated.  We turn to our attention to two basic examples: Arnol'd's example and 
the elliptic problem, both near a resonance.  

In terms of a perturbation parameter $\eps$ of a nearly integrable system 
one would like to answer quantitatively the following natural questions. 
Fix a resonant segment $\Gamma$ and consider   $\Gamma_\eps$ a $\sqrt \eps$-neighborhood of this resonant segment.

\begin{itemize}
\item What is the natural time scale of diffusion?
One would expect that there is an $\eps$-dependent time scale $T_\eps$ in which one orbit
diffuses by $\OO(1)$ in action space and there is another time scale
$T^*_\eps$ in which many orbits, in the measure sense, diffuse by  $\OO(1)$.
\item Is there a natural time scale $T^*_\eps$ so that positive
fraction of orbits in a $\eps$-dependent region in $\Gamma_\eps$ diffuses
by $\OO(1)$?

\item At junctions of two resonances which fraction of orbits
chooses one resonance over the other?
\end{itemize}

Call $T^*_\eps$ {\it time scale of diffusion}. It seems a sophisticated
question to distinguish orbits starting in $\Gamma_\eps$ and staying inside
such a neighborhood in time scale of diffusion from those getting stuck near
KAM tori located $C \sqrt \eps$-away \footnote{This is so-called stickiness
phenomenon (see e.g.
\cite{MorbidelliG95,PerryW94}) } from $\Gamma$ with $C$ large. In this paper
we consider only the {\it a priori unstable case}, proposed by Arnol'd \cite{Arnold64}.
In this case, away from small velocities, there is only one dominant resonance
and making precise conjectures is simpler. This case will also motivate
conjectures for certain a priori chaotic systems.

\subsection{Speed of diffusion for a priori unstable systems and
  Positive measure}

Consider the following nearly integrable  Hamiltonian  system
proposed by Arnol'd \cite{Arnold64}:
\begin{equation}\label{Arnold-example}
H_\eps(p,q,I,\phi,t)=\frac 12 p^2 + \cos q -1 +
\frac 12 I^2 + \eps H_1(p,q,I,\phi,t), \quad
\text{ where } p, I \in \RR,\ \phi, q, t \in \TT,
\end{equation}
for an analytic perturbation $\eps H_1$. This system is usually called {\it a
  priori unstable}. Proving {\it Arnol'd diffusion } for this system
consists in showing that, for all small $\eps>0$ and a generic $\eps
H_1$, there exists orbits with
\[
|I(t)-I(0)|>\OO(1),
\]
where $\OO(1)$ is independent of $\eps$. There has been a fascinating progress in this problem achieved by
several groups (see \cite{Bernard08, ChengY04, DelshamsLS06a,
DelshamsH09, Treschev04}).  Treschev \cite{Treschev04} not only
proved existence of Arnol'd diffusion, but also gave an optimal
estimate on speed, namely, he constructed orbits
\[
|I(t)-I(0)|> c \dfrac{\eps}{|\ln \eps|}\ t
\]
for some $c>0$. One can see that this estimate is optimal,
i.e. $|I(t)-I(0)|<C \dfrac{\eps}{|\ln \eps|}\ t$ for some
$C>c$.

Heuristically the mechanism of diffusion is the following.  For small
$\eps>0$ the Hamiltonian $H_\eps$ has a $3$-dimensional normally
hyperbolic invariant cylinder $\Lambda_\eps$ close to $\Lambda_0=\{p=q=0\}$.  
A hypothetical diffusing orbit starts close to $\La_\eps$ and makes 
a homoclinic excursion. Each homoclinic excursion takes approximately 
$\OO(|\ln \eps|)$-time.  Increment of $I(t)$ after such an excursion 
is $\OO(\eps)$.\footnote{This is only an heuristic description 
as dynamics inside of the cylinder should come into play. Near 
so-called double resonance dynamics is different from the one near 
single resonances} If one can arrange that all excursions lead 
to increments of $I(t)$ of the same sign, the result follows.

It seems natural that orbits will be trapped inside
the resonance $p=0$ for polynomially long time.
Using this heuristic description one can conjecture
that increments $I(t)$ can behave as a random walk
for positive conditional measure for polynomially
large time.\\

\noindent{\bf Positive measure conjecture}
{\it Consider the  Hamiltonian $H_\eps$ with a generic perturbation $\eps H_1$.
Pick an  $\eps$-ball $B_\eps$ of initial conditions, whose center projects
  into $(p,q)=0$, and denote the Lebesgue probability measure
  supported on it by $Leb_\eps$.  Then, for some constants $c,C>0$
  independent of $\eps$, the set of initial conditions satisfying
  \[
  |I(T)-I(0)|>1 \qquad \text{ for some }\qquad 0<T<
  C \dfrac{|\ln \eps|}{\eps^2}
  \]
  is denoted $\text{Diff}$ and has measure
  $Leb_{\sqrt \eps}\,(\,\text{Diff}\,)>c$.\\
}

Since a typical excursion takes $\OO(|\ln \eps|)$-time and each
increment is $\OO(\eps)$, we essentially conjecture that after
$\OO(\eps^{-2})$ excursions with uniformly positive probability there
will be drift of order $\OO(\eps) \OO(\eps^{-1})=\OO(1)$.

\subsection{Structure of the  restricted planar elliptic three-body
  problem} \label{diffusion-structure}

In this appendix we relate a priori unstable systems and the 
restricted planar elliptic three-body problem.  Recall that we managed to write
the Hamiltonian of the latter problem in the form
\[
\begin{split}
H_{\text{ell}}(L,\ell,G,g,t)=&\dps H_{\text{circ}}(L,\ell,G,g,\mu)+
\mu e_0 \Delta H_{\text{ell}}(L,\ell,G,g,t,\mu,e_0)\\
=&\dps H^*_0(L,G)+
\mu \Delta H_{\text{circ}} (L,\ell,G,g,\mu)+
\mu e_0\Delta H_{\text{ell}}(L,\ell,G,g,t,\mu,e_0)\\
=& \dps 
-\frac{1}{2L^2}-G+\mu\Delta  H_{\text{circ}} (L,\ell,G,g,\mu)
+ \mu e_0\Delta H_{\text{ell}} (L,\ell,G,g,t,\mu,e_0).
\end{split}
\]
We have that
\begin{itemize}
\item $H^*_0$ is an integrable Hamiltonian.
\item $H_{\text{circ}}$ is non-integrable and for $H_{\text{circ}}$
in a certain interval of energy levels $[J_-,J_+]$ there is a family of hyperbolic
periodic orbits $\{p_J\}$ whose invariant manifolds intersect
transversally along at least one homoclinic.
\item  $H_{\text{ell}}$ is a $\OO(\mu e_0)$-perturbation of
$H_{\text{circ}}$ such that  certain Melnikov integral evaluated along
a transverse homoclinic of $H_{\text{circ}}$ is non-degenerate in two different
ways: dependence on time is non-trivial and relation between
inner and outer integrals is non-degenerate (see (\ref{def:no-common-curves})).
\end{itemize}

Having all these non-degeneracy conditions we prove the
existence of diffusing orbits. It is not difficult to prove, using averaging techniques, that for
$\mu>0$ small, there is a family of saddle periodic orbit
$\{\gamma_J\}$ on some interval $[J_-,J_+]$, whose hyperbolicity is $\sim\sqrt{\mu}$. It seems, however, to be
a non-trivial problem to establish the splitting of its separatrices.
Due to reversibility (\ref{involutionCartesian}) there are at least
four homoclinic intersections (two for upper separatrices and two for
lower ones). Having these two conditions it is natural to expect that
at least one of the four associated Melnikov integrals is
non-degenerate.  Qualitative analysis shows that it should be possible
to have a homoclinic excursion $\OO(\mu e_0)$-close to the invariant
cylinder.  Such an excursion takes $\OO(|\ln (\mu e_0)|)/\sqrt{\mu}$-time. If the
excursion is selected properly, then the result of the excursion is
that the increment of the eccentricity is $\OO(\mu e_0)$. This makes
us believe that the instability time obeys $T \sim -\dfrac{ \ln (\mu
  e_0)}{\mu^{3/2} e_0}$ \ \ stated in (\ref{instability-time}).

Let us point out that we believe that our diffusion
mechanism survives even for non-infinitesimal $e_0$'s,
e.g. realistic $e_0=0.048$. To justify this, we review
the above structure.

Notice that we use a $3$-dimensional normally hyperbolic invariant
cylinder and the intersection of its invariant manifolds  to diffuse.
The cylinder arises from the family of periodic orbits
$\{\gamma_J\}_{J\in [J_-,J_+]}$ of the circular problem, which persist
under the elliptic time-periodic
perturbation $\mu e_0\Delta H_\eell(L,\ell,G, g,t,\mu,e_0)$
(see (\ref{def:HamDelaunayNonRot}) and the derivation in
the corresponding section). As  the analysis carried out in
Section \ref{sec:Expansion:Hamiltonian} shows, in
the neighborhood of this family $\{\gamma_J\}_{J\in [J_-,J_+]}$,
the perturbation $\mu e_0 \Delta H_\eell(L,\ell,G, \g-t,t,\mu,e_0)$
can be averaged out to $\OO(\mu e_0^6)$.
Thus, invariant cylinders could persist even for not very small
$e_0$'s. However, estimating remainders analytically after several
steps of averaging is nearly impossible. Numerically though it might
be feasible.

Once the existence of an invariant cylinder is established, we need to
justify the existence of transverse intersections of its manifolds.
As before, analytically it is an insurmountable task, but numerically
it seems to be an achievable goal.

If these two steps are done, then one could try to compute numerically
inner and outer maps and show that they do not have common invariant
curves. This is again a difficult, but numerically realistic task (see
\cite{DelshamsMR08} for the computation of the outer map in another
problem in Celestial Mechanics).

On the other hand, the above asymptotics probably does not hold in the
neighborhood of circular motions of the massless body, which might be
much more stable than more eccentric motions. Yet many other factors
might influence the local stability or instability of various objects
(see section \ref{sec:capture}).

\subsection{The Mather accelerating problem and its speed of diffusion}
The structure we use to build diffusion is similar to the Mather
acceleration problem. Let us recall this problem and state
an interesting result of Piftankin \cite{Piftankin06} on speed
of diffusion.

Consider a Hamiltonian system
\[
H(q,p,t)=K(q,p)+V(q,t), \qquad q \in \mathbb T^2, \qquad
p\in \mathbb R^2,\qquad t\in \mathbb T,
\]
where $K(q,p)=\frac 12 \langle A^{-1}(q)p,p \rangle$ ---
kinetic energy corresponding to a riemannian metric
$K(q,p) = \frac 12 \langle A^{-1}(q) p,p \rangle, \ p = A(q) \,\dot q, \ \dot q \in
T_q \mathbb T^2$ and $V(q,t)$
is a time-periodic potential energy. Since the system
is not autonomous energy is not conserved.

\begin{description}
 \item[H1] Suppose the geodesic flow associated to $K$ has a hyperbolic
periodic orbit $\Gamma$ and transversal intersection
of its invariant manifolds, which contains a homoclinic
orbit $\gamma(t),\ t\in \mathbb R$.
\item [H2] The Melnikov integral is not constant. More exactly,
define a function
\[
\mathcal{L}(t)=\lim_{T \to +\infty}
\int_{-T}^T V(\gamma(t),t)\,dt  -
\int_{-T+t^u}^{T+t^s} V(\gamma(t),t)\, dt.
\]
The limit turns out to
exist and is independent of a choice of $t^s,\ t^u$.
This function is assumed to be non-constant.
\end{description}

Mather and his followers \cite{Mather96,BolotinT99,
DelshamsLS00,Gelfreich:2008,Kaloshin03,Piftankin06}
proved existence of an orbit $(q_\tau(t),p_\tau(t)),
\ t\in \mathbb R$ of unbounded energy. De la Llave
\cite{delaLlave04}, Piftankin
\cite{Piftankin06}, and Gelfreich-Turaev \cite{Gelfreich:2008}
proved that such an orbit can be chosen to have
linear growth of energy
\[
H(q_\tau(t),p_\tau(t))\ge At + B \qquad \text{ for all }
\qquad t\ge 0
\]
for some $A>0$ and $B\in \mathbb R$.

Notice that for large energies $H \sim \eps^{-2}$ the conformal change of
coordinates
\[
\hat p = \frac p \eps, \qquad
H= \eps^{-2} \hat H, \qquad t=\eps \hat t
\]
leads to the new Hamiltonian
\[
\hat H(q,p,t)=K(q,p)+\eps^2 V(q,\eps \hat t).
\]
It was shown in \cite{delaLlave04, Piftankin06, Gelfreich:2008} that
there are orbits diffusing linearly in the size of the perturbation.
In order to see these orbits, notice that $K(q,p)$ has a horseshoe.
Then, $\hat H$ can be considered as a time-periodic perturbation over
such a horseshoe.  It is shown by different methods in
\cite{delaLlave04, Piftankin06, Gelfreich:2008} that for a generic
time-periodic perturbation of a horseshoe there are linearly diffusing
orbits.

\subsection{Modified positive measure conjecture}

For systems with the properties discussed above
we can modify the positive measure conjecture as follows:

\noindent{\bf Positive measure conjecture for Mather type systems}\ {\it
  Consider the Hamiltonian
\[
H_{\mu,\eps}(L,\ell,G,g,t)=H^*_0(L,G)+\mu \Delta H_0(L,\ell,G,g,\mu)+
\mu e_0 \Delta H_1(L,\ell,G,g,t,\mu,e_0)
\]
such that
\begin{itemize}
\item for some interval $[J_-,J_+]$ the Hamiltonian
$H^*_0+\mu \Delta H_0$ has a family of saddle periodic
orbits $\{p_J\}_{J\in [J_-,J_+]}$,

\item for each $J\in [J_-J_+]$ there is at least one transverse
intersection of its invariant manifolds,

\item A Melnikov integral evaluated along a transverse homoclinic and
  inner dynamics are non-degenerate: the dependence of the Melnikov
  integral on time is non-trivial and the relation between inner and
  outer maps is non-degenerate (see (\ref{def:no-common-curves})).
\end{itemize}

Pick a $\mu e_0$-ball of initial conditions $B_{\mu e_0}$ whose action
components are centered at a resonance between $\ell$ and $g$.
Denote the Lebesgue probability measure supported on the ball
$B_{\mu e_0}$ by $Leb$.
Then for some constants $c,C>0$ independent of $\mu$ and $e_0$,
the set of initial conditions satisfying
\[
|G(T)-G(0)|>1 \qquad \text{ for some }\qquad 0<T<
C \dfrac{|\ln \mu e_0|}{\mu^{5/2} e_0^2}
\]
is denoted $\text{Diff}$ and has measure
$Leb\,(\,\text{Diff}\,)>c$.\\
}

Here is an important difference between the system $H_{\mu,e_0}$ and
an priori unstable one $H_\eps$, given by (\ref{Arnold-example}): the Hamiltonian $H^*_0+\mu \Delta H_0$ already has ``chaos'' and a
family of horseshoes on each energy surface with $J\in [J_-J_+]$,
while $H_0=H_\eps - \eps H_1$ is integrable. As we pointed out above,
for a generic time-periodic perturbation over a horseshoe there are
orbits diffusing linearly fast \cite{delaLlave04, Piftankin06,
  Gelfreich:2008}.  Yet we are interested in a set of conditional
positive measure.

In order to see the time of diffusion on an heuristic level, notice
that $H^*_0+\mu \Delta H_0$ has a family of saddle periodic orbits
$\{p_J\}_{J\in [J_-,J_+]}$ whose exponents are $\sim \sqrt \mu$.
Thus, one homoclinic excursion passing $\mu e_0$-close to separatrices
takes $|\ln \mu e_0|/\sqrt \mu$-time. Each excursion might lead to
increment of $G$ of size $\sim \mu e_0$.  Conjecturing that random
walk approximation holds true to have $\OO(1)$-changes in $G$, we need
$\OO(\mu^{-2} e_0^{-2})$ excursions.
\section*{Acknowledgements}

The authors acknowledge useful discussions with Abed Bounemoura, Marc
Chaperon, Alain Chenciner, Anatole Katok, {\`A}ngel Jorba, Mark Levi,
John Mather, Gennadi Piftankin, Philippe Robutel and Ke Zhang. P. R. 
acknowledges the assistance of {\`A}.~Jorba with the ``taylor'' package 
\newline 
(see \url{http://www.maia.ub.es/~angel/taylor}).

The authors warmly thank the Observatoire de Paris, the University of
Maryland at College Park, the Pennsylvania State University, the
Universitat Politecnica de Catalunya and the Fields Institute for
their hospitality, stimulating atmosphere, and support.


J. F. has been partially supported by the French ANR (Projets
ANR-12-BS01-0020 WKBHJ and ANR-10-BLAN 0102 DynPDE), M. G. and P.  R
by the Spanish MCyT/FEDER grant MTM2009-06973 and the Catalan SGR
grant 2009SGR859, and V. K. by NSF grant DMS-0701271.

\bibliography{references}
\bibliographystyle{alpha}
\end{document}